\documentclass[12pt]{amsart}
\usepackage{amscd}     
\usepackage{amssymb}
\usepackage{amsmath, amsthm, graphics, dsfont}
\usepackage[all]{xy}
\xyoption{2cell} \UseAllTwocells \xyoption{frame} \CompileMatrices
\allowdisplaybreaks[3]
\usepackage{amsfonts}
\usepackage[normalem]{ulem}
\usepackage{hyperref}
\usepackage{tikz}
\usepackage{mathtools}
\usepackage{verbatim} 
\usepackage{MnSymbol}
\usepackage{latexsym}
\usepackage{epsfig}
\usepackage{enumerate}
\usepackage{times}
\usepackage{xcolor}
\usepackage{geometry}

\newgeometry{vmargin={25mm,25mm}, hmargin={25mm,25mm}}   % set the margins

\newtheorem{prop}{Proposition}[section]
\newtheorem{lem}[prop]{Lemma}

\newtheorem{cor}[prop]{Corollary}
\newtheorem{thm}[prop]{Theorem}
\newtheorem{rem}[prop]{Remark}

\newtheorem*{thm*}{Main Theorem}

\numberwithin{equation}{section}

\newcommand{\CnZn}{[\mathbb{C}^n/\mathbb{Z}_n]}

\DeclareMathOperator{\diag}{diag}
\DeclareMathOperator{\ad}{ad}

\begin{document}

\title[HIGHER GENUS GW THEORY OF $\CnZn$ I: HAE\MakeLowercase{s}]{Higher genus Gromov-Witten theory of $[\mathbb{C}^n/\mathbb{Z}_n]$ I: \\ Holomorphic Anomaly Equations}

\author[Genlik]{Deniz Genlik}
\address{Department of Mathematics\\ Ohio State University\\ 100 Math Tower, 231 West 18th Ave. \\ Columbus,  OH 43210\\ USA}
\email{genlik.1@osu.edu}

\author[Tseng]{Hsian-Hua Tseng}
\address{Department of Mathematics\\ Ohio State University\\ 100 Math Tower, 231 West 18th Ave. \\ Columbus,  OH 43210\\ USA}
\email{hhtseng@math.ohio-state.edu}

\begin{abstract}
We study the structure of higher genus Gromov-Witten theory of the quotient stack $[\mathbb{C}^n/\mathbb{Z}_n]$. We prove holomorphic anomaly equations for $[\mathbb{C}^n/\mathbb{Z}_n]$, generalizing previous results of Lho-Pandharipande \cite{lho-p2} for the case of $[\mathbb{C}^3/\mathbb{Z}_3]$ and ours \cite{gt} for the case $[\mathbb{C}^5/\mathbb{Z}_5]$ to arbitrary $n\geq{3}$.
\end{abstract}

\date{\today}

\maketitle

\tableofcontents

\setcounter{section}{-1}

\section{Introduction}\label{sec:intro}

For an integer $n\geq 2$, the cyclic group $\mathbb{Z}_n$ acts naturally on $\mathbb{C}^n$ by letting its generator $1\in\mathbb{Z}_n$ act via the $n\times n$ matrix $$\text{diag}(e^{\frac{2\pi\sqrt{-1}}{n}},..., e^{\frac{2\pi\sqrt{-1}}{n}}).$$
The quotient $[\mathbb{C}^n/\mathbb{Z}_n]$ is a smooth Deligne-Mumford stack. The diagonal action of the torus $\mathrm{T}=(\mathbb{C}^*)^n$ on $\mathbb{C}^n$ induces a $\mathrm{T}$-action on $[\mathbb{C}^n/\mathbb{Z}_n]$, making it a toric Deligne-Mumford stack.

This paper is concerned with $\mathrm{T}$-equivariant Gromov-Witten invariants of $[\mathbb{C}^n/\mathbb{Z}_n]$. By definition, these are the following integrals
\begin{equation}\label{eqn:GWinv}
   \left\langle \prod_{i=1}^m \gamma_i\psi_i^{k_i} \right\rangle_{g,m}^{\CnZn}:= \int_{\left[\overline{M}_{g, m}^{\mathrm{orb}}\left(\left[\mathbb{C}^{n} / \mathbb{Z}_{n}\right], 0\right)\right]^{vir}} \prod_{i=1}^{m} \mathrm{ev}_{i}^{*}\left(\gamma_{i}\right)\psi_i^{k_i}. 
\end{equation}
Here, $[\overline{M}_{g, m}^{\mathrm{orb}}\left(\left[\mathbb{C}^{n} / \mathbb{Z}_{n}\right], 0\right)]^{vir}$ is the ($\mathrm{T}$-equivariant) virtual fundamental class of the moduli space $\overline{M}_{g, m}^{\mathrm{orb}}\left(\left[\mathbb{C}^{n} / \mathbb{Z}_{n}\right], 0\right)$ of stable maps to $[\mathbb{C}^n/\mathbb{Z}_n]$. $\psi_i\in H^2(\overline{M}_{g, m}^{\mathrm{orb}}\left(\left[\mathbb{C}^{n} / \mathbb{Z}_{n}\right], 0\right), \mathbb{Q})$ are descendant classes. 
\begin{equation*}
  \mathrm{ev}_i: \overline{M}_{g, m}^{\mathrm{orb}}\left(\left[\mathbb{C}^{n} / \mathbb{Z}_{n}\right], 0\right)\to I[\mathbb{C}^n/\mathbb{Z}_n]  
\end{equation*}
are evaluation maps, which take values in the inertia stack $I[\mathbb{C}^n/\mathbb{Z}_n]$ of $[\mathbb{C}^n/\mathbb{Z}_n]$. $\gamma_i$ are classes in the $\mathrm{T}$-equivariant Chen-Ruan cohomology of $\CnZn$, $$\gamma_i\in H^*_\mathrm{T, Orb}([\mathbb{C}^n/\mathbb{Z}_n]):=H^*_\mathrm{T}(I\CnZn).$$

Let $$\lambda_0,...,\lambda_{n-1}\in H^*_\mathrm{T}(\text{pt})=H^*(B\mathrm{T})$$ be the first Chern classes of the tautological line bundles of $B\mathrm{T}=(B\mathbb{C}^*)^n$. Then (\ref{eqn:GWinv}) takes value in $\mathbb{Q}(\lambda_0,...,\lambda_{n-1})$.

Foundational treatments of orbifold Gromov-Witten theory can be found in many references, the original being \cite{agv}. The ($\mathrm{T}$-equivariant) Gromov-Witten theory of the non-compact target $\CnZn$ is by definition a {\em twisted} Gromov-Witten theory of the classifying stack $B\mathbb{Z}_n$. Foundational discussions on twisted Gromov-Witten theory of orbifolds can be found in \cite{ccit} and \cite{Tseng}.

The main results of this paper concern structures of Gromov-Witten invariants (\ref{eqn:GWinv}), formulated in terms of generating functions. The definition of inertia stacks implies that 
\begin{equation*}
    I\CnZn=\CnZn\cup\bigcup_{k=1}^{n-1} B\mathbb{Z}_n.
\end{equation*}
Let 
\begin{equation*}
    \phi_0=1\in H^0_\mathrm{T}(\CnZn), \phi_k=1\in H^0_\mathrm{T}(B\mathbb{Z}_n), 1\leq k\leq n-1.
\end{equation*}
Then, $\{\phi_0,...,\phi_{n-1}\}$ is an additive basis of $H^*_\mathrm{T,Orb}(\CnZn)$. The orbifold Poincar\'e dual $\{\phi^0,\ldots,\phi^{n-1}\}$ of this basis is given by
\begin{equation*}
\begin{split}
\phi^0&=n\lambda_0\cdots\lambda_{n-1}\phi_0,\\
\phi^1&=n\phi_{n-1},\\
&\,\,\,\vdots\\
\phi^{n-1}&=n\phi_1.
\end{split}
\end{equation*}
To simplify notation, in what follows, we set \begin{equation*}
\phi_i\coloneqq\phi_j\quad\text{if }j\equiv{i}\mod{n}\quad\text{and}\quad\phi^i\coloneqq\phi^j\quad\text{if }j\equiv{i}\mod{n},
\end{equation*}
for all $i\geq{0}$ and $0\leq{j}\leq{n-1}$.   

Associated to $\phi_{c_{1}}, \ldots,\phi_{c_{m}}\in H^{\star}_{\mathrm{T,Orb}}\left(\left[\mathbb{C}^n/\mathbb{Z}_n\right]\right)$ we define the Gromov-Witten potential by
\begin{equation*}
\mathcal{F}_{g, m}^{\left[\mathbb{C}^{n} / \mathbb{Z}_{n}\right]}\left(\phi_{c_{1}}, \ldots, \phi_{c_{m}}\right)=\sum_{d=0}^{\infty} \frac{\Theta^{d}}{d !} \int_{\left[\overline{M}_{g, m+d}^{\mathrm{orb}}\left(\left[\mathbb{C}^{n} / \mathbb{Z}_{n}\right], 0\right)\right]^{v i r}} \prod_{i=1}^{m} \mathrm{ev}_{i}^{*}\left(\phi_{c_{i}}\right) \prod_{i=m+1}^{m+d} \mathrm{ev}_{i}^{*}\left(\phi_{1}\right).    
\end{equation*}
We also use the following standard double-bracket notation, 
\begin{equation*}
\left\langle\left\langle\phi_{c_{1}}, \ldots, \phi_{c_{m}}\right\rangle\right\rangle_{g, m}^{\left[\mathbb{C}^{n} / \mathbb{Z}_{n}\right]}=\mathcal{F}_{g, m}^{\left[\mathbb{C}^{n} / \mathbb{Z}_{n}\right]}\left(\phi_{c_{1}}, \ldots, \phi_{c_{m}}\right).
\end{equation*}

We use the following involutions throughout the paper to present equations more efficiently:
$$\mathrm{Inv}:\{0,...,n-1\}\rightarrow\{0,...,n-1\}$$ with $\mathrm{Inv}(0)=0$ and $\mathrm{Inv}(i)=n-i$ for $1\leq i \leq n-1$, and 
$$\mathrm{Ion}:\{0,...,n\}\rightarrow\{0,...,n\}$$ with $\mathrm{Ion}(0)=n$, $\mathrm{Ion}(n)=0$, and $\mathrm{Ion}(i)=i$ for $1\leq i \leq n-1$.

Holomorphic anomaly equations are partial differential equations predicted by physicists as a part of the higher genus mirror symmetry conjecture for Calabi-Yau threefolds \cite{BCOV1,BCOV2}. Lho-Pandharipande provided mathematical proofs of holomorphic anomaly equations for $K\mathbb{P}^2$ \cite{lho-p} and formal quintic \cite{lho-p3} using stable quotient theory. Their equations exactly match with physics calculations for these Calabi-Yau threefolds given in \cite{asyz}. Motivated by  $K\mathbb{P}^2$, Lho-Pandharipande also proved holomorphic anomaly equations for $[\mathbb{C}^3/\mathbb{Z}_3]$ in \cite{lho-p2}. We generalize their work on holomorphic anomaly equations for $[\mathbb{C}^3/\mathbb{Z}_3]$ to $\CnZn$ for $n\geq{3}$.

The main results of this paper, summarized below, are differential equations for these generating functions $\mathcal{F}^{\CnZn}_{g}$ when $n\geq 3$ and $g\geq{2}$ after the following specializations of equivariant parameters: for $0\leq i \leq n-1$, 
\begin{equation}\label{eqn:specialization}
   \lambda_i=\begin{cases} 
      e^{\frac{2\pi\sqrt{-1}i}{n}}e^{\frac{\pi\sqrt{-1}}{n}}& \text{if $n$ is even,}\\
      e^{\frac{2\pi\sqrt{-1}i}{n}} & \text{if $n$ is odd}.
     \end{cases}
\end{equation}
Although physics prediction of holomorphic anomaly equations was meant for Calabi-Yau manifolds of dimension $3$, borrowing terminology from String Theory, we call the differential equations in our main results as {\em holomorphic anomaly equations} for $\CnZn$.
\begin{thm*}[Finite generation property and holomorphic anomaly equations]
\hfill
\begin{enumerate}
\item (=Corollary \ref{cor:VertexEdgeCont}) The Gromov-Witten potential lies in a certain polynomial ring:
\begin{equation*}
\mathcal{F}_{g, m}^{\left[\mathbb{C}^{n} / \mathbb{Z}_{n}\right]}\left(\phi_{c_{1}}, \ldots, \phi_{c_{m}}\right)\in\mathbb{C}[L^{\pm{1}}][\mathfrak{S}_n][\mathfrak{C}_n].
\end{equation*}

\item (=Theorem \ref{thm:HAE_n_odd})
Let $n\geq{3}$ be an odd number with $n=2s+1$, and $g\geq{2}$. We have
\begin{equation*}
\frac{C_{s+1}}{(2s+1)L}\frac{\partial}{\partial A_{s}}\mathcal{F}_{g}^{\left[\mathbb{C}^{n} / \mathbb{Z}_{n}\right]}
=\frac{1}{2}\mathcal{F}_{g-1,2}^{\left[\mathbb{C}^n/ \mathbb{Z}_n\right]}\left(\phi_s,\phi_s\right)+\frac{1}{2}\sum_{i=1}^{g-1}\mathcal{F}_{g-i,1}^{\left[\mathbb{C}^n/ \mathbb{Z}_n\right]}\left(\phi_s\right)\mathcal{F}_{i,1}^{\left[\mathbb{C}^n/ \mathbb{Z}_n\right]}\left(\phi_s\right).
\end{equation*}

\item (=Theorem \ref{thm:HAE_n_even})
Let $n\geq{4}$ be an even number with $n=2s$, and $g\geq{2}$. We have
\begin{equation*}
\frac{C_{s+1}}{2sL}\frac{\partial}{\partial A_{s-1}}\mathcal{F}_{g}^{\left[\mathbb{C}^{n} / \mathbb{Z}_{n}\right]}
=\mathcal{F}_{g-1,2}^{\left[\mathbb{C}^n/ \mathbb{Z}_n\right]}\left(\phi_{s-1},\phi_s\right)+\sum_{i=1}^{g-1}\mathcal{F}_{g-i,1}^{\left[\mathbb{C}^n/ \mathbb{Z}_n\right]}\left(\phi_{s-1}\right)\mathcal{F}_{i,1}^{\left[\mathbb{C}^n/ \mathbb{Z}_n\right]}\left(\phi_s\right).
\end{equation*}    
\end{enumerate}
\end{thm*}

We refer to Corollary \ref{cor:VertexEdgeCont}, Theorem \ref{thm:HAE_n_odd}, and Theorem \ref{thm:HAE_n_even} for more details. Theorem \ref{thm:HAE_n_odd} is a generalization of the differential equation obtained in \cite{lho-p2} for $[\mathbb{C}^3/\mathbb{Z}_3]$. 

The proofs of Theorems \ref{thm:HAE_n_odd} and \ref{thm:HAE_n_even} follow the approach taken in \cite{lho-p2} for the case $n=3$. The approach is based on the {\em cohomological field theory} (in the sense of \cite{km}) nature of Gromov-Witten theory of $\CnZn$ and relies heavily on the Givental-Teleman classification \cite{g3}, \cite{t}, of semisimple cohomological field theories. A survey of Givental-Teleman classification can be found in \cite{Picm}.

More precisely, the proofs of Theorems \ref{thm:HAE_n_odd} and \ref{thm:HAE_n_even} use a formula obtained from Givental-Teleman classification which expresses the potential $\mathcal{F}_{g, m}^{\left[\mathbb{C}^{n} / \mathbb{Z}_{n}\right]}$ as a sum over graphs whose summands only require {\em genus $0$} Gromov-Witten theory of $\CnZn$, see equation (\ref{eqn:formula_Fg}) and Theorem \ref{prop:contributions} for details. This approach thus requires a detailed study of genus $0$ Gromov-Witten theory of $\CnZn$, which is worked out in Section \ref{sec:genus_0}. Many specific power series arise in the analysis of the genus $0$ theory. Properties of these power series and the rings containing them are studied in details in Section \ref{sec:diff_ring}. Holomorphic anomaly equations, Theorems \ref{thm:HAE_n_odd} and \ref{thm:HAE_n_even}, are described and proved in Section \ref{sec:HAE}. In Section \ref{sec:HAE_insertion}, following a question by the referee, we apply the same approach to obtain holomorphic anomaly equations for Gromov-Witten potentials with insertions. Appendix \ref{appendix:stirling} contains discussions on properties of Stirling numbers used in this paper. Appendix \ref{appendix:I-function} contains a detailed analysis of the $I$-functions of $\CnZn$.

Some previous studies related to holomorphic anomaly equations in dimension $>3$ can be found in \cite{lho}, \cite{lho3}, \cite{Ob}. In \cite{gt}, we use results in \cite{lho} to derive two holomorphic anomaly equations for $[\mathbb{C}^5/\mathbb{Z}_5]$. One of them is the $n=5$ case of Theorem \ref{thm:HAE_n_odd}, and the other is new.

Studying higher genus Gromov-Witten theory of $K\mathbb{P}^{n-1}$ in detail and comparing its cohomological field theory structure to that of $\CnZn$ described in this paper, we obtain a crepant resolution correspondence for $K\mathbb{P}^{n-1}$ and $\CnZn$ in all genera \cite{gt2}.

\subsection{Acknowledgment}
We are very grateful to the referee for useful comments and suggestions. We thank C.-C. Liu and G. Oberdieck for their interests in this work. D. G. would like to thank Aniket Shah for a discussion that turned out to be useful for proof of Lemma \ref{lem:DkImCommutator}.
D. G. is supported in part by a Special Graduate Assignment fellowship by the OSU Department of Mathematics, and H.-H. T. is supported in part by a Simons Foundation collaboration grant.

\section{Genus zero theory}\label{sec:genus_0}
\subsection{Mirror theorem}

Applying the methods of \cite{ccit}, we obtain\footnote{Here, $\langle \alpha \rangle$ is the fractional part of $\alpha$.} the twisted $I$-function for $[\mathbb{C}^n/\mathbb{Z}_n]$, 
\begin{equation*}
I^{\text{tw}}(\mathbf{x},z)=z\sum_{k_0,...,k_{n-1}\geq{0}}\frac{\prod_{\substack{b:0\leq b<\alpha(\Vec{k}) \\ \langle b \rangle=\langle\alpha(\Vec{k})\rangle}}\prod_{i=0}^{n-1}\left(\lambda_i-bz\right)}{z^{k_0+...+k_{n-1}}}\frac{x_0^{k_0}\cdots x_{n-1}^{k_{n-1}}}{{k_0}!\cdots k_{n-1}!}\phi_{n\alpha(\Vec{k})}
\end{equation*}
where
\begin{equation*}
\mathbf{x}=\sum_{i=0}^{n-1}x_i\phi_i\quad\text{and}\quad
\alpha(\Vec{k})=\sum_{i=0}^{n-1}\frac{ik_i}{n}\quad\text{with}\quad\Vec{k}=(k_0,...,k_{n-1})\in\mathbb{Z}^n.
\end{equation*}
The $J$-function of $[\mathbb{C}^n/\mathbb{Z}_n]$ is characterized by
\begin{equation*}
J^{\text{tw}}(\tau,-z)=-z+\tau+O(z^{-1}).
\end{equation*}
To get a mirror theorem, we need to find the appropriate locus to restrict twisted $I$-function $I^{\text{tw}}(x,z)$. For that, we need the following lemma.

\begin{lem}
For every integer $n\geq 3$ and for every integer $l\in\{2,...,n-1\}$ there exists an integer $k$ such that $n-k\geq 1$ and $1<\frac{kl}{n}<2$.
\end{lem}

\begin{proof}
For $l=2$, let $k=n-1$. For $3\leq l\leq \frac{n}{2}$, let $k=\lfloor\frac{n}{l}\rfloor$. For $\frac{n}{2}<l\leq n-1$, let $k=2$.
\end{proof}

By the existence of such $k$, we see that if we let all $k_i$ to be 0 except when $i=l$ and if we let $k_l=k$ then the coefficient of the monomial $x_l^{k_l}$ has a term of $z$-degree greater than equal to $2$. So, we should restrict the twisted $I$-function $I^{\text{tw}}(x,z)$ to the locus $x_2=...=x_{n-1}=0$ to be able to look for a mirror theorem by the characterization of the $J$-function.

Applying \cite[Theorem 4.8]{ccit}, we obtain the following generalization of the mirror theorem for $[\mathbb{C}^3/\mathbb{Z}_3]$ in \cite[Section 6.3]{ccit}.

\begin{thm}\label{thm:mirrorthm}
For $n\geq 3$, the twisted $I$-function and the $J$-function of  $[\mathbb{C}^n/\mathbb{Z}_n]$ satisfies the following equality
\begin{equation*}
I^{\text{tw}}\left(x_0\phi_0+x_1\phi_1, z\right)=J^{\text{tw}}\left(\tau^0\phi_0+\tau^1\phi_1, z\right)
\end{equation*}
with
\begin{equation*}
\tau^0=x_0\quad\text{and}\quad
\tau^1=\sum_{k\geq 0}\frac{(-1)^{nk}x_1^{nk+1}}{(nk+1)!}\left(\frac{\Gamma\left(k+\frac{1}{n}\right)}{\Gamma\left(\frac{1}{n}\right)}\right)^n.
\end{equation*}
\end{thm}
\begin{proof}
We decompose $I^{\text{tw}}\left(x_0\phi_0+x_1\phi_1, z\right)$ as
\begin{equation}\label{splittwisted}
I^{\text{tw}}\left(x_0\phi_0+x_1\phi_1, z\right)=z\phi_0+\sum_{k_0\geq{1}}\frac{1}{z^{k_0-1}}\frac{x_0^{k_0}}{{k_0}!}\phi_{0}
+\sum_{k_0\geq{0},k_1\geq{1}}\frac{\gamma_{k_1}(z)}{z^{k_0+k_{1}-1}}\frac{x_0^{k_0}x_{1}^{k_{1}}}{{k_0}!k_{1}!}\phi_{k_1}
\end{equation}
where 
\begin{equation*}
\gamma_{k_1}(z)=\prod_{\substack{b:0\leq b<\frac{k_1}{n} \\ \langle b \rangle=\langle\frac{k_1}{n} \rangle}}\prod_{i=0}^{n-1}\left(\lambda_i-bz\right)\quad\text{for}\quad k_1\geq 1.
\end{equation*}
By induction on $k_1$, we can show that
$\gamma_{k_1}(z)$ is a polynomial of degree\footnote{Here, $\lceil - \rceil$ is the ceiling function.} $m_{k_1}=\left(\left\lceil{\frac{k_1}{n}}\right\rceil-1\right)n$ with the leading coefficient
\begin{equation*}
l_{k_1}=
\begin{cases}
{\displaystyle\prod_{i=1}^{\left\lceil{\frac{k_1}{n}}\right\rceil-1}}\left(i-\frac{k_1}{n}\right)^n&\text{if}\quad n\nmid k_1,\\
(-1)^{n}{\displaystyle\prod_{i=1}^{\left\lceil{\frac{k_1}{n}}\right\rceil-1}}(-1)^{n}i^n&\text{if}\quad n \mid k_1.
\end{cases}
\end{equation*}
When $k_0\geq{0}$, observe that
\begin{equation*}
\deg\left(\frac{\gamma_{k_1}(z)}{z^{k_0+k_{1}-1}}\right)=\left(\left\lceil{\frac{k_1}{n}}\right\rceil-1\right)n+1-k_1-k_0\leq{0}.
\end{equation*}
Hence, by equation (\ref{splittwisted}) we see that the twisted $I$-function $I^{\text{tw}}\left(x_0\phi_0+x_1\phi_1, z\right)$ is of the form \begin{equation*}
I^{\text{tw}}(x_0\phi_0+x_1\phi_1,z)=z+\tau(x_0\phi_0+x_1\phi_1)+O(z^{-1}).  
\end{equation*}
To write $\tau(x_0\phi_0+x_1\phi_1)$ explicitly, we need to find the summands of equation (\ref{splittwisted}) which are constant in $z$. Clearly, the only contribution is $x_0\phi_0$ from the first sum. Let $$\gamma_{k_1}(z)=\sum_{j=0}^{m_{k_1}}\gamma_{k_1}^j(z)$$ where $\gamma_{k_1}^j(z)$ is the monomial of degree $j$ in $z$. Then, for the second sum, we need to find $(k_0,k_1)$ satisfying
\begin{equation*}
\deg\left(\frac{\gamma_{k_1}^j(z)}{z^{k_0+k_{1}-1}}\right)=j+1-(k_0+k_1)=0,
\end{equation*}
or equivalently, we need to find $(k_0,k_1)$ satisfying
\begin{equation*}
j=k_0+k_1-1.
\end{equation*}
Since $0\leq j \leq m_{k_1}$ and $k_1\geq 1$, the only such possibility is $(k_0,k_1)=(0,nk+1)$ for some $k\geq 0$. In this case, we have $j=nk=m_{k_1}=\deg{\gamma_{k_1}(z)}$ and the leading coefficient of ${\gamma_{k_1}(z)}$ is
\begin{equation*}
l_{k_1}=
{\displaystyle\prod_{i=1}^{k}}\left(i-\frac{nk+1}{n}\right)^n=(-1)^{nk}\left({\displaystyle\prod_{i=1}^{k}}\left(k-i+\frac{1}{n}\right)\right)^n=(-1)^{nk}\left(\frac{\Gamma\left(k+\frac{1}{n}\right)}{\Gamma\left(\frac{1}{n}\right)}\right)^n.
\end{equation*}
So, we see that
\begin{equation*}
\tau(x_0\phi_0+x_1\phi_1)=\tau^0\phi_0+\tau^1\phi_1
\end{equation*}
with
\begin{equation*}
\tau^0=x_0\quad\text{and}\quad
\tau^1=\sum_{k\geq 0}\frac{(-1)^{nk}x_1^{nk+1}}{(nk+1)!}\left(\frac{\Gamma\left(k+\frac{1}{n}\right)}{\Gamma\left(\frac{1}{n}\right)}\right)^n.
\end{equation*}
\end{proof}

In what follows, we impose the specializations (\ref{eqn:specialization}). Then we have
\begin{equation*}
\prod_{i=0}^{n-1}\left(\lambda_i-bz\right)=
\begin{cases} 
    1+(bz)^n & \text{if $n$ is even,}\\
    1-(bz)^n & \text{if $n$ is odd} 
\end{cases}
=1+(-1)^n(bz)^n.
\end{equation*}
Using the twisted $I$-function $I^{\text{tw}}$, the above specializations, and the convention of \cite{{lho-p}}, we define the $I$-function for  $[\mathbb{C}^n/\mathbb{Z}_n]$ :
\begin{equation}\label{def:I-function}
I\left(x, z\right)=
        \sum_{k=0}^{\infty}\frac{x^k}{{z^k}k!}\prod_{\substack{b:0\leq b<\frac{k}{n} \\ \langle b \rangle=\langle\frac{k}{n}\rangle}}\left(1+(-1)^n(bz)^n\right)\phi_k.
\end{equation}
It is easy to see that $I$-function (\ref{def:I-function}) of $[\mathbb{C}^n/\mathbb{Z}_n]$ is of the form
\begin{equation}\label{eq:IfuncAsSumIk}
I\left(x, z\right)=\sum_{k=0}^{\infty}\frac{I_k(x)}{z^k}\phi_k.
\end{equation}
For $0\leq i \leq n-1$, define
\begin{equation}\label{eq:tilde_I_i_functions}
\widetilde{I}_i(x,z)=\sum_{l=0}^{\infty}\frac{I_{nl+i}(x)}{z^{nl+i}}.
\end{equation}
Then, by equation (\ref{eq:IfuncAsSumIk}) we see that the $I$-function (\ref{def:I-function}) can be written as
\begin{equation}\label{eq:I_func_as_sum_of_tilde_I_i}
 I(x,z)=\sum_{i=0}^{n-1}\widetilde{I}_i(x,z)\phi_i=\widetilde{I}_0(x,z)\phi_0+\ldots+\widetilde{I}_{n-1}(x,z)\phi_{n-1}.  
\end{equation}
By keeping track of the degrees, we see that 
\begin{equation*}
I_k(x)=\sum_{l=0}^{\infty}\frac{(-1)^{nl}x^{nl+k}}{(nl+k)!}\left(\frac{\Gamma\left(l+\frac{k}{n}\right)}{\Gamma\left(\frac{k}{n}\right)}\right)^n
\end{equation*}
for $0\leq{k}\leq{n-1}.$

The small $J$-function for $[\mathbb{C}^n/\mathbb{Z}_n]$ is defined by 
\begin{equation*}
J\left(\Theta,z\right)=\phi_0+\frac{\Theta\phi_1}{z}+\sum_{i=0}^{n-1}\phi^i\left\langle\left\langle\frac{\phi_i}{z(z-\psi)}\right\rangle\right\rangle_{0,1}^{[\mathbb{C}^n/\mathbb{Z}_n]}.
\end{equation*}
Theorem \ref{thm:mirrorthm} implies
\begin{equation}\label{eq:smallmirrorthm}
J\left(\Theta(x),z\right)=I(x,z)
\end{equation}
with the mirror transformation
\begin{equation}\label{eq:mirrortransform}
\Theta(x)=I_1(x).
\end{equation}

\subsection{Picard-Fuchs equations}
Define the operator $$D:\mathbb{C}[\![x]\!]\rightarrow \mathbb{C}[\![x]\!]$$ and its inverse $$D^{-1}:x\mathbb{C}[\![x]\!]\rightarrow x\mathbb{C}[\![x]\!]$$ by
\begin{equation*}
Df(x)=x\frac{df(x)}{dx}, \quad D^{-1}f(x)=\int_{0}^x\frac{f(t)}{t}dt.
\end{equation*}
 We have the following identity
\begin{equation}\label{stirling1}
{x^m}\frac{d^m}{dx^m}=D(D-1)...(D-m+1)=\sum_{k=1}^ms_{m,k}{D^k}
\end{equation}
where $s_{m,k}$ are Stirling numbers of first kind. For a brief account of the properties of Stirling numbers, see Appendix \ref{appendix:stirling}.

\begin{prop}
The $I$-function of $[\mathbb{C}^n/\mathbb{Z}_n]$ satisfies the following Picard-Fuchs (type) equation
\begin{equation}\label{eq:PF1}
\frac{1}{x^n}D(D-1)...(D-n+1)I(x,z)-(-1)^n\left(\frac{1}{n}\right)^nD^nI(x,z)=\left(\frac{1}{z}\right)^nI(x,z).
\end{equation}
\end{prop}

\begin{proof}
Applying the operator $\frac{d^n}{dx^n}$ to the function $I(x,z)$ we obtain
\begin{align*}
\frac{d^n}{dx^n}I(x,z)=&\sum_{k=n}^{\infty}\frac{x^{k-n}}{{z^k}(k-n)!}\prod_{\substack{b:0\leq b<\frac{k}{n} \\ \langle b \rangle=\langle\frac{k}{n}\rangle}}\left(1+(-1)^n(bz)^n\right)\phi_k\\
=&\sum_{k=0}^{\infty}\frac{x^{k}}{{z^{k+n}}k!}\prod_{\substack{b:0\leq b<1+\frac{k}{n} \\ \langle b \rangle=\langle\frac{k}{n}\rangle}}\left(1+(-1)^n(bz)^n\right)\phi_k \quad\text{ by shifting index and }\phi_{k+n}=\phi_k\\
=&\sum_{k=0}^{\infty}\frac{x^{k}}{{z^{k+n}}k!}\prod_{\substack{b:\frac{k}{n}\leq b<1+\frac{k}{n} \\ \langle b \rangle=\langle\frac{k}{n}\rangle}}\left(1+(-1)^n(bz)^n\right)\prod_{\substack{b:0\leq b<\frac{k}{n} \\ \langle b \rangle=\langle\frac{k}{n}\rangle}}\left(1+(-1)^n(bz)^n\right)\phi_k\\
=&\left(\frac{1}{z}\right)^n\sum_{k=0}^{\infty}\frac{x^k}{{z^k}k!}\prod_{\substack{b:0\leq b<\frac{k}{n} \\ \langle b \rangle=\langle\frac{k}{n}\rangle}}\left(1+(-1)^n(bz)^n\right)\phi_k\\
&+(-1)^n\left(\frac{1}{n}\right)^n\sum_{k=0}^{\infty}\frac{k^nx^k}{{z^k}k!}\prod_{\substack{b:0\leq b<\frac{k}{n} \\ \langle b \rangle=\langle\frac{k}{n}\rangle}}\left(1+(-1)^n(bz)^n\right)\phi_k\\
=&\left(\frac{1}{z}\right)^nI(x,z)+(-1)^n\left(\frac{1}{n}\right)^nD^nI(x,z).
\end{align*}
Using equation (\ref{stirling1}), we complete the proof.
\end{proof}

By equation (\ref{stirling1}), we can rewrite the Picard-Fuchs equation (\ref{eq:PF1}) as
\begin{equation}\label{eqn:PF1.5}
\frac{1}{x^n}\sum_{k=1}^{n}s_{n,k}D^kI(x,z)-(-1)^n\left(\frac{1}{n}\right)^nD^nI(x,z)=\left(\frac{1}{z}\right)^nI(x,z).
\end{equation}
Since $s_{n,n}=1$, we can rewrite equation (\ref{eqn:PF1.5}) further as
\begin{equation}\label{eq:PF2}
x^{-n}\left(\left(1-(-1)^n\left(\frac{x}{n}\right)^n\right)D^nI(x,z)+\sum_{k=1}^{n-1}s_{n,k}D^kI(x,z)\right)=z^{-n}I(x,z).
\end{equation}
We define\footnote{When $n=3$, our $L$ differs from $L$ defined in \cite{lho-p2} by a sign.} the following series in $\mathbb{C}[\![x]\!]$:
\begin{equation}\label{Lseries}
    L(x)=x\left(1-(-1)^n\left(\frac{x}{n}\right)^n\right)^{-\frac{1}{n}}.
\end{equation}
By equation (\ref{eq:DLLLemma1}) below, we obtain the following alternative form of the Picard-Fuchs equation (\ref{eq:PF2}), which we frequently use:
\begin{equation}\label{eq:PF3}
L^{-n}\left(D^{n}I(x,z)+\frac{D L}{L} \sum_{k=1}^{n-1} s_{n,k} D^{k}I(x,z)\right)=z^{-n}I(x,z).
\end{equation}

Due to the particular form (\ref{eq:IfuncAsSumIk}) of $I$-function, in order to define some series avoiding $\phi_k$'s, we also introduce the function $E(x,z)$
\begin{equation}\label{def:E-function}
E\left(x, z\right)=\sum_{k=0}^{\infty}\frac{x^k}{{z^k}k!}\prod_{\substack{b:0\leq b<\frac{k}{n} \\ \langle b \rangle=\langle\frac{k}{n}\rangle}}\left(1+(-1)^n(bz)^n\right)=\sum_{k=0}^{\infty}\frac{I_k(x)}{z^k}
\end{equation}
just by removing the $\phi_k$ from the expression of the $I$-function (\ref{def:I-function}). Also, substituting equation (\ref{eq:IfuncAsSumIk}) into Picard-Fuchs equation (\ref{eq:PF3}) and analyzing the coefficients of both sides, we obtain
\begin{equation}\label{eq:PFforIks}
D^{n}I_k+\frac{D L}{L} \sum_{k=1}^{n-1} s_{n,k} D^{k}I_k=0
\end{equation}
for $0\leq{k}\leq{n-1}$.

\subsection{Birkhoff factorization}\label{subsection:Birkhoff_Factorizaton}
Next, we define\footnote{For a series $F(x,z)\in\mathbb{C}[\![x,\frac{1}{z}]\!]$, the constant term of $F(x,z)$ with respect to $\frac{1}{z}$ is denoted by $F(x,\infty)$.} the series $E_i(x,z)$ and $C_i(x)$ for $i\geq 0$:
\begin{equation}\label{def:EiCi}
E_i(x,z)=\mathds{M}^iE(x,z)\quad\text{and}\quad C_i(x)=E_i(x,\infty)
\end{equation}
where $\mathds{M}$ is the Birkhoff operator defined by
\begin{equation}\label{eqn:BF_operator}
    \mathds{M}F(x,z)=zD\frac{F(x,z)}{F(x,\infty)}
\end{equation}
for any $F(x,z)$ with non-zero $F(x,\infty)$.
We also define\footnote{It is easy to show that two definitions of $C_i$'s are equivalent.} the series $C_i(x)$ inductively as follows:
\begin{equation}\label{altdefCis}
C_0=I_0=1\quad\text{and}\quad C_{i}=D\mathfrak{L}_{i-1}...\mathfrak{L}_{0}I_i\quad\text{for}\quad i\geq 1
\end{equation}
where $$\mathfrak{L}_i=C_i^{-1}D$$ for $i\geq 1$ and $\mathfrak{L}_0$ is the identity.
For any $l\geq 0$, we define the following series in $x$
\begin{equation*}
K_l=\prod_{i=0}^lC_i.
\end{equation*}
We have the following identities for the series $C_i$ and $K_l$, proved in Appendix \ref{appendix:I-function-Part1}, see Lemma \ref{propertiesofCfunctions} and Corollary \ref{Kfunctions}.
\begin{enumerate}
    \item $C_{k+n}=C_k$ for all $k\geq 1$\label{Cfunctions1},
    \item $\prod_{k=1}^nC_k={L}^n$ \label{Cfunctions2},
    \item $C_{k}=C_{n+1-k}$ for all $1\leq k \leq n$ \label{Cfunctions4}.
    \item $K_{n+l}=L^nK_l$ for all $l\geq{0}$, in particular $K_n=L^n$, \label{Kfunctions1}
    \item $K_lK_{n-l}=L^n$ and $K_lK_{\mathrm{Inv}(l)}=L^{l+\mathrm{Inv}(l)}$ for all $0\leq l \leq n-1$.\label{Kfunctions2}
\end{enumerate}

Define the $\mathds{S}$-operator for $[\mathbb{C}^n/\mathbb{Z}_n]$ by
\begin{equation*}
\mathds{S}^{[\mathbb{C}^n/\mathbb{Z}_n]}\left(\Theta,z\right)\left(\gamma\right)=\gamma+\sum_{i=0}^{n-1}\phi^{i}\left\langle\left\langle\frac{\phi_i}{z-\psi},\gamma\right\rangle\right\rangle_{0,2}^{[\mathbb{C}^n/\mathbb{Z}_n]}
\end{equation*}
for $\gamma\in H^{\star}_{\mathrm{T,Orb}}\left(\left[\mathbb{C}^n/\mathbb{Z}_n\right]\right)$. The $\mathds{S}$-operator satisfies the following identities :
\begin{align}
\mathds{S}^{[\mathbb{C}^n/\mathbb{Z}_n]}\left(\Theta,z\right)\left(\phi_0\right)=&I(x,z),\label{PropertiesofSOperator1}\\
\mathds{S}^{[\mathbb{C}^n/\mathbb{Z}_n]}\left(\Theta,z\right)\left(\phi_i\right)=&z\mathfrak{L}_i\mathds{S}^{[\mathbb{C}^n/\mathbb{Z}_n]}\left(\Theta,z\right)\left(\phi_{i-1}\right)\quad\text{for}\quad i\geq 1.\label{PropertiesofSOperator2}
\end{align}
Equation (\ref{PropertiesofSOperator1}) is a direct consequence of the definitions of $J$-function, $\mathds{S}$-operator, and equation (\ref{eq:smallmirrorthm}). For equation (\ref{PropertiesofSOperator2}), we see that both sides are of the form $\phi_i+O(z^{-1})$, hence they must match by properties of the Lagrangian cone \cite{g3}.

\begin{lem}\label{lem:PFFactorization}
We have the following factorization of the operator acting on the left-hand side of the Picard-Fuchs equation (\ref{eq:PF3})
\begin{equation}
\mathfrak{L}_1\cdots\mathfrak{L}_n=\mathfrak{L}_n\cdots\mathfrak{L}_1=L^{-n}\left(D^{n}+\frac{D L}{L} \sum_{k=1}^{n-1} s_{n,k} D^{k}\right).
\end{equation}
\end{lem}

\begin{proof}
The first equality is a direct result of the definition of $\mathfrak{L}_i$ and Lemma \ref{propertiesofCfunctions}, i.e., we have
\begin{equation*}
\mathfrak{L}_i=\mathfrak{L}_{n+1-i}
\end{equation*}
for all $1\leq{i}\leq{n}$.

Using identities (\ref{PropertiesofSOperator1}) and (\ref{PropertiesofSOperator2}), we obtain the following identity
\begin{equation}\label{eq:PFFactored}
\mathfrak{L}_n\cdots\mathfrak{L}_1I(x,z) =z^{-n}I(x,z).
\end{equation}
For $0\leq i \leq n-1$, we defined the functions $\widetilde{I}_i(x,z)$ by equation (\ref{eq:tilde_I_i_functions}). Due to particular form (\ref{eq:I_func_as_sum_of_tilde_I_i}) of $I$-function, we see that the set $\{{\widetilde{I}_i(x,z)}\}_{0\leq{i}\leq{n-1}}$ is a basis of solutions to equation (\ref{eq:PF3}) and equation (\ref{eq:PFFactored}). Moreover, for all $1\leq{i}\leq{n-1}$, we have
\begin{equation}
\mathfrak{L}_{i+1}\mathfrak{L}_{i}=\frac{1}{C_{i+1}}D\frac{1}{C_{i}}D=\frac{1}{C_{i+1}C_i}(D-X_i)D\quad\text{with } X_i=\frac{DC_i}{C_i}.
\end{equation}
Applying this procedure repeatedly, we see that
\begin{equation*}
\begin{split}
\mathfrak{L}_n\cdots\mathfrak{L}_1
&=\frac{1}{\prod_{i=1}^nC_i}(D-\alpha_{n-1})\cdots(D-\alpha_1)D\\
&=L^{-n}(D-\alpha_{n-1})\cdots(D-\alpha_1)D\quad \text{by Lemma \ref{propertiesofCfunctions}},
\end{split}
\end{equation*}
with
\begin{equation*}
\alpha_i=\sum_{j=1}^{i}X_j
\end{equation*}
for $1\leq{i}\leq{n-1}$. This shows that equations (\ref{eq:PF3}) and (\ref{eq:PFFactored}) have the same leading coefficients. Since both equations have the same solution space and the same leading coefficient, some elementary arguments from the theory of linear ordinary differential equations imply that (\ref{eq:PF3}) and (\ref{eq:PFFactored}) must be exactly the same equation. This completes the proof.
\end{proof}

\subsection{Quantum product}\label{sec:quantum_product}
Let $\gamma=\sum_{i=0}^{n-1}t_i\phi_i\in H^{\star}_{\mathrm{T,Orb}}\left(\left[\mathbb{C}^n/\mathbb{Z}_n\right]\right)$. The full genus $0$ Gromov-Witten potential is defined to be
\begin{equation}\label{eqn:full_GW_potential}
\begin{split}
    \mathcal{F}_0^{\left[\mathbb{C}^n / \mathbb{Z}_n\right]}(t, \Theta)
    &=\sum_{m=0}^{\infty} \sum_{d=0}^{\infty} \frac{1}{m ! d !}\int_{\left[\overline{M}_{0, m+d}^{\mathrm{orb}}\left(\left[\mathbb{C}^{n} / \mathbb{Z}_{n}\right], 0\right)\right]^{v i r}}  \prod_{i=1}^m \operatorname{ev}_i^*(\gamma) \prod_{i=m+1}^{m+d} \mathrm{ev}_i^*\left(\Theta \phi_1\right)\\
    &=\sum_{m=0}^{\infty} \sum_{d=0}^{\infty} \frac{1}{m ! d !}\left\langle \underbrace{\gamma,...,\gamma}_{m-\text{times}}, \underbrace{\Theta\phi_1,...,\Theta\phi_1}_{d-\text{times}} \right\rangle_{0, m+d}^{\CnZn}.
\end{split}
\end{equation}

The orbifold Poincar\'e pairing 
\begin{equation*}
    g(-,-): H^{\star}_{\mathrm{T,Orb}}\left(\left[\mathbb{C}^n/\mathbb{Z}_n\right]\right)\times H^{\star}_{\mathrm{T,Orb}}\left(\left[\mathbb{C}^n/\mathbb{Z}_n\right]\right)\to \mathbb{Q}(\lambda_0,...,\lambda_{n-1})
\end{equation*}
in the basis $\{\phi_0,...,\phi_{n-1}\}$ and under the specialization (\ref{eqn:specialization}), has the matrix representation  $G=[G_{ij}]$ given by
\begin{equation*}
G_{ij}=g(\phi_i,\phi_j)=
\begin{cases}
\frac{1}{n}\text{ if } i+j=0\mod{n},\\
0\text{ if } i+j\neq 0\mod{n}
\end{cases}
=\frac{1}{n}\delta_{\mathrm{Inv}(i),j}=\frac{1}{n}\delta_{i,\mathrm{Inv}(j)}.
\end{equation*}
Its inverse $G^{-1}=[G^{ij}]$ is given by
\begin{equation*}
G^{ij}=
n\delta_{\mathrm{Inv}(i),j}=n\delta_{i,\mathrm{Inv}(j)}
\end{equation*}
where $0\leq i,j \leq n-1$.

The quantum product $\bullet_\gamma$ at $\gamma\in H^{\star}_{\mathrm{T,Orb}}\left(\left[\mathbb{C}^n/\mathbb{Z}_n\right]\right)$ is a product structure on $H^{\star}_{\mathrm{T,Orb}}\left(\left[\mathbb{C}^n/\mathbb{Z}_n\right]\right)$. It can be defined as follows:
\begin{equation*}
    g(\phi_i\bullet_\gamma \phi_j, \phi_k):=\frac{\partial^3}{\partial t_i\partial t_j\partial t_k}\mathcal{F}_0^{\left[\mathbb{C}^n / \mathbb{Z}_n\right]}(t, \Theta).
\end{equation*}
In what follows, we focus on the quantum product $\bullet_{\gamma=0}$ at $\gamma=0\in H^{\star}_{\mathrm{T,Orb}}\left(\left[\mathbb{C}^n/\mathbb{Z}_n\right]\right)$, which we denote by $\bullet$. Note that $\bullet$ still depends on $\Theta$.

\begin{lem}
\begin{equation}\label{twopointfunctioncalc}
\left\langle\left\langle\phi_i,\phi_j\right\rangle\right\rangle_{0,2}^{\left[\mathbb{C}^n / \mathbb{Z}_n\right]}=\begin{cases} 
      0 & \text{if}\quad i+j\neq n-1, \\
      \frac{1}{n}\mathfrak{L}_{i}...\mathfrak{L}_{0}I_{i+1} & \text{if}\quad i+j=n-1.
   \end{cases}
\end{equation}
\end{lem}

\begin{proof}
By expanding equations (\ref{PropertiesofSOperator1}), (\ref{PropertiesofSOperator2}) and matching the coefficients of $z^{-1}$, we obtain the following identity for any $0\leq j\leq n-1$:
\begin{equation*}
\phi^0\left\langle\left\langle\phi_0,\phi_i\right\rangle\right\rangle_{0,2}^{\left[\mathbb{C}^n / \mathbb{Z}_n\right]}+\phi^1\left\langle\left\langle\phi_1,\phi_i\right\rangle\right\rangle_{0,2}^{\left[\mathbb{C}^n / \mathbb{Z}_n\right]}+...+\phi^{n-1}\left\langle\left\langle\phi_{n-1},\phi_i\right\rangle\right\rangle_{0,2}^{\left[\mathbb{C}^n / \mathbb{Z}_n\right]}=\mathfrak{L}_{i}...\mathfrak{L}_{0}I_{i+1}\phi_{i+1}.
\end{equation*}
Equating the coefficients, we complete the proof.
\end{proof}

\begin{lem}\label{lem:2point3point}
For any $0\leq i,j \leq n-1$, we have
\begin{equation*}
\frac{1}{C_1}D\left\langle\left\langle\phi_i,\phi_j\right\rangle\right\rangle_{0,2}^{\left[\mathbb{C}^n / \mathbb{Z}_n\right]}=\left\langle\left\langle\phi_i,\phi_j,\phi_1\right\rangle\right\rangle_{0,3}^{\left[\mathbb{C}^n / \mathbb{Z}_n\right]}.
\end{equation*}
\end{lem}

\begin{proof}
The proof is the following direct computation:
\begin{align*}
\frac{1}{C_1}D\left\langle\left\langle\phi_i,\phi_j\right\rangle\right\rangle_{0,2}^{\left[\mathbb{C}^n / \mathbb{Z}_n\right]}
=&\frac{1}{D\Theta}D\left\langle\left\langle\phi_i,\phi_j\right\rangle\right\rangle_{0,2}^{\left[\mathbb{C}^n / \mathbb{Z}_n\right]}\\
&=\frac{1}{D\Theta}\sum_{d=1}^{\infty} \frac{\Theta^{d-1}D\Theta}{(d-1)!} \int_{\left[\overline{M}_{g, d+2}^{\mathrm{orb}}\left(\left[\mathbb{C}^{n} / \mathbb{Z}_{n}\right], 0\right)\right]^{v i r}} \mathrm{ev}_{1}^{*}\left(\phi_i\right)\mathrm{ev}_{2}^{*}\left(\phi_j\right)\prod_{l=3}^{d+2} \mathrm{ev}_{l}^{*}\left(\phi_{1}\right)\\
&=\sum_{d=1}^{\infty} \frac{\Theta^{d-1}}{(d-1)!} \int_{\left[\overline{M}_{g, d+2}^{\mathrm{orb}}\left(\left[\mathbb{C}^{n} / \mathbb{Z}_{n}\right], 0\right)\right]^{v i r}} \mathrm{ev}_{1}^{*}\left(\phi_i\right)\mathrm{ev}_{2}^{*}\left(\phi_j\right)\prod_{l=3}^{d+2} \mathrm{ev}_{l}^{*}\left(\phi_{1}\right)\\
&=\sum_{d=0}^{\infty} \frac{\Theta^{d}}{d!} \int_{\left[\overline{M}_{g, d+3}^{\mathrm{orb}}\left(\left[\mathbb{C}^{n} / \mathbb{Z}_{n}\right], 0\right)\right]^{v i r}} \mathrm{ev}_{1}^{*}\left(\phi_i\right)\mathrm{ev}_{2}^{*}\left(\phi_j\right)\mathrm{ev}_{3}^{*}\left(\phi_1\right)\prod_{l=4}^{d+3} \mathrm{ev}_{l}^{*}\left(\phi_{1}\right)\\
&=\left\langle\left\langle\phi_i,\phi_j,\phi_1\right\rangle\right\rangle_{0,3}^{\left[\mathbb{C}^n / \mathbb{Z}_n\right]}.
\end{align*}
The fist line follows from equation (\ref{eq:smallmirrorthm}) and the definition of $C_1$.
\end{proof}

\begin{prop}\label{generatingquantumprod}
For all $i\geq 0$, the quantum product at $0\in H^{\star}_{\mathrm{T,Orb}}\left(\left[\mathbb{C}^n/\mathbb{Z}_n\right]\right)$ satisfies
\begin{equation*}
\phi_1\bullet\phi_i=\frac{C_{i+1}}{C_1}\phi_{i+1}.
\end{equation*}
\end{prop}

\begin{proof}
Initially, we assume $0\leq i \leq n-1$. Using equation (\ref{twopointfunctioncalc}) and Lemma \ref{lem:2point3point}, we obtain
\begin{equation*}
\left\langle\left\langle\phi_1,\phi_i,\phi_j\right\rangle\right\rangle_{0,3}^{\left[\mathbb{C}^n / \mathbb{Z}_n\right]}
=\frac{D\left\langle\left\langle\phi_i,\phi_j\right\rangle\right\rangle_{0,2}^{\left[\mathbb{C}^n / \mathbb{Z}_n\right]}}{C_1}
=\begin{cases} 
      0 & \text{if}\quad i+j\neq n-1, \\
      \frac{1}{n}\frac{C_{i+1}}{C_1} & \text{if}\quad i+j=n-1.
\end{cases}
\end{equation*}
Write
\begin{equation*}
\phi_1\bullet\phi_i=\sum_{l=0}^{n-1}a_{li}\phi_l.
\end{equation*}
Then, for $0\leq j \leq n-1$, we have
\begin{equation*}
g\left(\phi_1\bullet\phi_i,\phi_j\right)=\frac{1}{n}a_{\mathrm{Inv}(j)i}.
\end{equation*}
On the other hand, the relation $g(X\bullet Y,Z)=\left\langle\left\langle X,Y,Z\right\rangle\right\rangle_{0,3}^{\left[\mathbb{C}^n / \mathbb{Z}_n\right]}$ gives 
\begin{equation*}
g\left(\phi_1\bullet\phi_i,\phi_j\right)
=\begin{cases} 
      0 & \text{if}\quad i+j\neq n-1, \\
      \frac{1}{n}\frac{C_{i+1}}{C_1} & \text{if}\quad i+j=n-1.
\end{cases}
\end{equation*}
So, we obtain
\begin{equation*}
a_{li}
=\begin{cases} 
      0 & \text{if}\quad i+\mathrm{Inv}(l)\neq n-1, \\
      \frac{C_{i+1}}{C_1} & \text{if}\quad i+\mathrm{Inv}(l)=n-1
\end{cases}
=\frac{C_{i+1}}{C_1}\delta_{i, \mathrm{Ion}(l)-1}.
\end{equation*}
Parts (\ref{Cfunctions1}) and (\ref{Cfunctions4}) of Lemma \ref{propertiesofCfunctions} finish the proof.
\end{proof}

\begin{cor}\label{quantumprod}
For any $i,j\geq 0$, the quantum product at $0\in H^{\star}_{\mathrm{T,Orb}}\left(\left[\mathbb{C}^n/\mathbb{Z}_n\right]\right)$ is given by
\begin{equation*}
\phi_i\bullet\phi_j=\frac{K_{i+j}}{K_iK_j}\phi_{i+j}
\end{equation*}
and hence the genus $0$, $3$-point Gromov-Witten invariants are
\begin{equation*}
\left\langle\left\langle\phi_i,\phi_j,\phi_k\right\rangle\right\rangle_{0,3}^{\left[\mathbb{C}^n/ \mathbb{Z}_n\right]}=\frac{K_{i+j}}{K_iK_j}\frac{1}{n}\delta_{\mathrm{Inv}(i+j \,\,\mathrm{mod}\,\,n), k}.   
\end{equation*}
\end{cor}
\begin{proof}
Using Proposition \ref{generatingquantumprod} and noting that $C_0=1$, inductively we show that for any $l\geq 0$ we have
\begin{equation*}
    \phi_l=\frac{C_1^l}{K_l}\underbrace{\phi_1\bullet...\bullet\phi_1}_{l-\text{times}}.
\end{equation*}
This implies
\begin{equation*}
      \phi_i\bullet\phi_j=\frac{C_1^{i+j}}{K_iK_j}\underbrace{\phi_1\bullet...\bullet\phi_1}_{i+j-\text{times}}=\frac{C_1^{i+j}}{K_iK_j}\frac{K_{i+j}}{C_1^{i+j}}\phi_{i+j}, 
\end{equation*}
and the genus $0$, $3$-point Gromov-Witten invariants part of the lemma follows from
\begin{equation*}
\left\langle\left\langle\phi_i,\phi_j,\phi_k\right\rangle\right\rangle_{0,3}^{\left[\mathbb{C}^n/ \mathbb{Z}_n\right]}=g(\phi_i\bullet\phi_j,\phi_k)=\frac{K_{i+j}}{K_iK_j}\frac{1}{n}\delta_{\text{Inv}(i+j \text{ mod }n), k}.    
\end{equation*}
\end{proof}

For all $i\geq{0}$, define
\begin{equation}\label{def:phi_tilde}
\widetilde{\phi}_i=\frac{K_i}{L^i}\phi_i.
 \end{equation}

\begin{lem}\label{lem:quantum_for_tildephi} 
For all $i,j\geq{0}$, we have $\widetilde{\phi}_{i+n}=\widetilde{\phi}_{i}$ and $\widetilde{\phi}_i\bullet\widetilde{\phi}_j=\widetilde{\phi}_{i+j}$.
\end{lem}

\begin{proof}
The first part follows from
\begin{equation*}
\widetilde{\phi}_{i+n}=\frac{K_{i+n}}{L^{i+n}}\phi_{i+n}=\frac{K_{i}L^n}{L^{i+n}}\phi_{i}=\frac{K_i}{L^i}\phi_i.
\end{equation*}
Here, we used the properties of $K_i$ listed in Section \ref{subsection:Birkhoff_Factorizaton} and proved in Appendix \ref{appendix:I-function-Part1}.
The second part follows from
\begin{equation*}
\widetilde{\phi}_i\bullet\widetilde{\phi}_j=\frac{K_i}{L^i}\phi_i\bullet\frac{K_j}{L^j}\phi_j=\frac{K_iK_j}{L^{i+j}}\phi_i\bullet\phi_j=\frac{K_iK_j}{L^{i+j}}\frac{K_{i+j}}{K_iK_j}\phi_{i+j}=\frac{K_{i+j}}{L^{i+j}}\phi_{i+j}=\widetilde{\phi}_{i+j}.
\end{equation*}
\end{proof}

For $\alpha\geq 0$, define
\begin{equation}\label{def:idempotent}
    e_{\alpha}=\frac{1}{n}\sum_{i=0}^{n-1}\zeta^{-\alpha{i}}\widetilde{\phi}_i
\end{equation}
where $\zeta=e^{\frac{2\pi\sqrt{-1}}{n}}$ is an $n^{\text{th}}$ root of unity.

By Lemma \ref{lem:quantum_for_tildephi}, for any $0\leq j \leq n-1$, we have
\begin{equation}\label{eq:tilde_phi1_dot_e_alpha}
\widetilde{\phi}_j\bullet e_{\alpha}
=\frac{1}{n}\sum_{i=0}^{n-1}\zeta^{-\alpha{i}}\widetilde{\phi}_{i+j}
=\zeta^{\alpha{j}}\frac{1}{n}\sum_{i=0}^{n-1}\zeta^{-\alpha{(i+j)}}\widetilde{\phi}_{i+j}
=\zeta^{\alpha}\frac{1}{n}\sum_{i=0}^{n-1}\zeta^{-\alpha{i}}\widetilde{\phi}_{i}=\zeta^{\alpha{j}}e_{\alpha}.
\end{equation}
Consecutive application of the identity (\ref{eq:tilde_phi1_dot_e_alpha}) gives us
\begin{equation*}
e_{\alpha}\bullet e_{\beta}=\frac{1}{n}\sum_{i=0}^{n-1}\zeta^{-\beta{i}}\widetilde{\phi}_i\bullet e_{\beta}=\frac{1}{n}\sum_{i=0}^{n-1}\zeta^{-\alpha{i}}\zeta^{\beta{i}}e_{\beta}=\frac{1}{n}\sum_{i}^{n-1}\zeta^{(\beta-\alpha)i}e_{\beta}=\delta_{\alpha,\beta}e_{\beta}.
\end{equation*}
Hence, we obtain the following result.
\begin{prop}
The quantum product at $0\in H^{\star}_{\mathrm{T,Orb}}\left(\left[\mathbb{C}^n/\mathbb{Z}_n\right]\right)$ is semisimple with the idempotent basis $\{e_\alpha\}_{0\leq \alpha \leq {n-1}}$.
\end{prop}

\subsection{Frobenius structure}\label{sec:frob_str}
We describe in detail some ingredients of the Frobenius structure obtained from genus $0$ Gromov-Witten theory of $\CnZn$. We refer the readers to \cite{lp} for generalities on the Frobenius structure arising in Gromov-Witten theory.

By the results in Section \ref{sec:quantum_product}, the Frobenius structure\footnote{The Frobenius manifold here is over the ring $\mathbb{C}[\![\Theta]\!]$, or over the ring $\mathbb{C}[\![x]\!]$ by mirror map (\ref{eq:mirrortransform}).} on $H^{\star}_{\mathrm{T,Orb}}\left(\left[\mathbb{C}^n/\mathbb{Z}_n\right]\right)$ defined by the Gromov-Witten theory of $\CnZn$ is semisimple in a neighborhood of $0\in H^{\star}_{\mathrm{T,Orb}}\left(\left[\mathbb{C}^n/\mathbb{Z}_n\right]\right)$.

For any $0\leq \alpha \leq n-1$, we have
\begin{equation}\label{eq:g_e_alpha_phi_0}
g\left(e_\alpha, \phi_0\right)=\frac{1}{n} \sum_{i=0}^{n-1} \zeta^{-\alpha i} g\left(\widetilde{\phi}_i, \phi_0\right)=\frac{1}{n} \sum_{i=0}^{n-1} \zeta^{-\alpha i}\frac{K_i}{L^i}\frac{1}{n}\delta_{\mathrm{Inv}(i),0}=\frac{1}{n^2}.
\end{equation}
Using the identity (\ref{eq:g_e_alpha_phi_0}) and the Frobenius property, we calculate the metric $g$ in the idempotent basis $\{e_\alpha\}$:
\begin{equation*}
g\left(e_\alpha, e_\alpha\right)
=g\left(e_\alpha, e_\alpha \bullet \phi_0\right)=g\left(e_\alpha \bullet e_\alpha, \phi_0\right)=g\left(e_\alpha, \phi_0\right)=\frac{1}{n^2}.
\end{equation*}
So, the normalized idempotents are given by
\begin{equation*}
\widetilde{e}_{\alpha}=\frac{e_{\alpha}}{\sqrt{g\left(e_{\alpha},e_{\alpha}\right)}}=ne_{\alpha}.
\end{equation*}
The transition matrix $\Psi$ is given by $\Psi_{\alpha{i}}=g\left(\widetilde{e}_{\alpha},\phi_i\right)$ where $0\leq \alpha,i\leq n-1$. By equations (\ref{eq:tilde_phi1_dot_e_alpha}) and (\ref{eq:g_e_alpha_phi_0}), we  calculate
\begin{equation*}
\begin{aligned}
\Psi_{\alpha{i}}
&=g\left(\widetilde{e}_{\alpha},\phi_i\right)=n\frac{L^i}{K_i}g\left(e_{\alpha},\widetilde{\phi}_i\right)=n\frac{L^i}{K_i}g\left(e_{\alpha},\widetilde{\phi}_i\bullet \phi_0\right)\\
&=n\frac{L^i}{K_i}g\left(\widetilde{\phi}_i\bullet e_{\alpha},\phi_0\right)=n\zeta^{\alpha{i}}\frac{L^i}{K_i}g(e_\alpha,\phi_0)=\frac{1}{n}\zeta^{\alpha{i}}\frac{L^i}{K_i}.
\end{aligned}
\end{equation*}
\begin{lem}
The inverse of the transition matrix $\Psi^{-1}=\left[\Psi^{-1}_{j\beta}\right]$ is given by
\begin{equation*}
\Psi^{-1}_{{j}\beta}=\zeta^{-\beta{j}}\frac{K_{j}}{L^{j}}\quad\text{where}\quad 0\leq\beta,j\leq n-1.
\end{equation*}
\end{lem}

\begin{proof}
We calculate
\begin{equation*}
\left[\Psi \Psi^{-1}\right]_{\alpha\beta}
=\sum_{i=0}^{n-1}\Psi_{\alpha{i}}\Psi^{-1}_{{i}\beta}
=\sum_{i=0}^{n-1}\frac{1}{n}\zeta^{\alpha{i}}\frac{L^{i}}{K_{i}}\zeta^{-\beta{i}}\frac{K_{i}}{L^{i}}\\
=\sum_{i=0}^{n-1}\frac{1}{n}\zeta^{i(\alpha-\beta)}
=\delta_{\alpha,\beta}.
\end{equation*}
\end{proof}

Let $\left\{u^{\alpha}\right\}_{\alpha=0}^{n-1}$ be canonical coordinates associated to the idempotent basis $\left\{e_{\alpha}\right\}_{\alpha=0}^{n-1}$ which satisfy
\begin{equation*}
u^{\alpha}\left(t_i=0,\Theta=0\right)=0.
\end{equation*}
Since $e_{\alpha}=\frac{\partial}{\partial{u^{\alpha}}}$, we have
\begin{equation}\label{eq:idempotenttophi_1}
\sum_{\alpha=0}^{n-1}\frac{\partial {u^{\alpha}}}{\partial{t_1}}e_{\alpha}=\phi_1.
\end{equation}

\begin{lem}\label{canonicalcoorder}
We have
\begin{equation*}
\frac{\partial u^{\alpha}}{\partial {t_1}}|_{t=0}=\zeta^{\alpha}\frac{L}{C_1}.
\end{equation*}
\end{lem}

\begin{proof}
The result is obtained by the following calculation: at $t=0$, we have
\begin{align*}
\frac{\partial u^{\alpha}}{\partial {t_1}}|_{t=0}{e_{\alpha}}
=\sum_{\beta=0}^{n-1}\frac{\partial {u^{\beta}}}{\partial {t_1}}|_{t=0}\delta_{\alpha,\beta}{e_{\alpha}}
=\phi_1\bullet{e_{\alpha}}
=\frac{L}{K_1}\widetilde{\phi}_1\bullet e_\alpha
=\zeta^{\alpha}\frac{L}{C_1}e_{\alpha}.
\end{align*}
The equalities follow from equation (\ref{eq:tilde_phi1_dot_e_alpha}) and equation (\ref{eq:idempotenttophi_1}).
\end{proof}

The $R$-matrix has a central role in the Givental-Teleman classification of semisimple cohomological field theories. Let the $R$-matrix of the Frobenius structure associated to the ($\mathrm{T}$-equivariant) Gromov-Witten theory of $\CnZn$ near the semisimple point $0$ be denoted by 
\begin{equation*}
    R(z)=\sum_{k\geq 0} R_k z^k\in \mathrm{Id}+z\mathrm{End}(H^*_{\mathrm{T,Orb}}(\CnZn))[\![z]\!].
\end{equation*}
By the definition of $R$-matrix, $R(z)$ satisfies the symplectic condition 
\begin{equation*}
    R(z)\cdot R(-z)^*=\text{Id}.
\end{equation*}
Let $U$ be the diagonal matrix
\begin{equation*}
U=\diag(u^0,\ldots,u^{n-1})
\end{equation*}
associated to canonical coordinates $\left\{u^{\alpha}\right\}_{\alpha=0}^{n-1}$.
The $R$-matrix $R(z)$ also satisfies the following flatness equation
\begin{equation}\label{eqn:defn_of_R}
z(d\Psi^{-1})R+z\Psi^{-1}(dR)+\Psi^{-1}R (dU)-\Psi^{-1}(dU) R=0,    
\end{equation}
see \cite[Chapter 1, Section 4.6]{lp} and \cite[Proposition 1.1]{g1}. Here $d$ denotes total derivative with respect to $\{t_i\}$.

We examine the dependence on parameters of the full genus $0$ Gromov-Witten potential (\ref{eqn:full_GW_potential}): 
\begin{equation*}
\begin{split}
&\mathcal{F}_0^{\left[\mathbb{C}^n / \mathbb{Z}_n\right]}(t, \Theta)\\
&=\sum_{m=0}^{\infty} \sum_{d=0}^{\infty} \frac{1}{m ! d !}\left\langle \underbrace{\gamma,...,\gamma}_{m-\text{times}}, \underbrace{\Theta\phi_1,...,\Theta\phi_1}_{d-\text{times}} \right\rangle_{0, m+d}^{\CnZn}\\
&=\sum_{m=0}^{\infty} \sum_{d=0}^{\infty} \frac{1}{m ! d !}\left\langle \underbrace{\gamma|_{t_1=0}+t_1\phi_1,...,\gamma|_{t_1=0}+t_1\phi_1}_{m-\text{times}}, \underbrace{\Theta\phi_1,...,\Theta\phi_1}_{d-\text{times}} \right\rangle_{0, m+d}^{\CnZn}\\
&=\sum_{m=0}^{\infty} \sum_{d=0}^{\infty} \frac{1}{m ! d !}\sum_{b=0}^m\binom{m}{b}\left\langle \underbrace{\gamma|_{t_1=0},..., \gamma|_{t_1=0}}_{(m-b)-\text{times}}, \underbrace{t_1\phi_1,...,t_1\phi_1}_{b-\text{times}}, \underbrace{\Theta\phi_1,...,\Theta\phi_1}_{d-\text{times}} \right\rangle_{0, m+d}^{\CnZn}\\
&=\sum_{m=0}^{\infty} \sum_{d=0}^{\infty} \sum_{b=0}^m\frac{1}{b ! d ! (m-b)!}\left\langle \underbrace{\gamma|_{t_1=0},..., \gamma|_{t_1=0}}_{(m-b)-\text{times}}, \underbrace{t_1\phi_1,...,t_1\phi_1}_{b-\text{times}}, \underbrace{\Theta\phi_1,...,\Theta\phi_1}_{d-\text{times}} \right\rangle_{0, m+d}^{\CnZn}\\
&=\sum_{m=0}^{\infty} \sum_{d=0}^{\infty} \sum_{b=0}^m\frac{(b+d)!}{b ! d ! (m-b)!(b+d)!}\left\langle \underbrace{\gamma|_{t_1=0},..., \gamma|_{t_1=0}}_{(m-b)-\text{times}}, \underbrace{t_1\phi_1,...,t_1\phi_1}_{b-\text{times}}, \underbrace{\Theta\phi_1,...,\Theta\phi_1}_{d-\text{times}} \right\rangle_{0, m+d}^{\CnZn}\\
&=\sum_{m=0}^{\infty} \sum_{d=0}^{\infty} \frac{1}{m ! d !}\left\langle \underbrace{\gamma,...,\gamma}_{m-\text{times}}, \underbrace{(\Theta+t_1)\phi_1,...,(\Theta+t_1)\phi_1}_{d-\text{times}} \right\rangle_{0, m+d}^{\CnZn},\\
\end{split}    
\end{equation*}
namely
\begin{equation*}
    \mathcal{F}_0^{\left[\mathbb{C}^n / \mathbb{Z}_n\right]}(t, \Theta)=\mathcal{F}_0^{\left[\mathbb{C}^n / \mathbb{Z}_n\right]}(t|_{t_1=0}, \Theta+t_1).
\end{equation*}
That is, $\mathcal{F}_0^{\left[\mathbb{C}^n / \mathbb{Z}_n\right]}(t, \Theta)$ depends on $t_1$ and $\Theta$ through $\Theta+t_1$.  

The purpose of introducing the seemingly redundant variable $\Theta$ is to construct an equation that plays the role of divisor equation in orbifold Gromov-Witten theory, see \cite[Section 2.2]{bg} for a detailed discussion.

It follows from the construction of semisimple Frobenius structures that its ingredients also depend on $t_1$ and $\Theta$ through $\Theta+t_1$, for example, $$u^\alpha(t,\Theta)=u^\alpha(t|_{t_1=0}, \Theta+t_1).$$ 
In particular, the operator
\begin{equation}\label{annihilatoroperator}
\frac{\partial}{\partial{t_1}}-\frac{\partial}{\partial\Theta}
\end{equation}
annihilates the canonical coordinates $u^{\alpha}$. More precisely, we have
\begin{equation*}
\left(\frac{\partial u^\alpha}{\partial t_1}\right)|_{t=0}=\left(\frac{\partial}{\partial t_1}u^\alpha(t|_{t_1=0}, \Theta+t_1)\right)|_{t=0}=\left(\frac{\partial}{\partial \Theta}u^\alpha(t|_{t_1=0}, \Theta+t_1)\right)|_{t=0}=\frac{d}{d\Theta}\left(u^\alpha(t,\Theta)|_{t=0}\right).     
\end{equation*}
As a result, we have
\begin{equation*}
\left(\frac{\partial u^{\alpha}}{\partial t_1}\right)|_{t=0}=\frac{d(u^{\alpha}|_{t=0})}{d\Theta}=\frac{d(u^{\alpha}|_{t=0})}{dx}\frac{dx}{d\Theta}.
\end{equation*}
By Lemma \ref{canonicalcoorder}, we have
\begin{equation}\label{canonicalcoorder2}
\frac{d(u^{\alpha}|_{t=0})}{dx}=\zeta^{\alpha}L\frac{1}{x}
\end{equation}
at the semisimple point $0\in H^*_{\mathrm{T,Orb}}\left([\mathbb{C}^n/\mathbb{Z}_n]\right)$, i.e. at $t=0$.

Since $\mathcal{F}_0^{\left[\mathbb{C}^n / \mathbb{Z}_n\right]}(t, \Theta)$ depends on $t_1$ and $\Theta$ through $\Theta+t_1$, it follows that $\Psi$ and $R(z)$ also depend on $t_1$ and $\Theta$ through $\Theta+t_1$. Hence, the operator (\ref{annihilatoroperator}) also annihilates\footnote{An argument for this (for a different target space) from the CohFT viewpoint can be found in \cite[Section 3.3]{pt}.} $\Psi$ and $R(z)$, more precisely we have
\begin{equation*}
\left(\frac{\partial}{\partial t_1}\Psi\right)|_{t=0}=\frac{d}{d \Theta} (\Psi|_{t=0}), \quad
    \left(\frac{\partial}{\partial t_1}R(z)\right)|_{t=0}=\frac{d}{d \Theta} (R(z)|_{t=0}).
\end{equation*}
Next, in equation (\ref{eqn:defn_of_R}), we set all $t_i$'s to $0$ except $t_1$ and only consider $\frac{d}{dt_1}$. Since $U$, $\Psi$ and $R(z)$ are annihilated by the operator (\ref{annihilatoroperator}), it follows that when setting $t=0$, i.e. consider the restriction to the semisimple point $0\in H^*_{\mathrm{T,Orb}}\left([\mathbb{C}^n/\mathbb{Z}_n]\right)$, (\ref{eqn:defn_of_R}) implies
\begin{equation}\label{eqn:defn_of_R_2}
\begin{split}
&z\left(\frac{d}{d\Theta}(\Psi^{-1}|_{t=0})\right)R|_{t=0}+z(\Psi^{-1}|_{t=0})\left(\frac{d}{d\Theta}(R|_{t=0})\right)\\
&+(\Psi^{-1}|_{t=0})(R|_{t=0}) \left(\frac{d}{d\Theta}(U|_{t=0})\right)-(\Psi^{-1}|_{t=0})\left(\frac{d}{d\Theta}(U|_{t=0})\right) (R|_{t=0})=0. 
\end{split}
\end{equation}

In what follows, we only consider the Frobenius structure restricted to the semisimple point $0\in H^*_{\mathrm{T,Orb}}\left([\mathbb{C}^n/\mathbb{Z}_n]\right)$. For simplicity, we abuse notations and write $R|_{t=0}, \Psi|_{t=0}, U|_{t=0}$ as $R, \Psi, U$.

Using the mirror map $\Theta(x)=I_1(x)$ and the chain rule $$\frac{d}{d\Theta}=\frac{dx}{d\Theta}\frac{d}{dx},$$
we rewrite equation (\ref{eqn:defn_of_R_2}) as 
\begin{equation*}
   z\left(x\frac{d}{dx}\Psi^{-1}\right)R+z\Psi^{-1}\left(x\frac{d}{dx}R\right)+\Psi^{-1}R \left(x\frac{d}{dx}U\right)-\Psi^{-1}\left(x\frac{d}{dx}U\right) R=0. 
\end{equation*}
By matching coefficients of $z^k$, we can further rewrite it as\footnote{We set $R_k=0$ for $k<0$.}
\begin{equation}\label{flatness2}
D\left(\Psi^{-1}R_{k-1}\right)+\left(\Psi^{-1}R_k\right)DU-\Psi^{-1}\left(DU\right)\Psi\left(\Psi^{-1}R_k\right)=0   
\end{equation}
or equivalently
\begin{equation}\label{flatness1}
\Psi\left(D\Psi^{-1}\right)R_{k-1}+DR_{k-1}+R_k\left(DU\right)-\left(DU\right)R_k=0
\end{equation}
for $k\geq{0}$. Here $D=x\frac{d}{dx}$ as before. 

By equation (\ref{canonicalcoorder2}), we have, when $t=0$,  
\begin{equation}\label{DU}
DU=\diag(L,L\zeta,...,L\zeta^{n-1}).
\end{equation}

For $k\geq 0$, define the matrix $P_k$ by
$$P_k=\Psi^{-1}R_k$$ after being restricted to the semisimple point $0\in H^*_{\mathrm{T,Orb}}\left([\mathbb{C}^n/\mathbb{Z}_n]\right)$. Let $P_{i,j}^k$ denote the $(i,j)$ entry of the matrix $P_k$ where $0\leq i,j \leq{n-1}$.

\begin{lem}\label{PMatrixEqns}
For ${0\leq{i,j}\leq{n-1}}$ and $k\geq{0}$, we have
\begin{equation*}
DP_{i,j}^{k-1}=
C_{\mathrm{Ion}(i)}P_{\mathrm{Ion}(i)-1,j}^k-P_{i,j}^kL\zeta^j.
\end{equation*}
\end{lem}

\begin{proof}
Equation (\ref{flatness2}) can be rewritten as
\begin{equation*}
D\left(\Psi^{-1}R_{k-1}\right)=\Psi^{-1}\left(DU\right)\Psi\left(\Psi^{-1}R_k\right)-\left(\Psi^{-1}R_k\right)DU
\end{equation*}
which is the same as
\begin{equation}\label{flatness3}
DP_{k-1}=\Psi^{-1}DU\Psi P_{k}-P_{k}DU.
\end{equation}
We see that
\begin{align}\label{PsiDUPsi1}
\begin{split}
   \left(\Psi^{-1}DU\Psi\right)_{ij}
  =&\sum_{l=0}^{n-1}\left(\Psi^{-1}DU\right)_{il}\Psi_{lj}\\
  =&\sum_{l=0}^{n-1}\zeta^{-li}\frac{K_i}{L^i}L\zeta^l\frac{1}{n}\zeta^{lj}\frac{L^j}{K_j}\\
  =&\frac{1}{n}\frac{K_i}{K_j}\frac{L^{j+1}}{L^i}\sum_{l=0}^{n-1}\zeta^{l(j-i+1)}\\
  =&
  \begin{cases}
  \frac{K_i}{K_j}\frac{L^{j+1}}{L^i}&\quad\text{if}\quad i=j+1\mod{n},\\
  0&\quad\text{otherwise}
  \end{cases}\\
  =&
  \begin{cases}
  C_i&\quad\text{if}\quad 1\leq i\leq {n-1}\quad\text{and}\quad j={i-1},\\
  C_n&\quad\text{if}\quad i=0\quad\text{and}\quad j={n-1},\\
  0&\quad\text{otherwise}.
  \end{cases} 
\end{split}
\end{align}
Equation (\ref{PsiDUPsi1}) implies
\begin{equation}\label{PsiDUPsi2}
\left(\Psi^{-1}DU\Psi{P_{k}}\right)_{ij}
=\sum_{l=0}^{n-1}\left(\Psi^{-1}DU\Psi\right)_{il}P^k_{l,j}
=C_{\mathrm{Ion}(i)}P^k_{\mathrm{Ion}(i)-1,j}
\end{equation}
Equations (\ref{flatness3}) and (\ref{PsiDUPsi2}) finish the proof.
\end{proof}
For $0\leq{i,j}\leq{n-1}$, define
\begin{equation*}
P_{i,j}(z)=\sum_{k=0}^{\infty} P_{i,j}^{k} z^{k},\quad D_{L_j}=D+\frac{L_j}{z} \quad\text{and}\quad \mu_j=\int_0^x\frac{L_j(u)}{u}du
\end{equation*}
where $L_j=L\zeta^j$. Then Lemma \ref{PMatrixEqns} is equivalent to the following.
\begin{lem}
For $0\leq{i,j}\leq{n-1}$, we have
$D_{L_j}P_{i,j}(z)=C_{\mathrm{Ion}(i)}z^{-1}P_{\mathrm{Ion}(i)-1,j}(z)$.
\end{lem}

\begin{proof} The proof is the following direct computation:
\begin{equation*}
\begin{split}
D_{L_j} P_{i,j}(z) 
&=\sum_{k=0}^{\infty} D P_{i,j}^{k} z^{k}+\sum_{k=0}^{\infty} L_{j} P_{i,j}^{k} z^{k-1} \\
&=\sum_{k=1}^{\infty} D P_{i,j}^{k-1} z^{k-1}+\sum_{k=0}^{\infty} L_{j} P_{i,j}^{k} z^{k-1}\\
&=\sum_{k=0}^{\infty}\left(D P_{i,j}^{k-1}+L_{j} P_{i,j}^{k}\right) z^{k-1} \quad\text { because } P_{i,j}^{-1}=0 \\
&=\begin{cases}
\sum_{k=0}^{\infty}C_nP_{n-1,j}^kz^{k-1}&\quad\text{if}\quad{i=0},\\
\sum_{k=0}^{\infty}C_iP_{i-1,j}^kz^{k-1}&\quad\text{if}\quad{1\leq{i}\leq{n-1}}
\end{cases}
\quad\text{ by Lemma \ref{PMatrixEqns}}\\
&=\begin{cases}
C_nz^{-1}P_{n-1,j}(z)&\quad\text{if}\quad{i=0},\\
C_iz^{-1}P_{i-1,j}(z)&\quad\text{if}\quad{1\leq{i}\leq{n-1}}
\end{cases} \\
&=C_{\mathrm{Ion}(i)}z^{-1}P_{\mathrm{Ion}(i)-1,j}(z).
\end{split}
\end{equation*}
\end{proof}

It immediately follows that $P_{0,j}(z)$ satisfies the following differential equation
\begin{equation}
\frac{1}{C_1}D_{L_j} \cdots \frac{1}{C_n}D_{L_j}P_{0,j}(z)=z^{-n}P_{0,j}(z).
\end{equation}
By the definition of $\mathfrak{L}_i$ and equation (\ref{eqn:mujcommutation}), this equation can be rewritten as 
\begin{equation}\label{eqn:P0jODE}
\mathfrak{L}_1\cdots\mathfrak{L}_n\left(e^{\frac{\mu_{j}}{z}}P_{0,j}(z)\right)=z^{-n}e^{\frac{\mu_{j}}{z}}P_{0,j}(z)
\end{equation}
By Lemma \ref{lem:PFFactorization}, equation \eqref{eqn:P0jODE} reads as
\begin{equation}\label{eq:PDForP0jz}
L^{-n}\left(D^{n}\left(e^{\frac{\mu_{j}}{z}}P_{0,j}(z)\right)+\frac{D L}{L} \sum_{r=1}^{n-1} s_{n,r} D^{r}\left(e^{\frac{\mu_{j}}{z}}P_{0,j}(z)\right)\right)=z^{-n}e^{\frac{\mu_{j}}{z}}P_{0,j}(z).
\end{equation}
We see that $P_{0,j}(z)$ satisfies the assumption of Lemma \ref{lem:AymptoticPFSolution}. Hence, we obtain the following two results by Lemma \ref{lem:AymptoticPFSolution} and Corollary. \ref{cor:Phi_jk_identity}.
\begin{cor}\label{cor:P0jk_is_in_CL}
For $0\leq{j}\leq{n-1}$ and $k\geq{0}$, we have $P^{k}_{0,j}\in\mathbb{C}[L]\subseteq\mathbb{C}[L^{\pm{1}}]$.
\end{cor}

\begin{cor}\label{cor:LPijkidentity}
For $0\leq{j}\leq{n-1}$ and $k\geq 0$, we have
\begin{equation}
\mathds{L}_{j,1}(P^{k}_{0,j})+\frac{1}{L_j}\mathds{L}_{j,2}(P^{k-1}_{0,j})+\frac{1}{L_j^2}\mathds{L}_{j,3}(P^{k-2}_{0,j})+\cdots+\frac{1}{L_j^{n-1}}\mathds{L}_{j,n}(P^{k+1-n}_{0,j})=0
\end{equation}
where $\mathds{L}_{j,k}$ is given by equation (\ref{eq:LLjk}).
\end{cor}

\section{Rings of functions}\label{sec:diff_ring}
The purpose of this section is to define and study rings of functions that contain various ingredients of the Gromov-Witten theory of $\CnZn$.

As described in Theorem \ref{thm:mirrorthm}, the $J$-function of $[\mathbb{C}^n/\mathbb{Z}_n]$ is explicitly given by the $I$-function, which defines the functions $I_k(x), k\geq 0$ via (\ref{eq:IfuncAsSumIk}). The Birkhoff factorization procedure in Section \ref{subsection:Birkhoff_Factorizaton} leads to the functions $L$ and $C_i, K_i, i\geq 0$, which also appear in the explicit formula for quantum product in Proposition \ref{generatingquantumprod} and Corollary \ref{quantumprod}.

Since the genus $0$ Gromov-Witten theory of $\CnZn$ is semisimple near $0$, Givental-Teleman classification yields a formula for Gromov-Witten potentials of $\CnZn$ in terms of the action of the $R$-matrix of $\CnZn$, see equation (\ref{eqn:formula_Fg}) below. As discussed in Section \ref{sec:frob_str}, this $R$-matrix satisfies the flatness equations given in Lemma \ref{PMatrixEqns}.

Flatness equations in Lemma \ref{PMatrixEqns} involve functions $C_i$. After certain changes of variables flatness equations in Lemma \ref{PMatrixEqns} can be rewritten as modified flatness equations (\ref{eqn:modflateqn}). This form of equations turns out to be more convenient to work with and involves functions $A_i$ defined in Section \ref{sec:desc_ring}.

For our goal of analyzing properties of Gromov-Witten potentials, we need to establish properties of those functions $C_i$ and $K_i$ arising in genus $0$ Gromov-Witten theory of $\CnZn$ and $A_i$ in modified flatness equation (\ref{eqn:modflateqn}). The rest of this Section is devoted to carry out this rather technical step. In particular, we construct certain differential rings and obtain polynomial relations among the generators of these differential rings, see Proposition \ref{pro:CDA_Simplification}, which lead to the ring of functions mentioned in the {\bf Main Theorem} in Section \ref{sec:intro}.

\subsection{Preparations}
The main purpose here is to define and study the following series\footnote{It is straightforward to see that these series are convergent in a neighborhood of $x=0$.} 
\begin{equation*}
X_{k,l}=\frac{D^lC_k}{C_k}\in \mathbb{C}[\![x]\!],
\end{equation*}
for all $k,l\geq{0}$. We denote $X_{k,1}$ just by $X_k$. Also, note that $X_0=0$ since $C_0=1$.

The next two lemmas concern properties of $X_{k,l}$.
\begin{lem}\label{DplusZk}
We have
\begin{equation*}
X_{k,l}=\left(D+X_k\right)^{l-1}X_k
\end{equation*}
for all $k\geq{0}$ and $l\geq{1}$. In particular, $X_{k,l}$ is a polynomial in $\{X_k, DX_k,\ldots,D^{l-1}X_k\}$ and $D^{l-1}X_k$ is a polynomial in $\{X_{k,1},\ldots,X_{k,l}\}$.
\end{lem}

\begin{proof}
The first part follows by induction on $l$. The case $l=1$ is clear. The inductive step is as follows:
\begin{equation*}
DX_{k,l-1}
=D\left(\frac{D^{l-1}C_k}{C_k}\right)
=\frac{D^lC_k}{C_k}-\frac{D^{l-1}C_k}{C_k}\frac{DC_k}{C_k}
=X_{k,l}-X_{k,l-1}X_k
\end{equation*}
which is equivalent to
\begin{equation*}
X_{k,l}=\left(D+X_k\right)X_{k,l-1}.
\end{equation*}
The polynomiality of $X_{k,l}$ directly follows, and the polynomiality of $D^{l-1}X_k$ follows from a basic elimination.
\end{proof}

\begin{lem}\label{DLLemma}
We have
\begin{align}
\frac{DL}{L}&=1+(-1)^n\frac{L^n}{n^n}=\frac{L^n}{x^n},\label{eq:DLLLemma1}\\
\frac{DK_i}{K_i}&=\sum_{r=0}^iX_r,\label{eq:DLLLemma2}\\
n\frac{DL}{L}&=\sum_{r=0}^nX_r,\label{eq:DLLLemma3}
\end{align}
for $0\leq{i}\leq{n}$.
\end{lem}

\begin{proof}
By the logarithmic differentiation, we obtain
\begin{equation*}
\frac{DL}{L}=x\frac{\frac{dL}{dx}}{L}=1+\frac{(-1)^n\left(\frac{x}{n}\right)^{n}}{1-(-1)^n\left(\frac{x}{n}\right)^n}.
\end{equation*}
This implies
\begin{equation*}
\frac{DL}{L}
=1+(-1)^n\frac{x^n}{n^n}\left(1-(-1)^n\left(\frac{x}{n}\right)^n\right)^{-1}
=1+(-1)^n\frac{L^n}{n^n},
\end{equation*}
and
\begin{equation*}
\frac{DL}{L}
=\frac{1}{1-(-1)^n\left(\frac{x}{n}\right)^n}
=\left(1-(-1)^n\left(\frac{x}{n}\right)^n\right)^{-1}
=\frac{L^n}{x^n}.
\end{equation*}
The second equation follows directly from the definitions of $K_i$ and $X_r$. The last equation follows from the second equation and part (1) of Corollary \ref{Kfunctions}.
\end{proof}

Next, we consider two sets of functions $Z_{m,k}, B_{k,p}$ obtained from $C_i, i\geq 0$. 

For any $m\geq{1}$, define the following series in $x$
\begin{equation*}
Z_{m,k}=
\begin{cases}
D^{-1}C_{k+1}...D^{-1}C_{m}\quad&\text{if}\quad 0\leq k \leq m-1,\\
1\quad&\text{if}\quad k=m,\\
0\quad&\text{if}\quad k>m.
\end{cases}
\end{equation*}
From the definition of $Z_{m,k}$, we easily see that
\begin{equation}\label{eq:DZmkCk}
DZ_{m,k}=C_{k+1}Z_{m,k+1}
\end{equation}
for all $k\geq{0}$. We also recall that, by equation (\ref{altdefCis}), for $m\geq{1}$, $I_m=D^{-1}C_{1}...D^{-1}C_{m}$ which is just $Z_{m,0}$. Now for $k\geq 1$ define the following series in $x$:

\begin{equation*}
B_{k,p}=
\begin{cases}
D^{k-1}C_1\quad&\text{if}\quad p=1,\\
\sum\limits_{k_2=p-1}^{k_1-1}...\sum\limits_{k_p=1}^{k_{p-1}-1}\left(\prod\limits_{i=1}^{p-1}\binom{k_{i}-1}{k_{i+1}}\right)\left(D^{k_1-1-k_2}C_1\right)...\left(D^{k_{p-1}-1-k_p}C_{p-1}\right)\left(D^{k_p-1}C_p\right)\quad&\text{if}\quad 2\leq{p}\leq{k},\\
0\quad&\text{if}\quad p>k
\end{cases}
\end{equation*}
where $k_1=k$.

The next two lemmas are equations involving functions $Z_{m,k}, B_{k,p}$. These equations will be used in studying relations in rings of functions in Section \ref{sec:desc_ring} below.

\begin{lem}\label{lem:DkImCommutator}
For all $k,m\geq{1}$, we have
\begin{equation*}
D^kI_m=\sum\limits_{p=1}^kB_{k,p}Z_{m,p}.
\end{equation*}
\end{lem}

\begin{proof}
Let $\mathfrak{T}_1$ and $\mathfrak{T}_2$ be two linear operators acting on $\mathbb{C}[\![x]\!]$, and let  $$[\mathfrak{T}_1,\mathfrak{T}_2]=\mathfrak{T}_1\mathfrak{T}_2-\mathfrak{T}_2\mathfrak{T}_1$$ be their commutator. Set $\ad_{\mathfrak{T}_1}^0(\mathfrak{T}_2)=\mathfrak{T}_2$, and let $\ad_{\mathfrak{T}_1}^j(\mathfrak{T}_2)$ be the generalization of the commutator for $j\geq 1$ inductively defined by
$$\ad_{\mathfrak{T}_1}^j(\mathfrak{T}_2)=[\mathfrak{T}_1,\ad_{\mathfrak{T}_1}^{j-1}(\mathfrak{T}_2)].$$
By induction, it can be shown that the commutator of the operator $D$ and multiplication by a series $A$ is given by
\begin{equation*}
\ad_{D}^j(A)=(D^jA)
\end{equation*}
i.e., multiplication by the series $(D^jA)$, and the multiplication by $A$ followed by the operator $D^i$ is given by
\begin{equation}\label{eq:HigherDiA}
D^iA=\sum_{j=0}^i\binom{i}{j}\ad_{D}^j(A)D^{i-j}.
\end{equation}
Using the fact that for $m\geq{1}$, $I_m=D^{-1}C_{1}...D^{-1}C_{m}=Z_{m,0}$ together with equations (\ref{eq:DZmkCk}) and (\ref{eq:HigherDiA}), we inductively complete the proof.
\end{proof}

\begin{lem}
For all $1\leq{m}\leq{n-1}$, we have
\begin{equation}\label{PFforGraded}
B_{n,m}+\frac{DL}{L}\sum\limits_{k=m }^{n-1}s_{n,k}B_{k,m}=B_{n,m}+\frac{DL}{L}\sum\limits_{k=1}^{n-1}s_{n,k}B_{k,m}=0.
\end{equation}
\end{lem}

\begin{proof}
The first equality follows from the definition of $B_{k,m}$. For the second equality, we use induction on $m$. For $m=1$, it follows from $B_{k,1}=D^{k-1}C_1=D^kI_1$ and equation (\ref{eq:PFforIks}). The following completes the inductive step: 
\begin{align*}
0=&D^nI_m+\frac{DL}{L}\sum\limits_{k=1}^{n-1}s_{n,k}D^kI_m\quad\quad\text{by equation (\ref{eq:PFforIks})}\\
=&\sum\limits_{p=1}^nB_{n,p}Z_{m,p}+\frac{DL}{L}\sum\limits_{k=1}^{n-1}s_{n,k}\sum\limits_{p=1}^kB_{k,p}Z_{m,p}\quad\text{by Lemma \ref{lem:DkImCommutator}}\\
=&\sum\limits_{p=1}^mB_{n,p}Z_{m,p}+\frac{DL}{L}\sum\limits_{k=1}^{n-1}s_{n,k}\sum\limits_{p=1}^mB_{k,p}Z_{m,p}\quad\text{by definitions of }B_{k,p}\text{ and }Z_{m,p}\\
=&\sum\limits_{p=1}^m\left(B_{n,p}+\frac{DL}{L}\sum\limits_{k=1}^{n-1} s_{n,k}B_{k,p}\right)Z_{m,p}\\
=&B_{n,m}+\frac{DL}{L}\sum\limits_{k=1}^{n-1}s_{n,k}B_{k,m}+\sum\limits_{p=1}^{m-1}\underbrace{\left(B_{n,p}+\frac{DL}{L}\sum\limits_{k=1}^{n-1}s_{n,k}B_{k,p}\right)}_{=0\text{ by inductive hypothesis.}}Z_{m,p}.
\end{align*}
\end{proof}

\subsection{Descriptions of the rings}\label{sec:desc_ring}
Set 
\begin{equation*}
\mathbb{C}[L^{\pm 1}][\mathcal{DX}]:=\mathbb{C}[L^{\pm 1}][X_1,...,X_{n-1},DX_1,...,DX_{n-1},D^2X_1,...,D^2X_{n-1},...].
\end{equation*}
We simplify the description of $\mathbb{C}[L^{\pm 1}][\mathcal{DX}]$ by removing unnecessary generators. Set \begin{equation*}
\mathfrak{X}:=\{X_1,...,D^{n-3}X_1\}\cup,\ldots\cup\{X_i,...,D^{n-2-i}X_i\}\cup\ldots\cup\{X_{n-2}\}=\{D^jX_i\}_{1\leq i\leq n-2, 0\leq j\leq n-2-i}.
\end{equation*}

\begin{lem}\label{lem:DGradedRingX}
 $\mathbb{C}[L^{\pm 1}][\mathcal{DX}]$ is a quotient of the ring $\mathbb{C}[L^{\pm 1}][\mathfrak{X}]$.
\end{lem}

\begin{proof}
Now, for any $1\leq{p}\leq{k-1}$, define
\begin{align*}
\mathcal{Z}_{p,k}=&\{X_{1,1},...,X_{1,k-p},...,X_{p,1},...,X_{p,k-p}\}\quad\text{and}\quad \mathcal{S}_{p,k}=\mathcal{Z}_{p,k}\setminus\{X_{p,k-p}\},\\
\widetilde{\mathcal{Z}}_{p,k}=&\{X_1,...,D^{k-p-1}X_1,...,X_p,...,D^{k-p-1}X_p\}\quad\text{and}\quad\widetilde{\mathcal{S}}_{p,k}=\widetilde{\mathcal{Z}}_{p,k}\setminus\{D^{k-p-1}X_p\}.
\end{align*}
For each of these sets, and for a fixed $p$ we have
\begin{equation}\label{eq:setinclusionsforZpkSpk}
\mathcal{S}_{p,k}\subseteq\mathcal{Z}_{p,k}\subseteq\mathcal{S}_{p,k+1}\subseteq\mathcal{Z}_{p,k+1},\quad\text{and}\quad\widetilde{\mathcal{S}}_{p,k}\subseteq\widetilde{\mathcal{Z}}_{p,k}\subseteq\widetilde{\mathcal{S}}_{p,k+1}\subseteq\widetilde{\mathcal{Z}}_{p,k+1}.
\end{equation}
Note that for any $k\geq{1}$
\begin{equation*}
\frac{B_{k,p}}{K_p}=
\begin{cases}
X_{1,k-1}\quad&\text{if}\quad p=1,\\
\sum\limits_{k_2=p-1}^{k_1-1}...\sum\limits_{k_p=1}^{k_{p-1}-1}\left(\prod\limits_{i=1}^{p-1}\binom{k_{i}-1}{k_{i+1}}\right)X_{1,k_1-k_2-1}...X_{p-1,k_{p-1}-k_p-1}X_{p,k_p-1}\quad&\text{if}\quad 2\leq{p}\leq{k},\\
0\quad&\text{if}\quad p>k
\end{cases}
\end{equation*}
where $k_1=k$.
It follows that for any $1\leq{p}\leq{k-1}$, we have
\begin{equation}\label{eq:Zpkp}
\frac{B_{k,p}}{K_p}=X_{p,k-p}+\widetilde{B}_{k,p}
\end{equation}
where $\widetilde{B}_{k,p}$ is a polynomial in elements of $\mathcal{S}_{p,k}$. Then, dividing both sides of equation (\ref{PFforGraded}) by $K_m$ for any $1\leq{m}\leq{n-1}$, we obtain
\begin{equation}\label{eq:Xmnminusm}
\begin{split}
0=&\frac{B_{n,m}}{K_m}+\frac{DL}{L}\sum\limits_{k=m }^{n-1}s_{n,k}\frac{B_{k,m}}{K_m}\\
=&X_{m,n-m}+\widetilde{B}_{n,m}+\frac{DL}{L}\underbrace{\sum\limits_{k=m }^{n-1}s_{n,k}\frac{B_{k,m}}{K_m}}_{(\star)}.
\end{split}
\end{equation}
By set inclusions (\ref{eq:setinclusionsforZpkSpk}) and equation (\ref{eq:Zpkp}), it follows that $(\star)$ is a polynomial in elements of $\mathcal{Z}_{m,n-1}$. Since we know $\widetilde{B}_{n,m}$ is a polynomial in element of $\mathcal{S}_{m,n}$ and $\mathcal{Z}_{m,n-1}\subseteq\mathcal{S}_{m,n}$, it follows that $X_{m,n-m}$ is a polynomial in elements of $\mathcal{S}_{m,n}\cup\{{L^{\pm1}}\}$ by equation (\ref{eq:Xmnminusm}) and equation (\ref{eq:DLLLemma1}). This implies that $D^{n-m-1}X_m$ is a polynomial in elements of $\widetilde{\mathcal{S}}_{m,n}\cup\{L^{\pm1}\}$ by Lemma \ref{DplusZk}. This completes the proof.
\end{proof}

Now we present another description\footnote{By Lemma \ref{DLLemma}, $\frac{DL}{L}\in \mathbb{C}[L^\pm]$.} of the ring $\mathbb{C}[L^{\pm 1}][\mathcal{DX}]$ using a different set of generators $A_i$, $0\leq i\leq n$ defined by
\begin{equation*}
A_i=\frac{1}{L}\left(i\frac{DL}{L}-\sum_{r=0}^{i} X_{r}\right).
\end{equation*}
Set 
\begin{equation*}
\mathbb{C}[L^{\pm 1}][\mathcal{DA}]:=\mathbb{C}[L^{\pm 1}][A_1,...,A_{n-1},DA_1,...,DA_{n-1},D^2A_1,...,D^2A_{n-1},...],
\end{equation*}
and
\begin{equation*}
\mathfrak{A}:=\{A_1,...,D^{n-3}A_1\}\cup,\ldots\cup\{A_i,...,D^{n-2-i}A_i\}\cup\ldots\cup\{A_{n-2}\}=\{D^jA_i\}_{1\leq i\leq n-2, 0\leq j\leq n-2-i}.
\end{equation*}
The following is immediate from Lemma \ref{lem:DGradedRingX}.
\begin{cor}\label{lem:DGradedRingA}
$\mathbb{C}[L^{\pm 1}][\mathcal{DA}]$ is a quotient of the ring $\mathbb{C}[L^{\pm 1}][\mathfrak{A}]$.
\end{cor}
In what follows we further simplify the ring $\mathbb{C}[L^{\pm 1}][\mathfrak{A}]$. We begin with some basic properties of $A_i, 0\leq i\leq n$.

\begin{lem}\label{lem:PropertiesofAis}
For the series $A_i$, we have the following
\begin{enumerate}
    \item $A_i=-A_{n-i}$ for all $0\leq{i}\leq{n}$,
    \item $A_0=A_n=0$, and $A_{\frac{n}{2}}=0$ if $n$ is even,
    \item $\sum_{i=0}^{n}A_i=0$.
\end{enumerate}
\end{lem}
\begin{proof}
By Lemma \ref{propertiesofCfunctions}, we have $C_i=C_{n+1-i}$ for all $1\leq{i}\leq{n}$. Hence,  $X_i=X_{n+1-i}$ for all $1\leq{i}\leq{n}$. This gives the following reformulation of equation (\ref{eq:DLLLemma3}) :
\begin{equation*}
\sum_{r=0}^{i} X_{r}-i\frac{D L}{L}=(n-i) \frac{D L}{L}-\left(\sum_{r=0}^{n-i} X_{r}\right)\quad\text{for all}\quad 0\leq{i}\leq{n}.
\end{equation*}
This proves the first part of the lemma. The other two parts follow immediately.
\end{proof}

Next, we obtain certain relations involving derivatives of $A_i$. For this purpose we need to take a digression to flatness equations.

For $0\leq{i,j}\leq{n-1}$ and $k\geq{0}$, define
\begin{equation*}
\widetilde{P}_{i,j}^k=\frac{L^i}{K_i}P_{i,j}^k\zeta^{(k+i)j}.
\end{equation*}

\begin{lem}\label{PtildeMatrixEqns} 
For $0\leq i \leq n-1$, we have
\begin{equation*}
\widetilde{P}_{\mathrm{Ion}(i)-1,j}^k=\widetilde{P}_{i,j}^k+\frac{1}{L}D\widetilde{P}_{i,j}^{k-1}+\frac{1}{L}\left(\sum_{r=0}^i X_r-i\frac{DL}{L}\right)\widetilde{P}_{i,j}^{k-1}.
\end{equation*}
\end{lem}

\begin{proof}
This is just a reformulation of Lemma \ref{PMatrixEqns}. The LHS of Lemma \ref{PMatrixEqns} becomes
\begin{equation*}
DP_{i,j}^{k-1}
=\left(\frac{DK_i}{L^i}-i\frac{K_i}{L^i}\frac{DL}{L}\right)\widetilde{P}_{i,j}^{k-1}\zeta^{-(k-1+i)j}+\frac{K_i}{L^i}D\widetilde{P}_{i,j}^{k-1}\zeta^{-(k-1+i)j}
\end{equation*}
by the definition of $K_i$ and, by the use of Corollary \ref{Kfunctions}, we see that the RHS of Lemma \ref{PMatrixEqns} becomes
\begin{align*}
C_{\mathrm{Ion}(i)}P_{\mathrm{Ion}(i)-1,j}^k-P_{i,j}^kL\zeta^j
=&C_{\mathrm{Ion}(i)}\frac{K_{{\mathrm{Ion}(i)}-1}}{L^{{\mathrm{Ion}(i)}-1}}\widetilde{P}_{{\mathrm{Ion}(i)}-1,j}^k\zeta^{-(k-1+{\mathrm{Ion}(i)})j}-\frac{K_i}{L^{i-1}}\widetilde{P}_{i,j}^k\zeta^{-(k-1+i)j}\\
=&\frac{K_{{\mathrm{Ion}(i)}}}{L^{{\mathrm{Ion}(i)}-1}}\widetilde{P}_{{\mathrm{Ion}(i)}-1,j}^k\zeta^{-(k-1+{\mathrm{Ion}(i)})j}-\frac{K_i}{L^{i-1}}\widetilde{P}_{i,j}^k\zeta^{-(k-1+i)j}\\
=&\frac{K_i}{L^{i-1}}\widetilde{P}_{{\mathrm{Ion}(i)}-1,j}^k\zeta^{-(k-1+i)j}-\frac{K_i}{L^{i-1}}\widetilde{P}_{i,j}^k\zeta^{-(k-1+i)j}.
\end{align*}
Putting these together, using  one more time the definition of $K_i$, Corollary \ref{Kfunctions} and cancelling out some common factors we obtain
\begin{equation*}
\left(\frac{DK_i}{K_i}-i\frac{DL}{L}\right)\widetilde{P}_{i,j}^{k-1}+D\widetilde{P}_{i,j}^{k-1}=
\widetilde{P}_{{\mathrm{Ion}(i)}-1,j}^kL-\widetilde{P}_{i,j}^kL.
\end{equation*}
The rest follows from equation (\ref{eq:DLLLemma2}).
\end{proof}

Lemma \ref{PtildeMatrixEqns}
\iffalse
we have
\begin{equation*}
\widetilde{P}_{\mathrm{Ion}(i)-1,j}^k=\widetilde{P}_{i,j}^k+\frac{1}{L}D\widetilde{P}_{i,j}^{k-1}+\frac{1}{L}\left(\sum_{r=0}^i X_r-i\frac{DL}{L}\right)\widetilde{P}_{i,j}^{k-1},
\end{equation*}
which 
\fi
is equivalent to the following \textbf{\textit{modified flatness equations}} for $[\mathbb{C}^n/\mathbb{Z}_n]$:
\begin{equation}\label{eqn:modflateqn}
\widetilde{P}_{\mathrm{Ion}(i)-1,j}^k=\widetilde{P}_{i,j}^k+\frac{1}{L}D\widetilde{P}_{i,j}^{k-1}+A_{n-i}\widetilde{P}_{i,j}^{k-1}.
\end{equation}
We further analyze (\ref{eqn:modflateqn}). Let $k=0$. Then $\widetilde{P}_{\mathrm{Ion}(i)-1,j}^0=\widetilde{P}_{i,j}^0$ for all $0\leq{i}\leq{n-1}$. This means $\widetilde{P}_{i,j}^0=\widetilde{P}_{0,j}^0$ for all $0\leq{i}\leq{n-1}$. Now, let $k=1$. Then, we have
\begin{equation*}
\underbrace{\sum_{i=0}^{n-1}\widetilde{P}_{\mathrm{Ion}(i)-1,j}^1}_{(\star)}=\underbrace{\sum_{i=0}^{n-1}\widetilde{P}_{i,j}^1}_{(\star\star)}+\frac{1}{L}D\sum_{i=0}^{n-1}\widetilde{P}_{i,j}^0+\underbrace{\sum_{i=0}^{n-1}A_{n-i}\widetilde{P}_{0,j}^0}_{(\star\star\star)}.
\end{equation*}
Clearly, $(\star)$ and $(\star\star)$ are the same and $(\star\star\star)$ is zero. Since $\widetilde{P}_{i,j}^0=\widetilde{P}_{0,j}^0$, the above equation is 
\begin{equation*}
\frac{n}{L}D\widetilde{P}_{0,j}^0=0.
\end{equation*}
Hence, $\widetilde{P}_{i,j}^0=\widetilde{P}_{0,j}^0$ is a constant whose value depends on the initial conditions. Now, the equations with $k=1$ yield
\begin{equation*}
\begin{split}
\widetilde{P}_{n-1,j}^1&=\widetilde{P}_{0,j}^1+A_{n}\widetilde{P}_{0,j}^{0}\\
\widetilde{P}_{n-2,j}^1&=\widetilde{P}_{n-1,j}^1+A_{1}\widetilde{P}_{0,j}^0\\
&\,\,\,\vdots\\
\widetilde{P}_{n-i,j}^1&=\widetilde{P}_{n-i+1,j}^1+A_{i-1}\widetilde{P}_{0,j}^0.
\end{split}
\end{equation*}
Adding these equations side by side and using Lemma \ref{lem:PropertiesofAis} yield
\begin{equation*}
\widetilde{P}_{n-i,j}^{1}=\widetilde{P}_{0, j}^{1}+\sum_{r=0}^{i-1} A_{r} \widetilde{P}_{0, j}^{0}\quad\text{for}\quad 1\leq{i}\leq{n}.
\end{equation*}

Now, let $k=2$ in equation (\ref{eqn:modflateqn}) and plug the above equation into it, we find 
\begin{equation*}
\begin{split}
\widetilde{P}_{\mathrm{Ion}(i)-1,j}^2
=&
\widetilde{P}_{i,j}^2+\frac{1}{L}D\widetilde{P}_{i,j}^1+A_{n-i}\widetilde{P}_{i,j}^1\\
=&
\widetilde{P}_{i,j}^2+\frac{1}{L}D\left(\widetilde{P}_{0, j}^{1}+\sum_{r=0}^{n-i-1} A_{r} \widetilde{P}_{0, j}^{0}\right)+A_{n-i}\left(\widetilde{P}_{0, j}^{1}+\sum_{r=0}^{n-i-1} A_{r} \widetilde{P}_{0, j}^{0}\right)\\
=&
\widetilde{P}_{i,j}^2+\frac{1}{L}D\widetilde{P}_{0 ,j}^{1}+\frac{1}{L}\sum_{r=0}^{n-i-1} \left(DA_{r}\right) \widetilde{P}_{0, j}^{0}+A_{n-i}\widetilde{P}_{0, j}^{1}+\sum_{r=0}^{n-i-1}A_{n-i}A_{r} \widetilde{P}_{0, j}^{0}.
\end{split}
\end{equation*}
Summing this equality over $0\leq{i}\leq{n-1}$, cancelling out $\sum_{i=0}^{n-1}\widetilde{P}_{\mathrm{Ion}(i)-1,j}^2=\sum_{i=0}^{n-1}\widetilde{P}_{i,j}^2$, and noting that $\sum_{i=0}^{n-1}A_{n-i}\widetilde{P}_{0,j}^1=0$, we obtain
\begin{equation}\label{eq:nLDP1}
\frac{n}{L}D\widetilde{P}_{0,j}^1+\frac{1}{L}\sum_{i=0}^{n-1}\sum_{r=0}^{n-i-1}\left(DA_r\right)\widetilde{P}_{0,j}^0+\sum_{i=0}^{n-1}\sum_{r=0}^{n-i-1}A_{n-i}A_r\widetilde{P}_{0,j}^0=0.
\end{equation}

Set $k=1$ in Corollary \ref{cor:LPijkidentity}, we obtain
\begin{equation*}
\mathds{L}_{j,1}(P^{1}_{0,j})+\frac{1}{L_j}\mathds{L}_{j,2}(P^{0}_{0,j})=0
\end{equation*}
which reads as
\begin{equation*}
nDP^{1}_{0,j}+\frac{1}{L_j}\binom{n+1}{4}(Y^2-Y)P^{0}_{0,j}-\frac{1}{L_j}\binom{n}{2}YDP^{0}_{0,j}+\frac{1}{L_j}\binom{n}{2}D^2P^{0}_{0,j}=0.
\end{equation*}
Since ${P}_{0,j}^0=\widetilde{P}_{0,j}^0$ is constant and ${P}_{0,j}^1=\zeta^{-j}\widetilde{P}_{0,j}^1$, the equation becomes
\begin{equation*}
nD\widetilde{P}_{0,j}^1+\frac{1}{L}\binom{n+1}{4}Y(Y-1)\widetilde{P}^{0}_{0,j}=0.
\end{equation*}
By the definition of $Y$ in (\ref{eqn:XandY}), we obtain
\begin{equation}\label{eq:DPTilde1}
\begin{split}
D\widetilde{P}_{0,j}^1=&\frac{1}{n}\frac{1}{L}\binom{n+1}{4}Y(1-Y)\widetilde{P}_{0,j}^0\\
=&\frac{(-1)^{n-1}}{n}\binom{n+1}{4}\left(1+(-1)^n\frac{L^n}{n^n}\right)\frac{L^{n-1}}{n^n}\widetilde{P}_{0,j}^0\in\mathbb{C}[L^{\pm{1}}].
\end{split}
\end{equation}
Define $f_n(L)\in\mathbb{C}[L^{\pm{1}}]$ to be the right hand side of above equation without $\widetilde{P}_{0,j}^0$:
\begin{equation*}
f_n(L)=\frac{(-1)^{n-1}}{n}\binom{n+1}{4}\left(1+(-1)^n\frac{L^n}{n^n}\right)\frac{L^{n-1}}{n^n},
\end{equation*}
so $D\widetilde{P}_{0,j}^1=f_n(L)\widetilde{P}_{0,j}^0$.

\begin{lem}\label{lem:more_on_Ais}
For any $n\geq{3}$, we have
\begin{equation*}
\sum_{i=0}^{n-1}\sum_{r=0}^{n-i-1}DA_r=\sum_{r=1}^{\lfloor\frac{n-1}{2}\rfloor}(n-2r)DA_r\quad\text{and}\quad \sum_{i=0}^{n-1}\sum_{r=0}^{n-i-1}A_{n-i}A_r=-\sum_{r=1}^{\lfloor\frac{n-1}{2}\rfloor}A_r^2.
\end{equation*}
\end{lem}
\begin{proof}
This follows from the fact that $A_i=-A_{n-i}$.
\end{proof}
By equations (\ref{eq:nLDP1}), (\ref{eq:DPTilde1}), and Lemma \ref{lem:more_on_Ais}  we obtain the following
\begin{lem}\label{lem:Equations forDAl}
For any $n\geq{3}$, we have
\begin{equation*}
\frac{n}{L}f_n(L)+\frac{1}{L}\sum_{r=1}^{\lfloor\frac{n-1}{2}\rfloor}(n-2r)\left(DA_r\right)-\sum_{r=1}^{\lfloor\frac{n-1}{2}\rfloor}A_r^2=0.
\end{equation*}
Equivalently, dividing into even and odd cases, we have
\begin{equation*}
\begin{split}
2DA_{s-1}
&=\sum_{r=1}^{s-1}LA_r^2-\sum_{r=1}^{s-2}(n-2r)DA_r-2sf_{2s}(L)\quad\text{if}\quad n=2s\geq{4},\\
DA_s
&=\sum_{r=1}^{s}LA_r^2-\sum_{r=1}^{s-1}(n-2r)DA_r-({2s+1})f_{2s+1}(L)\quad\text{if}\quad n=2s+1\geq{3}.
\end{split}
\end{equation*}
\end{lem}

Equations in Lemma \ref{lem:Equations forDAl} are generalizations\footnote{For $n=5$, this generalization is explained in more detail by matching the functions in \cite{lho} with ours.} of equation (9) in \cite[Section 3]{lho} and second equation in \cite[Lemma 9]{lho-p2}.

The above relations allow us to further simplify the ring $\mathbb{C}[L^{\pm 1}][\mathcal{DA}]$. Let $n\geq 3$ be an odd number with $n=2s+1$, define
\begin{equation*}
\mathfrak{S}_{\text{odd}}=\{A_1,\ldots,D^{n-3}A_1\}\cup\cdots\cup\{A_{s-1},\ldots,D^{n-s+1}A_{s-1}\}\cup\{A_s\}.
\end{equation*}
Similarly, let $n\geq 4$ be an even number with $n=2s$, define
\begin{equation*}
\mathfrak{S}_{\text{even}}=\{A_1,\ldots,D^{n-3}A_1\}\cup\cdots\cup\{A_{s-2},\ldots,D^{n-s}A_{s-2}\}\cup\{A_{s-1}\}.
\end{equation*}
In either case, we denote both $\mathfrak{S}_{\text{odd}}$, and $\mathfrak{S}_{\text{even}}$ as $\mathfrak{S}_n$.

\begin{prop}\label{pro:CDA_Simplification}
$\mathbb{C}[L^{\pm 1}][\mathcal{DA}]$ is a quotient of the ring $\mathbb{C}[L^{\pm 1}][\mathfrak{S}_n]$.
\end{prop}
\begin{proof}
This follows easily from Lemmas \ref{lem:DGradedRingX}, \ref{lem:PropertiesofAis}, and \ref{lem:Equations forDAl}.
\end{proof}
As in \cite{lho-p} and \cite{lho-p2}, we do not know if there are any further polynomial relations among the elements of the differential ring $\mathbb{C}[L^{\pm 1}][\mathcal{DA}]$.

\section{Holomorphic anomaly equations}\label{sec:HAE}

\subsection{More on flatness equation} In this subsection, we will give a description of a canonical lift of each $\widetilde{P}_{i,j}^k$ to the free algebra $\mathbb{C}[L^{\pm 1}][\mathfrak{S}_n]$. Then, considering $\widetilde{P}_{i,j}^k$ as the elements of $\mathbb{C}[L^{\pm 1}][\mathfrak{S}_n]$, we will investigate the derivatives of $\widetilde{P}_{i,j}^k$ with respect to the generators of  $\mathbb{C}[L^{\pm 1}][\mathfrak{S}_n]$ which also appear in the holomorphic anomaly equations in the main theorem of the paper.

The modified flatness equations (\ref{eqn:modflateqn}), Lemma \ref{DplusZk}, and Corollary \ref{cor:P0jk_is_in_CL} imply that $\widetilde{P}_{i,j}^k\in\mathbb{C}[L^{\pm 1}][\mathcal{DA}].$ Through  Lemmas \ref{lem:DGradedRingX}, \ref{lem:PropertiesofAis},  \ref{lem:Equations forDAl}, and the modified flatness equations (\ref{eqn:modflateqn}), we have a canonical lift of each $\widetilde{P}_{i,j}^k$ to the free algebra $\mathbb{C}[L^{\pm 1}][\mathfrak{S}_n]$ via the following order:
\begin{equation}\label{eq:ModifiedFlatnessLift}
\begin{split}
\widetilde{P}_{n-1,j}^k&=\widetilde{P}_{0,j}^k+\frac{1}{L}D\widetilde{P}_{0,j}^{k-1}\in\mathbb{C}[L^{\pm 1}]\subseteq \mathbb{C}[L^{\pm 1}][\mathfrak{S}_n]\\
\widetilde{P}_{n-2,j}^k&=\widetilde{P}_{n-1,j}^k+\frac{1}{L}D\widetilde{P}_{n-1,j}^{k-1}+A_1\widetilde{P}_{n-1,j}^{k-1}\in\mathbb{C}[L^{\pm 1}][A_1]\subseteq \mathbb{C}[L^{\pm 1}][\mathfrak{S}_n]\\
\vdots \,\,\,\,\,\, &=\,\,\,\,\,\,\,\,\,\,\,\,\,\,\,\,\,\,\,\,\,\,\,\,\,\,\,\,\,\,\vdots
\end{split}
\end{equation}
More precisely, we start with $\widetilde{P}_{0,j}^k\in \mathbb{C}[L^{\pm 1}]$ and use equation (\ref{eqn:modflateqn}) for $i=n, n-1,..., 2$ in this descending order to inductively construct lifts of $\widetilde{P}_{i,j}^k$ for $i=n-1, n-2,...,1$ in this descending order. In this process, unnecessary $A_i$'s are eliminated using Lemmas \ref{lem:PropertiesofAis} and \ref{lem:Equations forDAl}, and orders of derivatives are bounded above using Lemma \ref{lem:DGradedRingX}.

In the rest of this subsection, we consider this lift and denote it also as $$\widetilde{P}_{i,j}^k\in\mathbb{C}[L^{\pm 1}][\mathfrak{S}_n].$$

\begin{lem}[Odd case]\label{lem:oddderivativeflatness}
Let $n\geq{3}$ be an odd number with $n=2s+1$. We have the following identity
\begin{equation*}
\frac{\partial\widetilde{P}_{i,j}^k}{\partial{A_s}}=\delta_{i,s}\widetilde{P}_{{s+1},j}^{k-1}.
\end{equation*}
\end{lem}

\begin{proof}
From the modified flatness equations (\ref{eqn:modflateqn}), and the lifting procedure (\ref{eq:ModifiedFlatnessLift}) we have
\begin{equation}\label{eq:oddderivativefirstpart}
\frac{\partial\widetilde{P}_{i,j}^k}{\partial{A_s}}=0
\end{equation}
for $i$ in the range $\{0\} \cup \{s+1,...,n-1\}$, since $\widetilde{P}_{i,j}^k$ does not contain $A_s$ term for this range of $i$'s.

Now observe the following two equations
\begin{align}
\widetilde{P}_{s,j}^k=&\widetilde{P}_{{s+1},j}^k+\frac{1}{L}D\widetilde{P}_{{s+1},j}^{k-1}+A_s\widetilde{P}_{{s+1},j}^{k-1}\,,\label{eq:oddmiddleequations1}\\
\widetilde{P}_{s-1,j}^k=&\widetilde{P}_{{s},j}^k+\frac{1}{L}D\widetilde{P}_{{s},j}^{k-1}-A_s\widetilde{P}_{{s},j}^{k-1}.\label{eq:oddmiddleequations2}
\end{align}
These are first two rows in modified flatness equations where we see $A_s$. From the first equation we see that
\begin{equation}\label{eq:odd_case_non_zero_derivative}
\frac{\partial\widetilde{P}_{s,j}^k}{\partial{A_s}}=\widetilde{P}_{{s+1},j}^{k-1}.
\end{equation}
Then, by equations (\ref{eq:oddderivativefirstpart}) and (\ref{eq:odd_case_non_zero_derivative}) we see that
\begin{equation}\label{eq:KH_for_odd}
\widetilde{P}_{s,j}^k=KA_s+H
\end{equation}
where $K=\widetilde{P}_{{s+1},j}^{k-1}$ and $H$ are constants with respect to $A_s$. Note that Lemma \ref{lem:Equations forDAl} gives
\begin{equation}\label{eq:partial_of_DAl}
\frac{\partial \left(DA_s\right)}{\partial A_s}=2LA_s.
\end{equation}

Now, by equations (\ref{eq:KH_for_odd}) and (\ref{eq:partial_of_DAl}) we observe the following
\begin{equation}\label{eq:Dplus2LAs}
\begin{split}
\frac{\partial(D{\widetilde{P}_{s,j}^k})}{\partial A_s}=&\frac{\partial}{\partial A_s}((DK)A_s+K(DA_s)+DH)\\
=&DK+K\frac{\partial \left(DA_s\right)}{\partial A_s}\\
=&D\widetilde{P}_{{s+1},j}^{k-1}+2LA_s\widetilde{P}_{{s+1},j}^{k-1}\,.
\end{split}
\end{equation}

Then, by equations (\ref{eq:oddmiddleequations2}) and (\ref{eq:Dplus2LAs}) we have
\begin{equation}
\begin{split}
\frac{\partial \widetilde{P}_{s-1,j}^k}{\partial A_s}
=&\frac{\partial \widetilde{P}_{{s},j}^k}{\partial A_s}+\frac{1}{L}\frac{\partial\left(D\widetilde{P}_{{s},j}^{k-1}\right)}{\partial A_s}-\widetilde{P}_{{s},j}^{k-1}-A_s\frac{\partial \widetilde{P}_{{s},j}^{k-1}}{\partial A_s}\\
=&\frac{\partial \widetilde{P}_{{s},j}^k}{\partial A_s}+\frac{1}{L}\left(D\widetilde{P}_{{s+1},j}^{k-2}+2LA_s\widetilde{P}_{{s+1},j}^{k-2}\right)-\widetilde{P}_{{s},j}^{k-1}-A_s\frac{\partial \widetilde{P}_{{s},j}^{k-1}}{\partial A_s},
\end{split}
\end{equation}
and equation (\ref{eq:oddmiddleequations1}) implies that
\begin{equation}\label{eq:Alvanishing}
\begin{split}
\frac{\partial \widetilde{P}_{s-1,j}^k}{\partial A_s}=&\widetilde{P}_{{s+1},j}^{k-1}+\frac{1}{L}D\widetilde{P}_{{s+1},j}^{k-2}+2A_s\widetilde{P}_{{s+1},j}^{k-2}-\widetilde{P}_{{s},j}^{k-1}-A_s\widetilde{P}_{{s+1},j}^{k-2}\\
=&\widetilde{P}_{{s+1},j}^{k-1}+\frac{1}{L}D\widetilde{P}_{{s+1},j}^{k-2}+A_s\widetilde{P}_{{s+1},j}^{k-2}-\widetilde{P}_{{s},j}^{k-1}\\
=&0.
\end{split}
\end{equation}

Equation (\ref{eq:Alvanishing}) shows that $\widetilde{P}_{s-1,j}^k$ does not depend on $A_s$. Since the lifting procedure is an inductive process, we see that $\widetilde{P}_{i,j}^k$ does not depend on $A_s$ for $1\leq i \leq s-1$ by the modified flatness equations (\ref{eqn:modflateqn}), and the lifting procedure (\ref{eq:ModifiedFlatnessLift}). Hence, we conlude  that
\begin{equation*}
\frac{\partial\widetilde{P}_{i,j}^k}{\partial{A_s}}=0
\end{equation*}
for $1\leq i \leq s-1$. Combining this equation with the equations (\ref{eq:oddderivativefirstpart}) and (\ref{eq:odd_case_non_zero_derivative}), we complete the proof.
\end{proof}

\begin{lem}[Even case]\label{lem:evenderivativeflatness}
Let $n\geq{4}$ be an even number with $n=2s$. We have the following identity
\begin{equation*}
\frac{\partial\widetilde{P}_{i,j}^k}{\partial{A_{s-1}}}=\delta_{i,s}\widetilde{P}_{{s+1},j}^{k-1}+\delta_{i,{s-1}}\widetilde{P}_{{s},j}^{k-1}.
\end{equation*}
\end{lem}

\begin{proof}
From the modified flatness equations (\ref{eqn:modflateqn}), and the lifting procedure (\ref{eq:ModifiedFlatnessLift}) we have
\begin{equation}\label{eq:evenderivativepart1}
\frac{\partial\widetilde{P}_{i,j}^k}{\partial{A_{s-1}}}=0
\end{equation}
for $i$ in the range $\{0\} \cup \{s+1,...,n-1\}$, since $\widetilde{P}_{i,j}^k$ does not contain $A_{s-1}$ term for this range of $i$'s.

Observe the following equations obtanined from the modified flatness equations (\ref{eqn:modflateqn}) and Lemma \ref{lem:PropertiesofAis}
\begin{equation}
\begin{split}
\widetilde{P}_{s,j}^k=&\widetilde{P}_{{s+1},j}^k+\frac{1}{L}D\widetilde{P}_{{s+1},j}^{k-1}+A_{s-1}\widetilde{P}_{{s+1},j}^{k-1}\\
\widetilde{P}_{s-1,j}^k=&\widetilde{P}_{{s},j}^k+\frac{1}{L}D\widetilde{P}_{{s},j}^{k-1}\\
\widetilde{P}_{s-2,j}^k=&\widetilde{P}_{{s-1},j}^k+\frac{1}{L}D\widetilde{P}_{{s-1},j}^{k-1}-A_{s-1}\widetilde{P}_{{s-1},j}^{k-1}.
\end{split}
\end{equation}
Two of these equations are first two rows in modified flatness equations where we see $A_{s-1}$. From the first equation we see that
\begin{equation}\label{eq:evenderivativepart2}
\frac{\partial\widetilde{P}_{s,j}^k}{\partial{A_{s-1}}}=\widetilde{P}_{{s+1},j}^{k-1}.
\end{equation}
Note that by Lemma \ref{lem:Equations forDAl}, we have
\begin{equation*}
\frac{\partial \left(DA_{s-1}\right)}{\partial A_{s-1}}=LA_{s-1}.
\end{equation*}
This implies 
\begin{equation}\label{eq:evenderivativepart3}
\begin{split}
\frac{\partial \widetilde{P}_{s-1,j}^k}{\partial A_{s-1}}=&\frac{\partial \widetilde{P}_{{s},j}^k}{\partial A_{s-1}}+\frac{1}{L}\frac{\partial\left(D\widetilde{P}_{{s},j}^{k-1}\right)}{\partial A_{s-1}}\\
=&\widetilde{P}_{{s+1},j}^{k-1}+\frac{1}{L}\left(LA_{s-1}\widetilde{P}_{s+1,j}^{k-2}+D\widetilde{P}_{s+1,j}^{k-2}\right)\\
=&\widetilde{P}_{{s},j}^{k-1}.
\end{split}
\end{equation}
Equations (\ref{eq:evenderivativepart1}), (\ref{eq:evenderivativepart2}), and (\ref{eq:evenderivativepart3}) simply show that $\widetilde{P}_{s-1,j}^k$ is of degree $2$ with respect to $A_{s-1}$. Let $\widetilde{P}_{s-1,j}^k$ be given by
\begin{equation*}
\begin{split}
\widetilde{P}_{s-1,j}^k=K\frac{A_{s-1}^2}{2}+HA_{s-1}+Q
\end{split}
\end{equation*}
where $K$, $H$, are constants with respect to $A_{s-1}$. Then, we see that $K$ and $H$ are given by
\begin{equation*}
K=\widetilde{P}_{s+1,j}^{k-2}\quad\text{and}\quad
H=\widetilde{P}_{s,j}^{k-1}-\widetilde{P}_{s+1,j}^{k-2}A_{s-1}=\widetilde{P}_{s+1,j}^{k-1}+\frac{1}{L}D\widetilde{P}_{s+1,j}^{k-2}.
\end{equation*}
We will need the following intermediate calculation to complete the proof:
\begin{equation*}
\begin{split}
\frac{\partial (D\widetilde{P}_{s-1,j}^k)}{\partial A_{s-1}}
=&\frac{\partial}{\partial A_{s-1}}((DK)\frac{A_{s-1}^2}{2}+KA_{s-1}DA_{s-1}+(DH)A_{s-1}+H(DA_{s-1})+DQ)\\
=&(DK)A_{s-1}+K(DA_{s-1})+KLA_{s-1}^2+DH+LHA_{s-1}\\
=&(D\widetilde{P}_{s+1,j}^{k-2})A_{s-1}+\widetilde{P}_{s+1,j}^{k-2}(DA_{s-1})+L\widetilde{P}_{s+1,j}^{k-2}A_{s-1}^2\\
&+D\widetilde{P}_{s,j}^{k-1}-(D\widetilde{P}_{s+1,j}^{k-2})A_{s-1}-\widetilde{P}_{s+1,j}^{k-2}(DA_{s-1})+L\widetilde{P}_{s,j}^{k-1}A_{s-1}-L\widetilde{P}_{s+1,j}^{k-2}A_{s-1}^2\\
=&D\widetilde{P}_{s,j}^{k-1}+L\widetilde{P}_{s,j}^{k-1}A_{s-1}.
\end{split}
\end{equation*}
Next, we compute
\begin{equation}\label{eq:Alminusonevanishing}
\begin{split}
\frac{\partial \widetilde{P}_{s-2,j}^k}{\partial A_{s-1}}
=&\frac{\partial \widetilde{P}_{{s-1},j}^k}{\partial A_{s-1}}+\frac{1}{L}\frac{\partial\left(D\widetilde{P}_{{s-1},j}^{k-1}\right)}{\partial A_{s-1}}-\widetilde{P}_{{s-1},j}^{k-1}-A_{s-1}\frac{\partial \widetilde{P}_{{s-1},j}^{k-1}}{\partial A_{s-1}}\\
=&\widetilde{P}_{s,j}^{k-1}+\frac{1}{L}D\widetilde{P}_{s,j}^{k-2}+\widetilde{P}_{s,j}^{k-2}A_{s-1}-\widetilde{P}_{s-1,j}^{k-1}-A_{s-1}\widetilde{P}_{s,j}^{k-2}\\
=&\widetilde{P}_{s,j}^{k-1}+\frac{1}{L}D\widetilde{P}_{s,j}^{k-2}-\widetilde{P}_{s-1,j}^{k-1}=0.
\end{split}
\end{equation}

Equation (\ref{eq:Alminusonevanishing}) shows that $\widetilde{P}_{s-2,j}^k$ does not depend on $A_{s-1}$. Since the lifting procedure is an inductive process, it follows that $\widetilde{P}_{i,j}^k$ does not depend on $A_{s-1}$ for $1\leq i \leq s-2$ by the modified flatness equations (\ref{eqn:modflateqn}), and the lifting procedure (\ref{eq:ModifiedFlatnessLift}). Hence, we conlude that
\begin{equation*}
\frac{\partial\widetilde{P}_{i,j}^k}{\partial{A_{s-1}}}=0
\end{equation*}
for $1\leq i \leq s-2$. Combining this equation together with the equations (\ref{eq:evenderivativepart1}), (\ref{eq:evenderivativepart2}), and (\ref{eq:evenderivativepart3}), we complete the proof.
\end{proof}

\subsection{Formula for potentials}
\subsubsection{Semisimple Cohomological Field Theories}
By general considerations, Gromov-Witten theory of $\CnZn$ has the structure of a cohomological field theory (CohFT). We refer to \cite{km} and \cite{Picm} for discussions on CohFTs. 

By the results of Section \ref{sec:genus_0}, this CohFT is semisimple. The Givental-Teleman classification of semisimple CohFTs \cite{g3}, \cite{t} states that a semisimple CohFT $\Omega$ can be obtained from its {\em topological part} via the actions of its $R$-matrix and $T$-vector, where $T(z)$ is given by $z(\text{Id}-R(z))$ applied to the unit. We refer to \cite{Picm} and \cite{ppz} for detailed discussions on this. 

Generating functions of a CohFT $\Omega$ can be defined by integrating CohFT classes. If $\Omega$ is semisimple, its topological part can be evaluated explicitly in the idempotent basis, see e.g. \cite[Section 2.5.1]{lho-p2}. A consequence of the Givental-Teleman classification is that the generating functions of $\Omega$ can be explicitly written as sums of graphs. A reference for this can be found in \cite[Section 2.5.2]{lho-p2}.

The $R$-matrix for the Gromov-Witten theory of $\CnZn$ is studied in Section \ref{sec:genus_0}. The general consideration on semisimple CohFTs recalled above yields a formula for the Gromov-Witten potential $\mathcal{F}_{g, m}^{\left[\mathbb{C}^{n} / \mathbb{Z}_{n}\right]}\left(\phi_{c_{1}}, \ldots, \phi_{c_{m}}\right)$. In the remainder of this subsection, we work out this formula in details.

\subsubsection{Graphs}
In order to state the formula for Gromov-Witten potentials, we need to describe certain graphs.

A \textit{stable graph} $\Gamma$ is a tuple
\begin{equation*}
\Gamma=\left(\mathrm{V}_{\Gamma}, \mathrm{g}: \mathrm{V}_{\Gamma} \rightarrow \mathbb{Z}_{\geq 0},  \mathrm{H}_{\Gamma}, \iota: \mathrm{H}_{\Gamma} \rightarrow \mathrm{H}_{\Gamma}, \mathrm{L}_{\Gamma}, \ell:\mathrm{L}_{\Gamma}\rightarrow\{1,\ldots,m\},  \nu: \mathrm{H}_{\Gamma} \rightarrow \mathrm{V}_{\Gamma}\right)
\end{equation*}
satisfying:
\begin{enumerate}
\item $\mathrm{V}_{\Gamma}$ is the set of vertices and $\mathrm{g}:\mathrm{V}_{\Gamma}\rightarrow\mathbb{Z}_{\geq 0}$ is a genus assignment,
    
\item $\mathrm{H}_{\Gamma}$ is the set of half-edges and  $\iota: \mathrm{H}_{\Gamma} \rightarrow \mathrm{H}_{\Gamma}$ is an involution,
    
\item $\mathrm{E}_{\Gamma}$ is the set of edges\footnote{Self-edges are allowed.} defined by the orbits of $\iota: \mathrm{H}_{\Gamma} \rightarrow \mathrm{H}_{\Gamma}$, and the tuple $\left(\mathrm{V}_{\Gamma},\mathrm{E}_{\Gamma}\right)$ defines a connected graph,
    
\item $\mathrm{L}_{\Gamma}$ is the set of legs, the subset of $\mathrm{H}_{\Gamma}$ fixed by the involution $\iota: \mathrm{H}_{\Gamma} \rightarrow \mathrm{H}_{\Gamma}$ and the map $\ell:\mathrm{L}_{\Gamma}\rightarrow\{1,\ldots,m\}$ is an isomorphism labeling legs,

\item The map $\nu: \mathrm{H}_{\Gamma} \rightarrow \mathrm{V}_{\Gamma}$ is a vertex assignment,
    
\item For each vertex $\mathfrak{v}$, let $\mathrm{l}(\mathfrak{v})$ and $\mathrm{h}(\mathfrak{v})$ be the number of legs and the number of edges attached to the vertex $\mathfrak{v}$ respectively and hence $\mathrm{n}(\mathfrak{v})=\mathrm{l}(\mathfrak{v})+\mathrm{h}(\mathfrak{v})$ be the valence of the vertex $\mathfrak{v}$. Then, for each vertex $\mathfrak{v}$ the following (stability) condition holds:
\begin{equation*}
2\mathrm{g}(\mathfrak{v})-2+\mathrm{n}(\mathfrak{v})>0.
\end{equation*}
\end{enumerate}
The \textit{genus} of $\Gamma$ is defined by
\begin{equation*}
\mathrm{g}(\Gamma)=h^1(\Gamma)+\sum_{\mathfrak{v}\in\mathrm{V}_{\Gamma}}\mathrm{g}(\mathfrak{v}).
\end{equation*}
An \textit{isomorphism} $\varphi:\Gamma \rightarrow \widetilde{\Gamma}$ of stable graphs is a collection of bijective maps
\begin{equation*}
\varphi_{\mathrm{V}}:\mathrm{V}_{\Gamma}\rightarrow \mathrm{V}_{\widetilde{\Gamma}}\quad \text{and} \quad \varphi_{\mathrm{H}}:\mathrm{H}_{\Gamma}\rightarrow \mathrm{H}_{\widetilde{\Gamma}}
\end{equation*}
of their sets of vertices and half-edges which are compatible with their genus, leg, vertex assignments and involutions $\iota$, $\tilde{\iota}$:
\begin{align*}
\tilde{\mathrm{g}}(\varphi_{\mathrm{V}}(\mathfrak{v}))=&\mathrm{g}(\mathfrak{v}),\\
\tilde{\nu}(\varphi_{\mathrm{H}}(\mathfrak{h}))=& \varphi_{\mathrm{V}}(\nu(\mathfrak{h})),\\
\tilde{\iota}(\varphi_{\mathrm{H}}(\mathfrak{h}))=& \varphi_{\mathrm{H}}(\iota(\mathfrak{h})),\\
\tilde{\ell}(\varphi_{\mathrm{H}}(\mathfrak{h}))=& \ell (\mathfrak{h}).
\end{align*}
Let $\mathrm{G}_{g,m}$ be the isomorphism classes of stable graphs of genus $g$ with $m$ legs.

In the formula for Gromov-Witten potentials, we need to work with {\em decorated} stable graphs. This has to do with the $T$-action on CohFTs. To see this, we recall the description of the $T$-actions in general, as follows. Let $\Omega=\{\Omega_{g,m}\}_{2g-2+m>0}$ be a CohFT based on the vector space $V$ and let
\begin{equation*}
T(z)=T_2z^2+T_3z^3+\cdots
\end{equation*}
be a $V$-valued power series with vanishing coefficients in degrees $0$ and $1$. The \textit{translation} of $\Omega$ by $T$ is the CohFT $T\Omega$ defined by

\begin{equation*}
T \Omega_{g, m}\left(v_{1}, \ldots,v_{m}\right)
=\sum_{k \geqslant 0} \frac{1}{k !}\left(\pi_{k}\right)_{*} \Omega_{g, m+k}\left(v_{1}, \ldots, v_{m}, T\left(\psi_{m+1}\right), \ldots, T\left(\psi_{m+k}\right)\right)
\end{equation*}
where $\pi_k:\overline{M}_{g,m+k}\rightarrow\overline{M}_{g,m}$ is the forgetful map dropping the last $k$ marked points. Here, by above notation, we actually mean
\begin{equation*}
\Omega_{g,m+k}\left(\ldots ,T\left(\psi_{i}\right),\ldots\right)=\sum_{r \geqslant 2} \psi_{i}^{r} \Omega_{g,m+k}\left(\ldots, T_{r}, \ldots \right).
\end{equation*}

Now, consider elements of $V$ written in terms of the normalized idempotent basis,
\begin{equation*}
v_j=\sum_{i=0}^{n-1}v_{ij}\widetilde{e}_i\quad\text{and}\quad T_r=\sum_{i=0}^{n-1}T_{ir}\widetilde{e}_i.
\end{equation*}
Then, assume in addition\footnote{This case suffices for our purpose, because in the formula for $\mathcal{F}_{g,m}$ from Givental-Teleman classification, $T$ acts on the topological part.} that $\Omega$ is a topological field theory, we have
\begin{equation}\label{eq:TmatrixActionExpanded}
\begin{split}
T \Omega_{g, m}\left(v_1, \ldots,v_{m}\right)
&=\sum_{k \geqslant 0} \frac{1}{k !}\left(\pi_{k}\right)_{*} \Omega_{g, m+k}\left(v_{1}, \ldots, v_{m}, T\left(\psi_{m+1}\right), \ldots, T\left(\psi_{m+k}\right)\right)\\
&=\sum_{k \geqslant 0} \frac{1}{k !}\sum_{r_1,\ldots,r_k\geq{2}}\left(\prod_{l=1}^k\left(\pi_{k}\right)_{*}\left(\psi^{r_l}_{m+l}\right)\right)\Omega_{g, m+k}\left(v_{1}, \ldots, v_{m}, T_{r_1}, \ldots, T_{r_k}\right)\\
&=\sum_{k \geqslant 0} \frac{1}{k !}\sum_{r_1,\ldots,r_k\geq{2}}\left(\prod_{l=1}^k\left(\pi_{k}\right)_{*}\left(\psi^{r_l}_{m+l}\right)\right)\\
&\times\sum_{0\leq i_1,\ldots,i_{m+k}\leq{n-1}}v_{i_11}\cdots v_{i_mm}T_{i_{m+1}r_1}\cdots T_{i_{m+k}r_k} \Omega_{g, m+k}\left(\widetilde{e}_{i_1}, \ldots, \widetilde{e}_{i_m}, \widetilde{e}_{i_{m+1}}, \ldots, \widetilde{e}_{i_{m+k}}\right)\\
&=\sum_{k \geqslant 0} \frac{1}{k !}\sum_{r_1,\ldots,r_k\geq{2}}\left(\prod_{l=1}^k\left(\pi_{k}\right)_{*}\left(\psi^{r_l}_{m+l}\right)\right)\sum_{i=0}^{n-1}v_{i1}\cdots v_{im}T_{ir_1}\cdots T_{ir_k}g({e}_i,{e}_i)^{-\frac{2g-2+m+k}{2}}\\
&=\sum_{i=0}^{n-1}v_{i1}\cdots v_{im}\sum_{k \geq 0} \frac{g({e}_i,{e}_i)^{-\frac{2g-2+m+k}{2}}}{k !}\sum_{r_1,\ldots,r_k\geq{2}}T_{ir_1}\cdots T_{ir_k}\left(\prod_{l=1}^k\left(\pi_{k}\right)_{*}\left(\psi^{r_l}_{m+l}\right)\right)\\
&=\sum_{i=0}^{n-1}{v_{i1}\cdots v_{im}\sum_{k \geq 0} \frac{g({e}_i,{e}_i)^{-\frac{2g-2+m+k}{2}}}{k !}\left(\pi_{k}\right)_{*}\left(t_i(\psi_{m+1})\cdots t_i(\psi_{m+k})\right)},
\end{split}
\end{equation}
where
\begin{equation*}
t_i(z)=\sum_{r\geq{2}}T_{ir}z^r.
\end{equation*}
Note that the above derivation uses the explicit evaluation of a topological field theory in the normalized idempotent basis, c.f. \cite[Section 2.5.1]{lho-p2}.

Setting the summand of $T \Omega_{g, m}\left(v_1, \ldots,v_{m}\right)$ to
\begin{equation*}
\widetilde{\Omega}^i_{g, m}\left(v_1, \ldots,v_{m}\right)={v_{i1}\cdots v_{im}\sum_{k \geq 0} \frac{g({e}_i,{e}_i)^{-\frac{2g-2+m+k}{2}}}{k !}\left(\pi_{k}\right)_{*}\left(t_i(\psi_{m+1})\cdots t_i(\psi_{m+k})\right)}
\end{equation*}
we can write it as
\begin{equation}\label{eqn:TOmega}
T\Omega_{g, m}\left(v_1, \ldots,v_{m}\right)=\sum_{i=0}^{n-1}\widetilde{\Omega}^i_{g, m}\left(v_1, \ldots,v_{m}\right).
\end{equation}
By the formula for generating functions, as described in \cite[Section 2.5.2]{lho-p2}, for each vertex in a stable graph, a class $T\Omega_{g,m}$ is inserted. By equation (\ref{eqn:TOmega}), this is equivalent to inserting $\widetilde{\Omega}^i_{g,m}$ if vertices of stable graphs carry extra labels $i$. This leads to the notion of decorated stable graphs.

More precisely, a {decorated} stable graph $$\Gamma\in\mathrm{G}_{g,m}^{\text{Dec}}(n)$$ of order $n$ is a stable graph $\Gamma\in\mathrm{G}_{g,m}$ with an extra assignment $\mathrm{p}: \mathrm{V}_{\Gamma}\rightarrow \{0,...,n-1\}$ to each vertex $\mathfrak{v}\in\mathrm{V}_{\Gamma}$. For a decorated stable graph $\Gamma\in\mathrm{G}_{g,m}^{\text{Dec}}(n)$ we denote its underlying stable graph by $$\Gamma^{\mathrm{St}}\in\mathrm{G}_{g,m}$$ after forgetting the decoration.

\subsubsection{Formula for $\mathcal{F}_g$}
By the discussions above, we have
\begin{equation}\label{eqn:formula_Fg}
\mathcal{F}_{g, m}^{\left[\mathbb{C}^{n} / \mathbb{Z}_{n}\right]}\left(\phi_{c_{1}}, \ldots, \phi_{c_{m}}\right)=\sum_{\Gamma\in\mathrm{G}_{g,m}^{\text{Dec}}(n)}\operatorname{Cont}_{\Gamma}\left(\phi_{c_{1}}, \ldots, \phi_{c_{m}}\right).
\end{equation}

The following is the generalization of Proposition 15 of \cite{lho-p} to $\left[\mathbb{C}^n/\mathbb{Z}_n\right]$.

\begin{prop}\label{prop:contributions}
For each decorated stable graph $\Gamma\in\mathrm{G}_{g,m}^{\text{Dec}}(n)$, the associated contribution is given by
\begin{equation*}
\operatorname{Cont}_{\Gamma}\left(\phi_{c_{1}}, \ldots, \phi_{c_{m}}\right)=\frac{1}{|\mathrm{Aut}(\Gamma^{\mathrm{St}})|} \sum_{\mathrm{A} \in \mathbb{Z}_{\geq 0}^{\mathrm{F}(\Gamma)}} \prod_{\mathfrak{v} \in \mathrm{V}_{\Gamma}} \mathrm{Cont}_{\Gamma}^{\mathrm{A}}(\mathfrak{v}) \prod_{\mathfrak{e}\in \mathrm{E}_{\Gamma}} \mathrm{Cont}_{\Gamma}^{\mathrm{A}}(\mathfrak{e}) \prod_{\mathfrak{l} \in \mathrm{L}_{\Gamma}} \mathrm{Cont}_{\Gamma}^{\mathrm{A}}(\mathfrak{l})
\end{equation*}
where $\mathrm{F}(\Gamma)=\left\vert\mathrm{H}_{\Gamma}\right\vert$. Here, $\mathrm{Cont}_{\Gamma}^{\mathrm{A}}(\mathfrak{v})$, $\mathrm{Cont}_{\Gamma}^{\mathrm{A}}(\mathfrak{e})$, and $\mathrm{Cont}_{\Gamma}^{\mathrm{A}}(\mathfrak{l})$ are the {\em vertex}, {\em edge} and {\em leg} contributions with flag $\mathrm{A}-$values\footnote{Notation: The values ${b_{\mathfrak{v}1}},\ldots,{b_{\mathfrak{v}\mathrm{h}(\mathfrak{v})}}$ and $b_{\mathfrak{e}1},b_{\mathfrak{e}2}$ are the entries of $(a_1,\ldots,a_m,b_{m+1},\ldots,b_{\left\vert\mathrm{H}_{\Gamma}\right\vert})$ corresponding to $\mathrm{Cont}_{\Gamma}^{\mathrm{A}}(\mathfrak{v})$ and $\mathrm{Cont}_{\Gamma}^{\mathrm{A}}(\mathfrak{e})$ respectively.} $(a_1,\ldots,a_m,b_{m+1},\ldots,b_{\left\vert\mathrm{H}_{\Gamma}\right\vert})$ respectively, and they are given by
\begin{equation*}
\begin{split}
    \mathrm{Cont}_{\Gamma}^{\mathrm{A}}(\mathfrak{v})
    =&\sum_{k \geq 0} \frac{g({e}_{\mathrm{p}(\mathfrak{v})},{e}_{\mathrm{p}(\mathfrak{v})})^{-\frac{2\mathrm{g}(\mathfrak{v})-2+\mathrm{n}(\mathfrak{v})+k}{2}}}{k !}\\
    &\times\int_{\overline{M}_{\mathrm{g}(\mathfrak{v}),\mathrm{n}(\mathfrak{v})+k}}\psi_1^{a_{\mathfrak{v}1}}\cdots\psi_{\mathrm{l}(\mathfrak{v})}^{a_{\mathfrak{v}\mathrm{l}(\mathfrak{v})}}\psi_{\mathrm{l}(\mathfrak{v})+1}^{b_{\mathfrak{v}1}}\cdots\psi_{\mathrm{n}(\mathfrak{v})}^{b_{\mathfrak{v}\mathrm{h}(\mathfrak{v})}}t_{\mathrm{p}(\mathfrak{v})}(\psi_{\mathrm{n}(\mathfrak{v})+1})\cdots t_{\mathrm{p}(\mathfrak{v})}(\psi_{\mathrm{n}(\mathfrak{v})+k}),\\
    \mathrm{Cont}_{\Gamma}^{\mathrm{A}}(\mathfrak{e})
    =&\frac{(-1)^{b_{\mathfrak{e}1}+b_{\mathfrak{e}2}}}{n} \sum_{j=0}^{b_{\mathfrak{e}2}}(-1)^{j} \sum_{r=0}^{n-1}\frac{\widetilde{P}_{\mathrm{Inv}(r),\mathrm{p}(\mathfrak{v}_1)}^{b_{\mathfrak{e}1}+j+1}\widetilde{P}_{r,\mathrm{p}(\mathfrak{v}_2)}^{b_{\mathfrak{e}2}-j}}{\zeta^{(b_{\mathfrak{e}1}+j+1+\mathrm{Inv}(r))\mathrm{p}(\mathfrak{v}_1)}\zeta^{(b_{\mathfrak{e}2}-j+r)\mathrm{p}(\mathfrak{v}_2)}},\\
    \mathrm{Cont}_{\Gamma}^{\mathrm{A}}(\mathfrak{l})
    =&\frac{(-1)^{a_{\ell(\mathfrak{l})}}}{n}\frac{K_{\mathrm{Inv}(c_{\ell(\mathfrak{l})})}}{L^{\mathrm{Inv}(c_{\ell(\mathfrak{l})})}}
    \frac{\widetilde{P}_{\mathrm{Inv}(c_{\ell(\mathfrak{l})}),\mathrm{p}(\nu(\mathfrak{l}))}^{{a_{\ell(\mathfrak{l})}}}}{     \zeta^{({a_{\ell(\mathfrak{l})}}+{\mathrm{Inv}(c_{\ell(\mathfrak{l})})})\mathrm{p}(\nu(\mathfrak{l}))}},
\end{split}
\end{equation*}
where
\begin{equation*}
t_{\mathrm{p}(\mathfrak{v})}(z)=\sum_{i\geq{2}}\mathrm{T}_{\mathrm{p}(\mathfrak{v})i}z^i\quad\text{with}\quad \mathrm{T}_{\mathrm{p}(\mathfrak{v})i}=\frac{(-1)^i}{n}\widetilde{P}_{0,\mathrm{p}(\mathfrak{v})}^{i-1}\zeta^{-(i-1)\mathrm{p}(\mathfrak{v})}.
\end{equation*}
\end{prop}

\begin{proof}
To simplify notations, write $\{\tilde{e}\}$ for the normalized idempotent basis $\{\tilde{e}_0,\ldots,\tilde{e}_{n-1}\}$ and $\{\phi\}$ for the basis $\{\phi_0,\ldots,\phi_{n-1}\}$. Let $\mathcal{T}_{\tilde{e}}^{\phi}$ be the transition matrix from $\{\tilde{e}\}$ to $\{\phi\}$ and let $\mathcal{T}_{\phi}^{\tilde{e}}$ be its inverse i.e. the transition matrix from $\{\phi\}$ to  $\{\tilde{e}\}$. Then, we have $$\mathcal{T}_{\tilde{e}}^{\phi}=\Psi^{-1},\quad \mathcal{T}_{\phi}^{\tilde{e}}=\Psi.$$

Let $G$ and $\widetilde{G}$ be matrix representations of the metric $g$ with respect to basis $\{\phi\}$ and  $\{\tilde{e}\}$. Then, the relation between them is given by
\begin{equation}\label{eq:contproofeq1}
\widetilde{G}=\left(\Psi^{-1}\right)^TG\Psi^{-1}.
\end{equation}
And it can easily be shown that we have $\widetilde{G}=\mathrm{Id}$.

Define $T(z)=z\left(\mathrm{Id}-R^{-1}(z)\right)\cdot{\phi_0}$. We provided $R$-matrix action with respect to normalized idempotent basis. To be consistent we need to write $\phi_0$ in terms of $\{\tilde{e}\}$ basis. Since we have
\begin{equation}\label{eq:contproofeq2}
\phi_0=\sum_{i=0}^{n-1}\Psi_{i0}\tilde{e}_i=\frac{1}{n}\left(\tilde{e}_0+\ldots+\tilde{e}_{n-1}\right),
\end{equation}
we see that $T(z)=z\left(\mathrm{Id}-R^{-1}(z)\right)v$ where $v=\frac{1}{n}[1\,\cdots\, 1]^T$.

We now find $R^{-1}(z)$. By the symplectic condition, $R^{-1}(z)=R^t(-z)$. Here $R^t(-z)$ means adjoint with respect to the metric $g$ in the basis $\{\tilde{e}\}$. We see that
\begin{equation}\label{eq:contproofeq3}
R^{-1}(z)=R^t(-z)=\widetilde{G}^{-1}R^T(-z)\widetilde{G}=R^T(-z)=\left(\Psi P(-z)\right)^T=P^T(-z)\Psi^T.
\end{equation}

Also, observe that 
\begin{equation}\label{eq:contproofeq4}
\begin{split}
{\left[ \Psi^Tv\right]}_i
=&\frac{1}{n}\sum_{j=0}^{n-1}\Psi_{ij}^T=\frac{1}{n}\sum_{j=0}^{n-1}\Psi_{ji}\\
=&\frac{1}{n}\sum_{j=0}^{n-1}\frac{1}{n}\zeta^{ij}\frac{L^i}{K_i}=\frac{1}{n^2}\frac{L^i}{K_i}\sum_{j=0}^{n-1}\zeta^{ij}=\frac{1}{n^2}n\delta_{i,0}=\frac{1}{n}\delta_{i,0}.
\end{split}
\end{equation}
So, we have $\Psi^Tv=\frac{1}{n}\left[1\,0\,\cdots\,0\right]^T$. This implies that the translation vector
\begin{equation}\label{eq:contproofeq5}
T(z)=z\left(\mathrm{Id}-R^{-1}(z)\right)v=T_2z^2+T_3z^3+\cdots
\end{equation}
where $T_k$ is the coefficient of $z^{k-1}$ in $-R^{-1}(z)v$ given by 
\begin{equation}\label{eq:contproofeq6}
\begin{split}
 T_{jk}
 &=\text{the $j^{\text{th}}$ entry of the coefficient of $z^{k-1}$  in }-R^{-1}(z)v\\
 &=\text{the $j^{\text{th}}$ entry of the coefficient of $z^{k-1}$  in }-P^T(-z)\Psi^Tv\\
 &=\frac{(-1)^k}{n}P_{0j}^{k-1}.
\end{split}
\end{equation}
This enables us to understand the translation action by $T(z)$ and vertex contributions after the translation action. Next, we will understand the effects of the $R$-matrix action and obtain the expressions for the contributions fully.

Now, consider
\begin{equation*}
F(z,w)=\frac{M(z,w)}{z+w}
\end{equation*}
with both $F(z,w),M(z,w)\in\mathbb{C}[\![z,w]\!]$, 
\begin{equation*}
F(z, w)=\sum_{a, b \geq 0} \beta_{a, b} z^{a} w^{b}\quad\text{and}\quad M(z, w)=\sum_{c, d \geq 0} \alpha_{c, d} z^{c} w^{d}.
\end{equation*}
Then, the coefficients $\beta_{a, b}$ are given by
\begin{equation}\label{eq:contproofeq7}
\beta_{a, b}=\sum_{m=0}^{b}(-1)^{m} \alpha_{a+m+1,b-m}.
\end{equation}

Now observe that
\begin{equation}\label{eq:contproofeq8}
\begin{split}
{\left[\Psi^T \Psi\right]}_{l j}
&=
\sum_{r=0}^{n-1} \Psi_{r l} \Psi_{r j}=\sum_{r=0}^{n-1} \frac{1}{n}\zeta^{r l} \frac{L^{l}}{K_{l}} \frac{1}{n}\zeta^{r j} \frac{L^{j}}{K_{j}}\\
&=
\frac{1}{n^2}\frac{L^{l}}{K_{l}}\frac{L^{j}}{K_{j}}\sum_{r=0}^{n-1}\zeta^{r(l-\mathrm{Inv}(j))}=\frac{1}{n^2}\frac{L^{l}}{K_{l}}\frac{L^{j}}{K_{j}}n\delta_{l,\mathrm{Inv}(j)}\\
&=
\frac{1}{n}\underbrace{\frac{L^{\mathrm{Inv}(j)+j}}{K_{\mathrm{Inv}(j)}K_j}}_{=1}\delta_{l,\mathrm{Inv}(j)}=\frac{1}{n}\delta_{l,\mathrm{Inv}(j)}.
\end{split}
\end{equation}

Next, in order to understand the edge contributions, we compute
\begin{equation}\label{eq:contproofeq9}
\begin{split}
{\delta_{i,j}-{\left[R^{-1}(z)R^{-1}(w)^T\right]}_{ij}}
&=\delta_{i,j}-\sum_{s,r=0}^{n-1}P_{i,s}^T(-z)\left[\Psi^T\Psi\right]_{sr} P_{r,j}(-w)\\
&=\delta_{i,j}-\sum_{s,r=0}^{n-1}P_{i,s}^T(-z)\frac{1}{n}\delta_{s,\mathrm{Inv}(r)} P_{r,j}(-w)\\
&=\delta_{i,j}-\frac{1}{n}\sum_{r=0}^{n-1}P_{i,\mathrm{Inv}(r)}^T(-z)P_{r,j}(-w)\\
&=\delta_{i,j}-\frac{1}{n}\sum_{r=0}^{n-1}\sum_{c,d\geq{0}}(-1)^{c+d}P^c_{\mathrm{Inv}(r),i}P^d_{r,j}z^cw^d\\
&=\delta_{i,j}-\frac{1}{n}\sum_{r=0}^{n-1}\sum_{c,d\geq{0}}(-1)^{c+d}\frac{K_{\mathrm{Inv}(r)}}{L^{\mathrm{Inv}(r)}}\frac{\widetilde{P}^c_{\mathrm{Inv}(r),i}}{\zeta^{(c+\mathrm{Inv}(r))i}}\frac{K_r}{L^r}\frac{\widetilde{P}^d_{r,j}}{\zeta^{(d+r)j}}z^cw^d\\
&=\delta_{i,j}-\frac{1}{n}\sum_{r=0}^{n-1}\sum_{c,d\geq{0}}(-1)^{c+d}\frac{\widetilde{P}^c_{\mathrm{Inv}(r),i}\widetilde{P}^d_{r,j}}{\zeta^{(c+\mathrm{Inv}(r))i}\zeta^{(d+r)j}}z^cw^d.
\end{split}
\end{equation}
So, we have
\begin{equation}\label{eq:contproofeq10}
\frac{{\delta_{i,j}-{\left[R^{-1}(z)R^{-1}(w)^T\right]}_{ij}}}{z+w}=\sum_{b_1,b_2\geq{0}}\beta^{i,j}_{b_1,b_2}z^{b_1}w^{b_2}
\end{equation}
with
\begin{equation}\label{eq:contproofeq11}
\beta^{i,j}_{b_1,b_2}=\frac{(-1)^{b_1+b_2}}{n}\sum_{m=0}^{b_2}(-1)^m\sum_{r=0}^{n-1}\frac{\widetilde{P}^{b_1+m+1}_{\mathrm{Inv}(r),i}\widetilde{P}^{b_2-m}_{r,j}}{\zeta^{(b_1+m+1+\mathrm{Inv}(r))i}\zeta^{(b_2-m+r)j}}
\end{equation}
by equation (\ref{eq:contproofeq7}).

In order to understand the leg contributions, we compute 
\begin{equation}\label{eq:contproofeq12}
\begin{split}
\left[R^{-1}(z)\cdot\phi_j\right]_i
&=\left[P^T(-z)\Psi^T\Psi\right]_{ij}=\sum_{a\geq{0}}(-1)^a\sum_{r=0}^{n-1}P_{r,i}^a\frac{1}{n}\delta_{r,\mathrm{Inv}(j)}z^a\\
&=\sum_{a\geq{0}}\frac{(-1)^a}{n}\frac{K_{\mathrm{Inv}(j)}}{L^{\mathrm{Inv}(j)}}\frac{\widetilde{P}^a_{\mathrm{Inv}(j),i}}{\zeta^{(a+\mathrm{Inv}(j))i}}z^a
\end{split}
\end{equation}
for each $0\leq{i,j}\leq{n-1}$.

Let $\mathfrak{v}\in\mathrm{V}_{\Gamma}$ be a vertex of a decorated stable graph $\Gamma$ with legs $\{\mathfrak{l}_{\mathfrak{v}1},\ldots,\mathfrak{l}_{\mathfrak{v}\mathrm{l}(\mathfrak{v})}\}\subseteq\mathrm{L}_{\Gamma}$ and edges $\{\mathfrak{e}_{\mathfrak{v}1},\ldots,\mathfrak{e}_{\mathfrak{v}\mathrm{h}(\mathfrak{v})}\}$. Then, the (cycle-valued) contribution associated with this vertex, its legs, and edges connected to this vertex is\footnote{Notation: Here $\mathfrak{v}'_i$ is the vertex at the other end of the edge $e_{\mathfrak{v}i}$ and $\psi_{\mathrm{l}(\mathfrak{v}'_i)+j}$ is the psi class associated to the marked point corresponding to the other half of the edge $e_{\mathfrak{v}i}$.}
\begin{equation*}
\begin{split}
\widetilde{\Omega}^{\mathrm{p}(\mathfrak{v})}_{\mathrm{g}(\mathfrak{v}),\mathrm{l}(\mathfrak{v})+\mathrm{h}(\mathfrak{v})}(R^{-1}(\psi_1)\cdot&\phi_{c_{\ell(l_{\mathfrak{v}1})}},\ldots,R^{-1}(\psi_{\mathrm{l}(\mathfrak{v})})\cdot\phi_{c_{\ell(l_{\mathfrak{v}\mathrm{l}(\mathfrak{v})})}},\underbrace{\widetilde{e}_{\mathrm{p}(\mathfrak{v})},\ldots,\widetilde{e}_{\mathrm{p}(\mathfrak{v})}}_{\mathrm{h}(\mathfrak{v})\text{ many}})\\
&\times\prod_{i=1}^{\mathrm{h}(\mathfrak{v})}\left(\sum_{{b_{\mathfrak{e}_{\mathfrak{v}i}}},b_{\mathfrak{e}'_i}\geq{0}}\beta_{b_{\mathfrak{e}_{\mathfrak{v}i}},b_{\mathfrak{e}'_i}}^{\mathrm{p}(\mathfrak{v}),\mathrm{p}(\mathfrak{v}'_i)}\psi_{\mathrm{l}(\mathfrak{v})+i}^{b_{\mathfrak{e}_{\mathfrak{v}i}}}\psi_{\mathrm{l}(\mathfrak{v}'_i)+j}^{b_{\mathfrak{e}'_i}}\right).
\end{split}
\end{equation*}
This together with equations (\ref{eq:TmatrixActionExpanded}), (\ref{eq:contproofeq11}), and (\ref{eq:contproofeq12}) complete the proof after integration of cycle-valued contributions of graph $\Gamma$ over $\overline{M}_{\mathrm{g,m}}$ and using functoriality of push-forward.
\end{proof}

Define the following free algebra $$\mathds{F}_{\CnZn}\coloneqq \mathbb{C}[L^{\pm{1}}][\mathfrak{S}_n][\mathfrak{C}_n]$$
where $\mathfrak{C}_n=\{C_1,\ldots,C_{n-1}\}$. Then, an immediate corollary of Proposition \ref{prop:contributions} is the following finite generation property or polynomiality of the Gromov-Witten potential $\mathcal{F}_{g, m}^{\left[\mathbb{C}^{n} / \mathbb{Z}_{n}\right]}\left(\phi_{c_{1}}, \ldots, \phi_{c_{m}}\right)$.
\begin{cor}\label{cor:VertexEdgeCont}
The vertex, edge, and leg contributions of $\operatorname{Cont}_{\Gamma}\left(\phi_{c_{1}}, \ldots, \phi_{c_{m}}\right)$ lie in certain polynomial rings. More, precisely
\begin{equation*}
\begin{split}
&\mathrm{Cont}_{\Gamma}^{\mathrm{A}}(\mathfrak{v})\in\mathbb{C}\left[{L}^{\pm{1}}\right],\\
&\mathrm{Cont}_{\Gamma}^{\mathrm{A}}(\mathfrak{e})\in\mathbb{C}[L^{\pm{1}}][\mathfrak{S}_n],\\
&\mathrm{Cont}_{\Gamma}^{\mathrm{A}}(\mathfrak{l})\in\mathbb{C}[L^{\pm{1}}][\mathfrak{S}_n][\mathfrak{C}_n]=\mathds{F}_{\CnZn}
\end{split}
\end{equation*}
Hence, we have
\begin{equation*}
\mathcal{F}_{g, m}^{\left[\mathbb{C}^{n} / \mathbb{Z}_{n}\right]}\left(\phi_{c_{1}}, \ldots, \phi_{c_{m}}\right)\in\mathds{F}_{\CnZn}.
\end{equation*}
\end{cor}

\begin{proof}
This is immediate from Lemma \ref{cor:P0jk_is_in_CL}, Proposition \ref{prop:contributions}, and the modified flatness equations (\ref{eqn:modflateqn}).
\end{proof}

We should note that $C_i$'s are related to each other via Lemma \ref{propertiesofCfunctions}, hence $\mathfrak{C}_n$ consists of $C_1,\ldots,C_{\lfloor{\frac{n+1}{2}}\rfloor}$. Also, we should emphasize that the Gromov-Witten potential $\mathcal{F}_{g, m}^{\left[\mathbb{C}^{n} / \mathbb{Z}_{n}\right]}\left(\phi_{c_{1}}, \ldots, \phi_{c_{m}}\right)$ may lie in a smaller ring depending on insertions. For example, $\mathcal{F}_g$ lies in $\mathbb{C}[L^{\pm{1}}][\mathfrak{S}_n]$.

The following two lemmas are crucial for the proof of holomorphic anomaly equations.

\begin{lem}\label{lem:oddderivativeofedge}
Let $n\geq{3}$ be an odd number with $n=2s+1$, then we have
\begin{equation*} 
    \frac{\partial}{\partial A_{s}}\mathrm{Cont}_{\Gamma}^{\mathrm{A}}(\mathfrak{e})
    =\frac{(-1)^{b_{\mathfrak{e}1}+b_{\mathfrak{e}2}}}{2s+1} \frac{\widetilde{P}_{s+1,\mathrm{p}(\mathfrak{v}_1)}^{b_{\mathfrak{e}1}}\widetilde{P}_{s+1,\mathrm{p}(\mathfrak{v}_2)}^{b_{\mathfrak{e}2}}}{\zeta^{(b_{\mathfrak{e}1}+s+1)\mathrm{p}(\mathfrak{e}_1)}\zeta^{(b_{\mathfrak{e}2}+s+1)\mathrm{p}(\mathfrak{v}_2)}}.
\end{equation*}
\end{lem}
\begin{proof}
The proof is the following direct computation:
\begin{equation*}
\begin{split}
    \frac{\partial}{\partial A_{s}}\mathrm{Cont}_{\Gamma}^{\mathrm{A}}(\mathfrak{e})
    =&\frac{(-1)^{b_{\mathfrak{e}1}+b_{\mathfrak{e}2}}}{n} \sum_{m=0}^{b_{\mathfrak{e}2}}(-1)^{m} \sum_{r=0}^{n-1}\frac{\frac{\partial}{\partial A_{s}}\left(\widetilde{P}_{\mathrm{Inv}(r),\mathrm{p}(\mathfrak{v}_1)}^{b_{\mathfrak{e}1}+m+1}\widetilde{P}_{r,\mathrm{p}(\mathfrak{v}_2)}^{b_{\mathfrak{e}2}-m}\right)}{\zeta^{(b_{\mathfrak{e}1}+m+1+\mathrm{Inv}(r))\mathrm{p}(\mathfrak{v}_1)}\zeta^{(b_{\mathfrak{e}2}-m+r)\mathrm{p}(\mathfrak{v}_2)}}\\
    =&\frac{(-1)^{b_{\mathfrak{e}1}+b_{\mathfrak{e}2}}}{n} \sum_{m=0}^{b_{\mathfrak{e}2}}(-1)^{m} \frac{\widetilde{P}_{s+1,\mathrm{p}(\mathfrak{v}_1)}^{b_{\mathfrak{e}1}+m}\widetilde{P}_{s+1,\mathrm{p}(\mathfrak{v}_2)}^{b_{\mathfrak{e}2}-m}}{\zeta^{(b_{\mathfrak{e}1}+m+s+1)\mathrm{p}(\mathfrak{v}_1)}\zeta^{(b_{\mathfrak{e}2}-m+s+1)\mathrm{p}(\mathfrak{v}_2)}}\\
    &+\frac{(-1)^{b_{\mathfrak{e}1}+b_{\mathfrak{e}2}}}{n} \sum_{m=0}^{b_{\mathfrak{e}2}-1}(-1)^{m}\frac{\widetilde{P}_{s+1,\mathrm{p}(\mathfrak{v}_1)}^{b_{\mathfrak{e}1}+m+1}\widetilde{P}_{s+1,\mathrm{p}(\mathfrak{v}_2)}^{b_{\mathfrak{e}2}-m-1}}{\zeta^{(b_{\mathfrak{e}1}+m+s+2)\mathrm{p}(\mathfrak{v}_1)}\zeta^{(b_{\mathfrak{e}2}-m+s)\mathrm{p}(\mathfrak{v}_2)}}\\
    =&\frac{(-1)^{b_{\mathfrak{e}1}+b_{\mathfrak{e}2}}}{n} \sum_{m=0}^{b_{\mathfrak{e}2}}(-1)^{m} \frac{\widetilde{P}_{s+1,\mathrm{p}(\mathfrak{v}_1)}^{b_{\mathfrak{e}1}+m}\widetilde{P}_{s+1,\mathrm{p}(\mathfrak{v}_2)}^{b_{\mathfrak{e}2}-m}}{\zeta^{(b_{\mathfrak{e}1}+m+s+1)\mathrm{p}(\mathfrak{v}_1)}\zeta^{(b_{\mathfrak{e}2}-m+s+1)\mathrm{p}(\mathfrak{v}_2)}}\\
    &+\frac{(-1)^{b_{\mathfrak{e}1}+b_{\mathfrak{e}2}}}{n} \sum_{m=1}^{b_{\mathfrak{e}2}}(-1)^{m-1} \frac{\widetilde{P}_{s+1,\mathrm{p}(\mathfrak{v}_1)}^{b_{\mathfrak{e}1}+m}\widetilde{P}_{s+1,\mathrm{p}(\mathfrak{v}_2)}^{b_{\mathfrak{e}2}-m}}{\zeta^{(b_{\mathfrak{e}1}+m+s+1)\mathrm{p}(\mathfrak{v}_1)}\zeta^{(b_{\mathfrak{e}2}-m+s+1)\mathrm{p}(\mathfrak{v}_2)}}\\
    =&\frac{(-1)^{b_{\mathfrak{e}1}+b_{\mathfrak{e}2}}}{2s+1} \frac{\widetilde{P}_{s+1,\mathrm{p}(\mathfrak{v}_1)}^{b_{\mathfrak{e}1}}\widetilde{P}_{s+1,\mathrm{p}(\mathfrak{v}_2)}^{b_{\mathfrak{e}2}}}{\zeta^{(b_{\mathfrak{e}1}+s+1)\mathrm{p}(\mathfrak{v}_1)}\zeta^{(b_{\mathfrak{e}2}+s+1)\mathrm{p}(\mathfrak{v}_2)}}.
\end{split}
\end{equation*}
The first equality is just the derivative of the edge contribution part of Proposition \ref{prop:contributions}. The second equality follows from Lemma \ref{lem:oddderivativeflatness} and $\widetilde{P}_{i,j}^k=0$ by definition if $k<0$. The rest is just shifting the index of the first sum and cancelling out the terms of the total expression.
\end{proof}
\begin{lem}\label{lem:evenderivativeofedge}
Let $n\geq{4}$ be an even number with $n=2s$, then we have
\begin{equation*} 
    \frac{\partial}{\partial A_{s-1}}\mathrm{Cont}_{\Gamma}^{\mathrm{A}}(\mathfrak{e})=\frac{(-1)^{b_{\mathfrak{e}1}+b_{\mathfrak{e}2}}}{2s}\left(\frac{\widetilde{P}_{s+1,\mathrm{p}(\mathfrak{v}_1)}^{b_{\mathfrak{e}1}}\widetilde{P}_{s,\mathrm{p}(\mathfrak{v}_2)}^{b_{\mathfrak{e}2}}}{\zeta^{(b_{\mathfrak{e}1}+s+1)\mathrm{p}(\mathfrak{v}_1)}\zeta^{(b_{\mathfrak{e}2}+s)\mathrm{p}(\mathfrak{v}_2)}}+\frac{\widetilde{P}_{s,\mathrm{p}(\mathfrak{v}_1)}^{b_{\mathfrak{e}1}}\widetilde{P}_{s+1,\mathrm{p}(\mathfrak{v}_2)}^{b_{\mathfrak{e}2}}}{\zeta^{(b_{\mathfrak{e}1}+s)\mathrm{p}(\mathfrak{v}_1)}\zeta^{(b_{\mathfrak{e}2}+s+1)\mathrm{p}(\mathfrak{v}_2)}}\right).
\end{equation*}
\end{lem}
\begin{proof}
The proof is similar to that of Lemma \ref{lem:oddderivativeofedge}. In this case, we use Lemma \ref{lem:evenderivativeflatness} instead of Lemma \ref{lem:oddderivativeflatness}.
\end{proof}

\subsection{The action on the edges}\label{sec:edge_action}

Let $\Gamma\in \mathrm{G}_{g,m}$ be a stable graph, and $\mathfrak{e}\in\mathrm{E}_{\Gamma}$ be an edge of $\Gamma$. The automorphism group $\mathrm{Aut}(\Gamma)$ acts on the set of edges $\mathrm{E}_{\Gamma}$ of $\Gamma$. Let $\mathrm{Stab}_{\Gamma}(\mathfrak{e})$ and $\mathrm{Orb_{\Gamma}(\mathfrak{e})}$ be the stabilizer and the orbit of the edge $\mathfrak{e}$ under the action of $\mathrm{Aut}(\Gamma)$, respectively.

We obtain a graph $\Gamma_{\mathfrak{e}}$ by  breaking an edge $\mathfrak{e}$ into two legs $\mathfrak{l}_{\mathfrak{e}}$ and $\mathfrak{l}_{\mathfrak{e}}'$. There are two possibilities: either the resulting graph is connected, or it is disconnected with two connected parts and each connected component has only one of $\mathfrak{l}_{\mathfrak{e}}$ and $\mathfrak{l}_{\mathfrak{e}}'$.
\begin{itemize}
    \item[(i)] If the resulting graph is connected, there is no canonical way of labeling the legs $\mathfrak{l}_{\mathfrak{e}}$ and $\mathfrak{l}_{\mathfrak{e}}'$. We can extend the labeling $\ell$ of the legs $\mathrm{L}_{\Gamma}$ of $\Gamma$ to a labeling of the legs $\mathrm{L}_{\Gamma_{\mathfrak{e}}}=\mathrm{L}_{\Gamma}\cup \{\mathfrak{l}_{\mathfrak{e}},\mathfrak{l}_{\mathfrak{e}}'\}$ of $\Gamma_{\mathfrak{e}}$ by labeling one of the $\mathfrak{l}_{\mathfrak{e}},\mathfrak{l}_{\mathfrak{e}}'$ as $m+1$ and the other as $m+2$. Fixing a labeling of $\mathfrak{l}_{\mathfrak{e}},\mathfrak{l}_{\mathfrak{e}}'$, the graph $\Gamma_{\mathfrak{e}}$ becomes a stable graph and it is an element of $\mathrm{G}_{g-1,m+2}$. Denote $\Gamma_{\mathfrak{e}}$ as $\Gamma_{\mathfrak{e},(m+1,m+2)}$ for one labeling and as $\Gamma_{\mathfrak{e},(m+2,m+1)}$ for the other labeling.
    
    \item[(ii)] If the resulting graph is disconnected, then we denote its connected components as $\Gamma_{{\mathfrak{e}}}^1$ and $\Gamma_{{\mathfrak{e}}}^2$. Without loss of generality assume  $\mathfrak{l}_{\mathfrak{e}}$ be a leg of $\Gamma_{{\mathfrak{e}}}^1$ and $\mathfrak{l}_{\mathfrak{e}}'$ be a leg of $\Gamma_{{\mathfrak{e}}}^2$.
    
    Let $m_i+1$ be the number of legs of $\Gamma_{{\mathfrak{e}}}^i$ and $g_i$ be the genus of $\Gamma_{{\mathfrak{e}}}^i$ for $i=1,2$. Then, we have $m_1+m_2=m$ and $g_1+g_2=g$. The labeling $\ell$ of the legs $\mathrm{L}_{\Gamma}$ of $\Gamma$ yields canonical labelings $\ell_i:\mathrm{L}_{\Gamma_{{\mathfrak{e}}}^i}\rightarrow \{1,\ldots,m_i+1\}$ of $\Gamma_{{\mathfrak{e}}}^i$ for $i=1,2$. This is achived by relabeling legs of $\Gamma_{{\mathfrak{e}}}^i$ coming from $\mathrm{L}_{\Gamma}$ in the ascending order of the labeling induced by $\ell$ and setting $\ell_1(\mathfrak{l}_{\mathfrak{e}})=m_1+1$ and $\ell_2(\mathfrak{l}_{\mathfrak{e}}')=m_2+1$. This makes each of $\Gamma_{{\mathfrak{e}}}^i$ a stable graph lying in $\mathrm{G}_{g_i,m_i+1}$ for $i=1,2$.
\end{itemize}

By the definition of the isomorphism of stable graphs and a basic graph analysis the following can be obtained.

\begin{lem}\label{lem:breaking_an_edge}
Assume two edges $\mathfrak{e}$ and $\tilde{\mathfrak{e}}$ of $\Gamma$ are in the same orbit. 
\begin{enumerate}
    \item For the connected case, the stable graph $\Gamma_{\mathfrak{e},(m+1,m+2)}$ is isomorphic to the stable graph $\Gamma_{\tilde{\mathfrak{e}},(m+1,m+2)}$ or the stable graph $\Gamma_{\tilde{\mathfrak{e}},(m+2,m+1)}$.
    \item For the disconnected case, the connected components of the graph $\Gamma_{\mathfrak{e}}$ are isomorphic as stable graphs to the connected components of the graph $\Gamma_{\tilde{\mathfrak{e}}}$.
\end{enumerate}
\end{lem}

\subsubsection{Case A} For a fixed edge $\mathfrak{e}$, assume all elements of $\mathrm{Aut}(\Gamma)$ fixes the half-edges of $\mathfrak{e}$, i.e. there does not exist an automorphism of $\Gamma$ interchanging the half-edges of $\mathfrak{e}$.
\begin{enumerate}
    \item For the connected case,  we have $$\vert \mathrm{Aut}({\Gamma_{\mathfrak{e}}})\vert=\vert \mathrm{Aut}(\Gamma_{\tilde{\mathfrak{e}},(m+1,m+2)})\vert=\vert \mathrm{Aut}(\Gamma_{\tilde{\mathfrak{e}},(m+2,m+1)})\vert=\vert \mathrm{Stab}_{\Gamma}(\mathfrak{e}) \vert .$$
    Hence, consequently we have
    $$\frac{\vert \mathrm{Aut}(\Gamma) \vert}{\vert \mathrm{Aut}({\Gamma_{\mathfrak{e}}}) \vert}=\vert \mathrm{Orb_{\Gamma}(\mathfrak{e})} \vert .$$
    
    \item For the disconnected case, we have
    $$\vert \mathrm{Aut}(\Gamma_{{\mathfrak{e}}}^1)\vert \vert \mathrm{Aut}(\Gamma_{{\mathfrak{e}}}^2)\vert=\vert \mathrm{Stab}_{\Gamma}(\mathfrak{e}) \vert\,.$$
    Hence, consequently we have
    $$\frac{\vert \mathrm{Aut}(\Gamma) \vert}{\vert \mathrm{Aut}(\Gamma_{{\mathfrak{e}}}^1)\vert \vert \mathrm{Aut}(\Gamma_{{\mathfrak{e}}}^2)\vert}=\vert \mathrm{Orb_{\Gamma}(\mathfrak{e})} \vert . $$
\end{enumerate}

\subsubsection{Case B} Now on the contrary assume there exists  an element $\varphi_{\mathfrak{e}}$ of $\mathrm{Aut}(\Gamma)$ interchanging the half edges of $\mathfrak{e}$. In the connected case, such $\varphi_{\mathfrak{e}}$ exists if and only if the stable graphs $\Gamma_{\tilde{\mathfrak{e}},(m+1,m+2)}$ and $\Gamma_{\tilde{\mathfrak{e}},(m+2,m+1)}$ are isomorphic. In the disconnected case, such $\varphi_{\mathfrak{e}}$ exists if and only if connected components are isomorphic to each other and have no legs other than ${\mathfrak{l}_{\mathfrak{e}}}$, and ${\mathfrak{l}_{\mathfrak{e}}'}$. If such $\varphi_{\mathfrak{e}}$ exists then we have the following situation:
\begin{itemize}
    \item If $\Gamma_{\mathfrak{e}}$ is connected, then $$\vert \mathrm{Aut}({\Gamma_{\mathfrak{e}}})\vert=\vert \mathrm{Aut}(\Gamma_{\tilde{\mathfrak{e}},(m+1,m+2)})\vert=\vert \mathrm{Aut}(\Gamma_{\tilde{\mathfrak{e}},(m+2,m+1)})\vert=\frac{1}{2}\vert \mathrm{Stab}_{\Gamma}(\mathfrak{e}) \vert .$$ Hence, consequently, we have $$\frac{\vert \mathrm{Aut}(\Gamma) \vert}{\vert \mathrm{Aut}({\Gamma_{\mathfrak{e}}}) \vert}=2\vert \mathrm{Orb_{\Gamma}(\mathfrak{e})} \vert .$$
    \item If $\Gamma_{\mathfrak{e}}$ is is disconnected, then we have $$\vert \mathrm{Aut}(\Gamma_{{\mathfrak{e}}}^1)\vert \vert \mathrm{Aut}(\Gamma_{{\mathfrak{e}}}^2)\vert=\frac{1}{2}\vert \mathrm{Stab}_{\Gamma}(\mathfrak{e}) \vert\,.$$
    Hence, consequently we have
    $$\frac{\vert \mathrm{Aut}(\Gamma) \vert}{\vert \mathrm{Aut}(\Gamma_{{\mathfrak{e}}}^1)\vert \vert \mathrm{Aut}(\Gamma_{{\mathfrak{e}}}^2)\vert}=2\vert \mathrm{Orb_{\Gamma}(\mathfrak{e})} \vert \,. $$
    
\end{itemize}

For both \textit{Case A} and \textit{Case B}, we will shortly use the notation $\mathrm{O}_{\Gamma}^{\mathfrak{e}}$ for the ratios

$$\frac{\vert \mathrm{Aut}(\Gamma) \vert}{\vert \mathrm{Aut}({\Gamma_{\mathfrak{e}}}) \vert}\quad \text{and} \quad \frac{\vert \mathrm{Aut}(\Gamma) \vert}{\vert \mathrm{Aut}(\Gamma_{{\mathfrak{e}}}^1)\vert \vert \mathrm{Aut}(\Gamma_{{\mathfrak{e}}}^2)\vert}\,.$$

\begin{rem}\label{rem:nonisomorphic_graph_contributions}
 For a decorated stable graph  $\Gamma\in\mathrm{G}_{g,m}^{\text{Dec}}(n)$ and and an edge ${\mathfrak{e}}\in\mathrm{E}_{\Gamma}$, we can define $\Gamma_{\mathfrak{e}}$ in a similar vein. If $\Gamma_{\mathfrak{e}}$ is connected then for any $0\leq i,j \leq n-1$, we have
\begin{equation}
\operatorname{Cont}_{\Gamma_{{\mathfrak{e},(m+1,m+2)}}}\left(\phi_{c_1},\ldots,\phi_{c_m},\phi_{i},\phi_{j}\right)
=
\operatorname{Cont}_{\Gamma_{{\mathfrak{e},(m+2,m+1)}}}\left(\phi_{c_1},\ldots,\phi_{c_m},\phi_{j},\phi_{i}\right)\,.
\end{equation}   
\end{rem} 
This suggests that in the connected case the notation
\begin{equation}
\operatorname{Cont}_{\Gamma_{\mathfrak{e}}}\left(\phi_{c_1},\ldots,\phi_{c_m},\phi_{i},\phi_{i}\right)
\end{equation}
is well-defined since it is independent of labelings of $\mathfrak{l}_{\tilde{\mathfrak{e}}}$ and $\mathfrak{l}'_{\tilde{\mathfrak{e}}}$. 

\begin{rem}\label{rem:isomorphic_graph_contribution}
If two stable graphs $\Gamma,\widetilde{\Gamma}\in\mathrm{G}_{g,m}$ are isomorphic, say via an isomorphism $\varphi:\Gamma \rightarrow \widetilde{\Gamma}$, then the sums of the contributions
$$\operatorname{Cont}_{\Gamma}(\phi_{c_1},\ldots,\phi_{c_m})\quad\text{and}\quad \operatorname{Cont}_{\Gamma}(\phi_{c_{\overline{\ell}(1)}},\ldots,\phi_{c_{\overline{\ell}(m)}})$$
over all possible decorations agree. Here, $\overline{\ell}:\{1,\ldots,m\}\rightarrow\{1,\ldots,m\}$ is the isomorphism defined by $\overline{\ell}=\tilde{\ell}\circ\varphi\circ\ell^{-1}.$  
\end{rem}

\subsection{Holomorphic anomaly equations}
Holomorphic anomaly equations we present here depend on the parity of $n$. 
\begin{thm}\label{thm:HAE_n_odd}
Let $n\geq{3}$ be an odd number with $n=2s+1$, and $g\geq{2}$. We have
\begin{equation*}
\frac{C_{s+1}}{(2s+1)L}\frac{\partial}{\partial A_{s}}\mathcal{F}_{g}^{\left[\mathbb{C}^{n} / \mathbb{Z}_{n}\right]}
=\frac{1}{2}\mathcal{F}_{g-1,2}^{\left[\mathbb{C}^n/ \mathbb{Z}_n\right]}\left(\phi_s,\phi_s\right)+\frac{1}{2}\sum_{i=1}^{g-1}\mathcal{F}_{g-i,1}^{\left[\mathbb{C}^n/ \mathbb{Z}_n\right]}\left(\phi_s\right)\mathcal{F}_{i,1}^{\left[\mathbb{C}^n/ \mathbb{Z}_n\right]}\left(\phi_s\right)
\end{equation*}
in $\mathbb{C}[L^{\pm{1}}][\mathfrak{S}_n][C_{s+1}].$
\end{thm}

\begin{proof}
Let $\tilde{\mathfrak{e}}\in\mathrm{E}_{\Gamma}$ be an edge of a decorated stable graph $\Gamma\in\mathrm{G}_{g,0}^{\text{Dec}}(n)$ and let $\tilde{\mathfrak{e}}\in\mathrm{E}_{\Gamma}$ be connecting two vertices $\mathfrak{v}_1$ and $\mathfrak{v}_{2}$. As described in Subsection \ref{sec:edge_action}, breaking the edge $\tilde{\mathfrak{e}}$ into two new legs $\mathfrak{l}_{\tilde{\mathfrak{e}}}$ and $\mathfrak{l}'_{\tilde{\mathfrak{e}}}$ results in new graphs. Recall that there are two possibilities: either the resulting graph is connected, or it is disconnected with two connected parts.
\begin{itemize}
    \item[(i)] If the resulting graph is connected, then it is an element of $\mathrm{G}_{g-1,2}^{\text{Dec}}(n)$ and it is underlying stable graph is an element of $\mathrm{G}_{g-1,2}$. We denote them as $\Gamma_{\tilde{\mathfrak{e}}}$ and $\Gamma_{\tilde{\mathfrak{e}}}^{\mathrm{St}}$ respectively. 
    
    \item[(ii)] If the resulting graph is disconnected, then we denote its connected components as $\Gamma_{\tilde{\mathfrak{e}}}^1\in\mathrm{G}_{g_1,1}^{\text{Dec}}(n)$ and $\Gamma_{\tilde{\mathfrak{e}}}^2\in\mathrm{G}_{g_2,1}^{\text{Dec}}(n)$ where $g=g_1+g_2$, and we denote their underlying stable graphs as $\Gamma_{\tilde{\mathfrak{e}}}^{1,\mathrm{St}}\in\mathrm{G}_{g_1,1}$ and $\Gamma_{\tilde{\mathfrak{e}}}^{2,\mathrm{St}}\in\mathrm{G}_{g_2,1}$.
\end{itemize}

By Proposition \ref{prop:contributions} and Lemma \ref{lem:oddderivativeofedge}, we observe that
\begin{equation*}
\begin{split}
\frac{\partial \mathrm{Cont}_{\Gamma}^{\mathrm{A}}(\tilde{\mathfrak{e}}) }{\partial A_s}
=&\frac{(-1)^{b_{\tilde{\mathfrak{e}}1}+b_{\tilde{\mathfrak{e}}2}}}{2s+1} \frac{\widetilde{P}_{s+1,\mathrm{p}(\mathfrak{v}_1)}^{b_{\tilde{\mathfrak{e}}1}}\widetilde{P}_{s+1,\mathrm{p}(\mathfrak{v}_2)}^{b_{\tilde{\mathfrak{e}}2}}}{\zeta^{(b_{\tilde{\mathfrak{e}}1}+s+1)\mathrm{p}(\mathfrak{v}_1)}\zeta^{(b_{\tilde{\mathfrak{e}}2}+s+1)\mathrm{p}(\mathfrak{v}_2)}}\\
=&(2s+1)\left(\frac{L^{s+1}}{K_{s+1}}\right)^2
\begin{cases}
\mathrm{Cont}_{\Gamma_{\tilde{\mathfrak{e}}}}^{\mathrm{A}}(\mathfrak{l}_{\tilde{\mathfrak{e}}})\mathrm{Cont}_{\Gamma_{\tilde{\mathfrak{e}}}}^{\mathrm{A}}(\mathfrak{l}'_{\tilde{\mathfrak{e}}})&\text{for the case (i)},\\
\mathrm{Cont}_{\Gamma_{\tilde{\mathfrak{e}}}^1}^{\mathrm{A}}(\mathfrak{l}_{\tilde{\mathfrak{e}}})\mathrm{Cont}_{\Gamma_{\tilde{\mathfrak{e}}}^2}^{\mathrm{A}}(\mathfrak{l}'_{\tilde{\mathfrak{e}}})&\text{for the case (ii)},
\end{cases}
\end{split}
\end{equation*}
where $c_{\ell(\mathfrak{l}_{\tilde{\mathfrak{e}}})}=c_{\ell(\mathfrak{l}'_{\tilde{\mathfrak{e}}})}=\mathrm{Inv}(s+1)=s$ in both cases (i) and (ii).

Using Corollary \ref{Kfunctions}, we also note that 
\begin{equation*}
\left(\frac{L^{s+1}}{K_{s+1}}\right)^2=\frac{L}{C_{s+1}}.
\end{equation*}
Then, for case (i), we easily see that we have
\begin{equation}\label{eqn:edge_deletion_cont_case_i}
\begin{split}
\operatorname{Cont}_{\Gamma_{\tilde{\mathfrak{e}}}}\left(\phi_{s},\phi_{s}\right)
=&\frac{1}{|\mathrm{Aut}(\Gamma_{\tilde{\mathfrak{e}}}^{\mathrm{St}})|} \sum_{\mathrm{A} \in \mathbb{Z}_{\geq 0}^{\mathrm{F}({\Gamma_{\tilde{\mathfrak{e}}}})}} \prod_{\mathfrak{v} \in \mathrm{V}_{\Gamma_{\tilde{\mathfrak{e}}}}} \mathrm{Cont}_{\Gamma_{\tilde{\mathfrak{e}}}}^{\mathrm{A}}(\mathfrak{v}) \prod_{\mathfrak{e} \in \mathrm{E}_{\Gamma_{\tilde{\mathfrak{e}}}}} \mathrm{Cont}_{\Gamma_{\tilde{\mathfrak{e}}}}^{\mathrm{A}}(\mathfrak{e})\prod_{\mathfrak{l} \in \mathrm{L}_{\Gamma_{\tilde{\mathfrak{e}}}}} \mathrm{Cont}_{\Gamma_{\tilde{\mathfrak{e}}}}^{\mathrm{A}}(\mathfrak{l})\\
=&\frac{\mathrm{O}_{\Gamma}^{e}}{|\mathrm{Aut}(\Gamma^{\mathrm{St}})|} \sum_{\mathrm{A} \in \mathbb{Z}_{\geq 0}^{\mathrm{F}(\Gamma)}}\frac{C_{s+1}}{(2s+1)L}\frac{\partial \mathrm{Cont}_{\Gamma}^{\mathrm{A}}(\tilde{\mathfrak{e}}) }{\partial A_s} \prod_{\mathfrak{v} \in \mathrm{V}_{\Gamma}} \mathrm{Cont}_{\Gamma}^{\mathrm{A}}(\mathfrak{v}) \prod_{\substack{
\mathfrak{e} \in \mathrm{E}_{\Gamma} \\ \mathfrak{e}\neq \tilde{\mathfrak{e}}}} \mathrm{Cont}_{\Gamma}^{\mathrm{A}}(\mathfrak{e}).
\end{split}
\end{equation}
Similarly, for case (ii), we observe the following
\begin{equation}\label{eqn:edge_deletion_cont_case_ii}
\begin{split}
\operatorname{Cont}_{\Gamma_{\tilde{\mathfrak{e}}}^1}\left(\phi_{s}\right)&\operatorname{Cont}_{\Gamma_{\tilde{\mathfrak{e}}}^2}\left(\phi_{s}\right)\\
=&\frac{1}{|\mathrm{Aut}(\Gamma_{\tilde{\mathfrak{e}}}^{1,\mathrm{St}})|} \sum_{\mathrm{A} \in \mathbb{Z}_{\geq 0}^{\mathrm{F}({\Gamma_{\tilde{\mathfrak{e}}}^1})}} \mathrm{Cont}_{\Gamma_{\tilde{\mathfrak{e}}}^1}^{\mathrm{A}}(\mathfrak{l}_{\tilde{\mathfrak{e}}})\prod_{\mathfrak{v} \in \mathrm{V}_{\Gamma_{\tilde{\mathfrak{e}}}^1}} \mathrm{Cont}_{\Gamma_{\tilde{\mathfrak{e}}}^1}^{\mathrm{A}}(\mathfrak{v}) \prod_{\mathfrak{e} \in \mathrm{E}_{\Gamma_{\tilde{\mathfrak{e}}}^1}} \mathrm{Cont}_{\Gamma_{\tilde{\mathfrak{e}}}^1}^{\mathrm{A}}(\mathfrak{e})\\
\times&\frac{1}{|\mathrm{Aut}(\Gamma_{\tilde{\mathfrak{e}}}^{2,\mathrm{St}})|} \sum_{\mathrm{A} \in \mathbb{Z}_{\geq 0}^{\mathrm{F}({\Gamma_{\tilde{\mathfrak{e}}}^2})}} \mathrm{Cont}_{\Gamma_{\tilde{\mathfrak{e}}}^2}^{\mathrm{A}}(\mathfrak{l}'_{\tilde{\mathfrak{e}}}) \prod_{v \in \mathrm{V}_{\Gamma_{\tilde{\mathfrak{e}}}^2}} \mathrm{Cont}_{\Gamma_{\tilde{\mathfrak{e}}}^2}^{\mathrm{A}}(\mathfrak{v}) \prod_{\mathfrak{e} \in \mathrm{E}_{\Gamma_{\tilde{\mathfrak{e}}}^2}} \mathrm{Cont}_{\Gamma_{\tilde{\mathfrak{e}}}^2}^{\mathrm{A}}(\mathfrak{e})\\
=&\frac{\mathrm{O}_{\Gamma}^{e}}{|\mathrm{Aut}(\Gamma^{\mathrm{St}})|}  \sum_{\mathrm{A} \in \mathbb{Z}_{\geq 0}^{\mathrm{F}(\Gamma)}}\frac{C_{s+1}}{(2s+1)L}\frac{\partial \mathrm{Cont}_{\Gamma}^{\mathrm{A}}(\tilde{\mathfrak{e}}) }{\partial A_s} \prod_{\mathfrak{v} \in \mathrm{V}_{\Gamma}} \mathrm{Cont}_{\Gamma}^{\mathrm{A}}(\mathfrak{v}) \prod_{\substack{
\mathfrak{e} \in \mathrm{E}_{\Gamma} \\ \mathfrak{e}\neq \tilde{\mathfrak{e}}}} \mathrm{Cont}_{\Gamma}^{\mathrm{A}}(\mathfrak{e}).
\end{split}
\end{equation}
By Corollary \ref{cor:VertexEdgeCont}, we have the following vanishing result
\begin{equation*}
\frac{\partial \mathrm{Cont}_{\Gamma}^{\mathrm{A}}(\mathfrak{v}) }{\partial A_s}=0
\end{equation*}
for any vertex $\mathfrak{v}\in\mathrm{V}_{\Gamma} $. This vanishing result gives us:
\begin{equation}\label{eqn:implication_of_vanishing}
\begin{split}
\frac{\partial \mathrm{Cont}_{\Gamma}}{\partial A_s}
=&\frac{1}{|\mathrm{Aut}(\Gamma^{\mathrm{St}})|} \sum_{\mathrm{A} \in \mathbb{Z}_{\geq 0}^{\mathrm{F}(\Gamma)}} \prod_{\mathfrak{v} \in \mathrm{V}_{\Gamma}} \mathrm{Cont}_{\Gamma}^{\mathrm{A}}(\mathfrak{v})\frac{\partial}{\partial A_s}\left(\prod_{\mathfrak{e}\in \mathrm{E}_{\Gamma}} \mathrm{Cont}_{\Gamma}^{\mathrm{A}}(\mathfrak{e})\right)\\
=&\frac{1}{|\mathrm{Aut}(\Gamma^{\mathrm{St}})|} \sum_{\mathrm{A} \in \mathbb{Z}_{\geq 0}^{\mathrm{F}(\Gamma)}} \prod_{\mathfrak{v} \in \mathrm{V}_{\Gamma}} \mathrm{Cont}_{\Gamma}^{\mathrm{A}}(\mathfrak{v})\prod_{\substack{
\mathfrak{e} \in \mathrm{E}_{\Gamma} \\ \mathfrak{e}\neq \tilde{\mathfrak{e}}}} \mathrm{Cont}_{\Gamma}^{\mathrm{A}}(\mathfrak{e})\frac{\partial \mathrm{Cont}_{\Gamma}^{\mathrm{A}}(\tilde{\mathfrak{e}}) }{\partial A_s}\\
=&\sum_{\tilde{\mathfrak{e}}\in \mathrm{E}_{\Gamma}}\frac{1}{|\mathrm{Aut}(\Gamma^{\mathrm{St}})|} \sum_{\mathrm{A} \in \mathbb{Z}_{\geq 0}^{\mathrm{F}(\Gamma)}}\frac{\partial \mathrm{Cont}_{\Gamma}^{\mathrm{A}}(\tilde{\mathfrak{e}}) }{\partial A_s} \prod_{\mathfrak{v} \in \mathrm{V}_{\Gamma}} \mathrm{Cont}_{\Gamma}^{\mathrm{A}}(\mathfrak{v})\prod_{\substack{
\mathfrak{e} \in \mathrm{E}_{\Gamma} \\ \mathfrak{e}\neq \tilde{\mathfrak{e}}}} \mathrm{Cont}_{\Gamma}^{\mathrm{A}}(\mathfrak{e}).
\end{split}
\end{equation}

By the Subsection \ref{sec:edge_action}, we complete the proof after summing the equations (\ref{eqn:edge_deletion_cont_case_i}), (\ref{eqn:edge_deletion_cont_case_ii}), (\ref{eqn:implication_of_vanishing}) over all decorated stable graphs. The ratio $\mathrm{O}_{\Gamma}^{\mathfrak{e}}$ is taken care of by Lemma \ref{lem:breaking_an_edge} and Remark \ref{rem:isomorphic_graph_contribution}. The reason we have a factor of $\frac{1}{2}$ on the right-hand side of the holomorphic anomaly equation is compensation  due to the following:
\begin{itemize}
    \item In \textit{Case A}, $\Gamma_{\tilde{\mathfrak{e}},(1,2)}$ and $\Gamma_{\tilde{\mathfrak{e}},(2,1)}$ are not isomorphic for the connected case. Yet, their contributions are the same by Remark \ref{rem:nonisomorphic_graph_contributions}. Similarly, for the disconnected case, $\Gamma_{\tilde{\mathfrak{e}}}^1$ and $\Gamma_{\tilde{\mathfrak{e}}}^1$ are not isomorphic. But, the products of the contributions appear twice. So, there is a double counting both in connected and disconnected cases.
    \item In \textit{Case B}, we do not have a double counting but we have $\mathrm{O}_{\Gamma}^{\mathfrak{e}}=2\vert \mathrm{Orb_{\Gamma}(\mathfrak{e})} \vert$.
\end{itemize}
\end{proof}

\begin{thm}\label{thm:HAE_n_even}
Let $n\geq{4}$ be an even number with $n=2s$, and $g\geq{2}$. We have
\begin{equation*}
\frac{C_{s+1}}{2sL}\frac{\partial}{\partial A_{s-1}}\mathcal{F}_{g}^{\left[\mathbb{C}^{n} / \mathbb{Z}_{n}\right]}
=\mathcal{F}_{g-1,2}^{\left[\mathbb{C}^n/ \mathbb{Z}_n\right]}\left(\phi_{s-1},\phi_s\right)+\sum_{i=1}^{g-1}\mathcal{F}_{g-i,1}^{\left[\mathbb{C}^n/ \mathbb{Z}_n\right]}\left(\phi_{s-1}\right)\mathcal{F}_{i,1}^{\left[\mathbb{C}^n/ \mathbb{Z}_n\right]}\left(\phi_s\right)
\end{equation*}
in $\mathbb{C}[L^{\pm{1}}][\mathfrak{S}_n][C_{s+1}].$
\end{thm}

\begin{proof}
The proof is similar to the proof of the odd case. By Proposition \ref{prop:contributions} and Lemma \ref{lem:evenderivativeofedge}, we observe that
\begin{equation}\label{eq:Asderivative_of_edge}
\begin{split}
\frac{\partial}{\partial A_{s-1}}\mathrm{Cont}_{\Gamma}^{\mathrm{A}}(\tilde{\mathfrak{e}})
=&\frac{(-1)^{b_{\mathfrak{e}1}+b_{\mathfrak{e}2}}}{2s}\left(\frac{\widetilde{P}_{s+1,\mathrm{p}(\mathfrak{v}_1)}^{b_{\mathfrak{e}1}}\widetilde{P}_{s,\mathrm{p}(\mathfrak{v}_2)}^{b_{\mathfrak{e}2}}}{\zeta^{(b_{\mathfrak{e}1}+s+1)\mathrm{p}(\mathfrak{v}_1)}\zeta^{(b_{\mathfrak{e}2}+s)\mathrm{p}(\mathfrak{v}_2)}}+\frac{\widetilde{P}_{s,\mathrm{p}(\mathfrak{v}_1)}^{b_{\mathfrak{e}1}}\widetilde{P}_{s+1,\mathrm{p}(\mathfrak{v}_2)}^{b_{\mathfrak{e}2}}}{\zeta^{(b_{\mathfrak{e}1}+s)\mathrm{p}(\mathfrak{v}_1)}\zeta^{(b_{\mathfrak{e}2}+s+1)\mathrm{p}(\mathfrak{v}_2)}}\right)\\
=&(2s)\frac{L^sL^{s+1}}{K_sK_{s+1}}
\begin{cases}
\mathrm{Cont}_{\Gamma_{\tilde{\mathfrak{e}}}}^{\mathrm{A}}(\mathfrak{l}_{\tilde{\mathfrak{e}}})\mathrm{Cont}_{\Gamma_{\tilde{\mathfrak{e}}}}^{\mathrm{A}}(\mathfrak{l}'_{\tilde{\mathfrak{e}}})+\mathrm{Cont}_{\Gamma_{\tilde{\mathfrak{e}}}}^{\mathrm{A}}(\mathfrak{l}_{\tilde{\mathfrak{e}}})\mathrm{Cont}_{\Gamma_{\tilde{\mathfrak{e}}}}^{\mathrm{A}}(\mathfrak{l}'_{\tilde{\mathfrak{e}}})&\text{for (i)},\\
\mathrm{Cont}_{\Gamma_{\tilde{\mathfrak{e}}}^1}^{\mathrm{A}}(\mathfrak{l}_{\tilde{\mathfrak{e}}})\mathrm{Cont}_{\Gamma_{\tilde{\mathfrak{e}}}^2}^{\mathrm{A}}(\mathfrak{l}'_{\tilde{\mathfrak{e}}})+\mathrm{Cont}_{\Gamma_{\tilde{\mathfrak{e}}}^1}^{\mathrm{A}}(\mathfrak{l}_{\tilde{\mathfrak{e}}})\mathrm{Cont}_{\Gamma_{\tilde{\mathfrak{e}}}^2}^{\mathrm{A}}(\mathfrak{l}'_{\tilde{\mathfrak{e}}})&\text{for (ii).}
\end{cases}
\end{split}
\end{equation}

Using Corollary \ref{Kfunctions}, we note that 
\begin{equation*}
\frac{L^sL^{s+1}}{K_sK_{s+1}}=\frac{L}{C_{s+1}}.
\end{equation*}
The rest of the proof is identical to the proof of Theorem \ref{thm:HAE_n_odd}
; however, this time, we obtain two products of leg contributions for both cases (i) and (ii) in equation (\ref{eq:Asderivative_of_edge}). For the both cases (i) and (ii), $c_{\ell(\mathfrak{l}_{\tilde{\mathfrak{e}}})}=\mathrm{Inv}(s+1)=s-1$ and $c_{\ell(\mathfrak{l}'_{\tilde{\mathfrak{e}}})}=\mathrm{Inv}(s)=s$ for the first product and $c_{\ell(\mathfrak{l}_{\tilde{\mathfrak{e}}})}=\mathrm{Inv}(s)=s$ and $c_{\ell(\mathfrak{l}'_{\tilde{\mathfrak{e}}})}=\mathrm{Inv}(s+1)=s-1$ for the second product. When we sum over all possible decorated stable graphs, this will result in a double counting compensating the earlier one, see the end of proof of Theorem \ref{thm:HAE_n_odd}. For this reason, we do not have the factor $\frac{1}{2}$ on the right-hand side of holomorphic anomaly equations in Theorem \ref{thm:HAE_n_even}.
\end{proof}

\subsection{Holomorphic anomaly equations with insertions}\label{sec:HAE_insertion}
Oberdieck and Pixton proved holomorphic anomaly equations with insertions for elliptic curves \cite{OP}. Motivated from their work, Lho and Pandharipande proved holomorphic anomaly equations with insertions for the total space $K\mathbb{P}^2$ of the canonical bundle of $\mathbb{P}^2$ \cite{lho-p}. Following a similar path, in this subsection, we extend our holomorphic anomaly equations to potentials with arbitrary insertions.

For $0\leq i \leq n-1$, let $r_i\geq 0$ such that $r_0+\cdots +r_{n-1}=m$. We can alternatively define the Gromov-Witten potential as

\begin{align*}
&\mathcal{F}_{g, m}^{\CnZn}[r_0,\ldots,r_{n-1}]\\
&=\sum_{d=0}^{\infty} \frac{\Theta^{d}}{d !} \int_{\left[\overline{M}_{g, m+d}^{\mathrm{orb}}\left(\left[\mathbb{C}^{n} / \mathbb{Z}_{n}\right], 0\right)\right]^{v i r}} \prod_{i=1}^{d_0} \mathrm{ev}_{i}^{*}\left(\phi_0\right) \prod_{i=d_0+1}^{d_1} \mathrm{ev}_{i}^{*}\left(\phi_1\right)\cdots \prod_{i=d_{n-2}+1}^{d_{n-1}} \mathrm{ev}_{i}^{*}\left(\phi_{n-1}\right) \prod_{i=m+1}^{m+d} \mathrm{ev}_{i}^{*}\left(\phi_{1}\right).    
\end{align*}
where
\begin{equation*}
d_i=\sum_{j=0}^{i}r_j\,.
\end{equation*}

Let $\pi$ be the morphism to the moduli space of stable curves determined by the domain, $$\pi:\overline{M}_{g, k}^{\mathrm{orb}}\left(\left[\mathbb{C}^{n} / \mathbb{Z}_{n}\right], 0\right)\rightarrow \overline{M}_{g, k}.$$
In a similar vein, we will define the Gromov-Witten potentials with certain types of ancestor insertions. When $n=2s+1$ is odd, define the Gromov-Witten potential with the ancestor insertions
$$\mathcal{F}_{g, m}^{\CnZn}[r_0,\ldots,r_{n-1};\delta]\quad \text{with}\quad r_0+\cdots r_{n-1}+\delta=m,$$ where $\delta$ keeps track of the number of the insertions of the form $$\pi^{\star}(\psi_i)\cdot \mathrm{ev}^{\star}_i(\phi_{s}).$$ Similarly, when $n=2s$ is even, define the Gromov-Witten potential with the ancestor insertions
$$\mathcal{F}_{g, m}^{\CnZn}[r_0,\ldots,r_{n-1};\delta_1,\delta_2]\quad \text{with}\quad r_0+\cdots r_{n-1}+\delta_1+\delta_2=m,$$
where $\delta_1$ keeps track of the number of the insertions of the form $$\pi^{\star}(\psi_i)\cdot \mathrm{ev}^{\star}_i(\phi_{s}),$$ and $\delta_2$ keeps track of the number of the insertions of the form $$\pi^{\star}(\psi_i)\cdot \mathrm{ev}^{\star}_i(\phi_{s-1}).$$

There are formulas for the above Gromov-Witten potentials with ancestor insertions similar to equation (\ref{eqn:formula_Fg}). Indeed, Gromov-Witten potentials $\mathcal{F}_{g, m}^{\left[\mathbb{C}^{n} / \mathbb{Z}_{n}\right]}\left(\phi_{c_{1}}, \ldots, \phi_{c_{m}}\right)$ are quantities of the form 
\begin{equation*}
\int_{\overline{M}_{g,m}}\Omega_{g,m+d}(\phi_{c_1}\otimes...\otimes\phi_{c_m}).    
\end{equation*}
Proposition \ref{prop:contributions} is obtained by unraveling the Givental-Teleman description of the (shifted) CohFT $\Omega$ given by its $R$-matrix acting on its topological part. The potentials $\mathcal{F}_{g, m}^{\CnZn}[r_0,\ldots,r_{n-1};\delta]$ and $\mathcal{F}_{g, m}^{\CnZn}[r_0,\ldots,r_{n-1};\delta_1,\delta_2]$ are quantities of the forms:
\begin{equation*}
\int_{\overline{M}_{g,m}}\Omega_{g,m}(\phi_0^{\otimes r_0}\otimes...\otimes\phi_{n-1}^{\otimes r_{n-1}}\otimes \phi_s^{\otimes \delta})\prod_{i=m-\delta+1}^m \psi_i,   
\end{equation*}
and
\begin{equation*}
\int_{\overline{M}_{g,m}}\Omega_{g,m}(\phi_0^{\otimes r_0}\otimes...\otimes\phi_{n-1}^{\otimes r_{n-1}}\otimes \phi_s^{\otimes \delta_1}\otimes \phi_{s-1}^{\otimes \delta_2})\prod_{i=m-(\delta_1+\delta_2)+1}^{m} \psi_i.   
\end{equation*}
Thus they are given by formulas similar to equation (\ref{eqn:formula_Fg}), with contributions given by Proposition \ref{prop:contributions} modified appropriately to include $\psi$-class insertions. 

Next, we present the analog of the Proposition \ref{prop:contributions} for the potentials $\mathcal{F}_{g, m}^{\CnZn}[r_0,\ldots,r_{n-1};\delta]$ and $\mathcal{F}_{g, m}^{\CnZn}[r_0,\ldots,r_{n-1};\delta_1,\delta_2]$.

Introduce indices $c_1,...,c_{m-\Delta}$ by setting
\begin{equation*}
(\phi_{c_{1}}, \ldots, \phi_{c_{m-\Delta}})=(\underbrace{\phi_0,\ldots,\phi_0}_{r_0-\text{times}},\underbrace{\phi_1,\ldots,\phi_1}_{r_1-\text{times}},\ldots,\underbrace{\phi_i,\ldots,\phi_i}_{r_i-\text{times}},\ldots,\underbrace{\phi_{n-1},\ldots,\phi_{n-1}}_{r_{n-1}-\text{times}})
\end{equation*}
where
\begin{equation*}
\Delta=
\begin{cases}
\delta\quad &\text{if }n\text{ is odd,}\\
\delta_1+\delta_2\quad &\text{if }n\text{ is even.}
\end{cases}
\end{equation*}

Now, consider the case when $n=2s+1$ is odd. The potential $\mathcal{F}_{g, m}^{\CnZn}[r_0,\ldots,r_{n-1};\delta]$ can be written as the graph sum
\begin{equation}
\mathcal{F}_{g, m}^{\CnZn}[r_0,\ldots,r_{n-1};\delta]=\sum_{\Gamma\in\mathrm{G}_{g,m}^{\text{Dec}}(n)}\operatorname{Cont}_{\Gamma}\left(\phi_{c_{1}}, \ldots, \phi_{c_{m-\delta}};\delta\right).
\end{equation}
Let $\mathrm{l}(\mathfrak{v})$, the number of legs entering the vertex $\mathfrak{v}$, be decomposed as
$$\mathrm{l}(\mathfrak{v})=\mathrm{l}_0(\mathfrak{v})+\mathrm{l}_{\delta}(\mathfrak{v})$$
where $\mathrm{l}_0(\mathfrak{v})$ is the number of legs with $\phi_{c_i}$ insertions, and $\mathrm{l}_{\delta}(\mathfrak{v})$ is the number of legs with insertions $\pi^{\star}(\psi_i)\cdot \mathrm{ev}^{\star}_i(\phi_{s})$. Then, the following is the analog of Proposition \ref{prop:contributions}.

\begin{prop}\label{prop:insertions_cont_odd}
For each decorated stable graph $\Gamma\in\mathrm{G}_{g,m}^{\text{Dec}}(n)$, the associated contribution is given by
\begin{equation*}
\operatorname{Cont}_{\Gamma}\left(\phi_{c_{1}}, \ldots, \phi_{c_{m-\delta}};\delta\right)=\frac{1}{|\mathrm{Aut}(\Gamma^{\mathrm{St}})|} \sum_{\mathrm{A} \in \mathbb{Z}_{\geq 0}^{\mathrm{F}(\Gamma)}} \prod_{\mathfrak{v} \in \mathrm{V}_{\Gamma}} \mathrm{Cont}_{\Gamma}^{\mathrm{A}}(\mathfrak{v}) \prod_{\mathfrak{e}\in \mathrm{E}_{\Gamma}} \mathrm{Cont}_{\Gamma}^{\mathrm{A}}(\mathfrak{e}) \prod_{\mathfrak{l} \in \mathrm{L}_{\Gamma}} \mathrm{Cont}_{\Gamma}^{\mathrm{A}}(\mathfrak{l})
\end{equation*}
where $\mathrm{F}(\Gamma)=\left\vert\mathrm{H}_{\Gamma}\right\vert$. Here, $\mathrm{Cont}_{\Gamma}^{\mathrm{A}}(\mathfrak{v})$, $\mathrm{Cont}_{\Gamma}^{\mathrm{A}}(\mathfrak{e})$, and $\mathrm{Cont}_{\Gamma}^{\mathrm{A}}(\mathfrak{l})$ are the {\em vertex}, {\em edge} and {\em leg} contributions with flag $\mathrm{A}-$values$(a_1,\ldots,a_m,b_{m+1},\ldots,b_{\left\vert\mathrm{H}_{\Gamma}\right\vert})$ respectively. The leg and vertex contributions are the same as Proposition \ref{prop:contributions}. The vertex contribution $\mathrm{Cont}_{\Gamma}^{\mathrm{A}}(\mathfrak{v})$ is given by
\begin{equation*}
\begin{split}
    \mathrm{Cont}&_{\Gamma}^{\mathrm{A}}(\mathfrak{v})
    =\sum_{k \geq 0} \frac{g({e}_{\mathrm{p}(\mathfrak{v})},{e}_{\mathrm{p}(\mathfrak{v})})^{-\frac{2\mathrm{g}(\mathfrak{v})-2+\mathrm{n}(\mathfrak{v})+k}{2}}}{k !}\\
    &\times\int_{\overline{M}_{\mathrm{g}(\mathfrak{v}),\mathrm{n}(\mathfrak{v})+k}}\psi_1^{a_{\mathfrak{v}1}}\cdots\psi_{\mathrm{l}_0(\mathfrak{v})}^{a_{\mathfrak{v}\mathrm{l}_0(\mathfrak{v})}}
    \psi_{\mathrm{l}_0(\mathfrak{v})+1}^{a_{\mathfrak{v}(\mathrm{l}_0(\mathfrak{v})+1)}+1}\cdots
    \psi_{\mathrm{l}(\mathfrak{v})}^{a_{\mathfrak{v}\mathrm{l}(\mathfrak{v})}+1}
    \psi_{\mathrm{l}(\mathfrak{v})+1}^{b_{\mathfrak{v}1}}\cdots\psi_{\mathrm{n}(\mathfrak{v})}^{b_{\mathfrak{v}\mathrm{h}(\mathfrak{v})}}t_{\mathrm{p}(\mathfrak{v})}(\psi_{\mathrm{n}(\mathfrak{v})+1})\cdots t_{\mathrm{p}(\mathfrak{v})}(\psi_{\mathrm{n}(\mathfrak{v})+k})
\end{split}
\end{equation*}
where $t_{\mathrm{p}(\mathfrak{v})}$ is the same as in Proposition \ref{prop:contributions}.
\end{prop}
To be more precise about the statement of Proposition \ref{prop:insertions_cont_odd}, we add the following remark which also explains the shifts of the powers of $\psi$-classes in the vertex contributions.
\begin{rem}
In Proposition \ref{prop:insertions_cont_odd}, the leg contributions with $\phi_{c_i}$ insertions are the same as in Proposition \ref{prop:contributions}. The leg contributions corresponding to $\pi^{\star}(\psi_i)\cdot \mathrm{ev}^{\star}_i(\phi_{s-1})$ insertions are the same as the leg contributions with $\phi_{s-1}$ insertions in Proposition \ref{prop:contributions} and the effect of the $\psi$-class is handled by the contribution of the vertex that the leg is attached to.
\end{rem}

Similarly, when $n=2s$ is even, the potential $\mathcal{F}_{g, m}^{\CnZn}[r_0,\ldots,r_{n-1};\delta_1,\delta_2]$ can be written as a graph sum
\begin{equation}
\mathcal{F}_{g, m}^{\CnZn}[r_0,\ldots,r_{n-1};\delta_1,\delta_2]=\sum_{\Gamma\in\mathrm{G}_{g,m}^{\text{Dec}}(n)}\operatorname{Cont}_{\Gamma}\left(\phi_{c_{1}}, \ldots, \phi_{c_{m-\delta}};\delta_1,\delta_2\right),
\end{equation}
Let $\mathrm{l}(\mathfrak{v})$, the number of legs entering the vertex $\mathfrak{v}$, be decomposed as
$$\mathrm{l}(\mathfrak{v})=\mathrm{l}_0(\mathfrak{v})+\mathrm{l}_{\delta_1}(\mathfrak{v})+\mathrm{l}_{\delta_2}(\mathfrak{v})$$
where $\mathrm{l}_0(\mathfrak{v})$ is the number of legs with $\phi_{c_i}$ insertions, and $\mathrm{l}_{\delta_1}(\mathfrak{v})$ is the number of legs with insertions $\pi^{\star}(\psi_i)\cdot \mathrm{ev}^{\star}_i(\phi_{s})$, and $\mathrm{l}_{\delta_2}(\mathfrak{v})$ is the number of legs with insertions $\pi^{\star}(\psi_i)\cdot \mathrm{ev}^{\star}_i(\phi_{s-1})$. We also set $$\mathrm{l}_1(\mathfrak{v})=\mathrm{l}_0(\mathfrak{v})+\mathrm{l}_{\delta_1}(\mathfrak{v}).$$
Then, the following is the analog of Proposition \ref{prop:contributions}.

\begin{prop}\label{prop:insertions_cont_even}
For each decorated stable graph $\Gamma\in\mathrm{G}_{g,m}^{\text{Dec}}(n)$, the associated contribution is given by
\begin{equation*}
\operatorname{Cont}_{\Gamma}\left(\phi_{c_{1}}, \ldots, \phi_{c_{m-\delta}};\delta_1,\delta_2\right)=\frac{1}{|\mathrm{Aut}(\Gamma^{\mathrm{St}})|} \sum_{\mathrm{A} \in \mathbb{Z}_{\geq 0}^{\mathrm{F}(\Gamma)}} \prod_{\mathfrak{v} \in \mathrm{V}_{\Gamma}} \mathrm{Cont}_{\Gamma}^{\mathrm{A}}(\mathfrak{v}) \prod_{\mathfrak{e}\in \mathrm{E}_{\Gamma}} \mathrm{Cont}_{\Gamma}^{\mathrm{A}}(\mathfrak{e}) \prod_{\mathfrak{l} \in \mathrm{L}_{\Gamma}} \mathrm{Cont}_{\Gamma}^{\mathrm{A}}(\mathfrak{l})
\end{equation*}
where $\mathrm{F}(\Gamma)=\left\vert\mathrm{H}_{\Gamma}\right\vert$. Here, $\mathrm{Cont}_{\Gamma}^{\mathrm{A}}(\mathfrak{v})$, $\mathrm{Cont}_{\Gamma}^{\mathrm{A}}(\mathfrak{e})$, and $\mathrm{Cont}_{\Gamma}^{\mathrm{A}}(\mathfrak{l})$ are the {\em vertex}, {\em edge} and {\em leg} contributions with flag $\mathrm{A}-$values$(a_1,\ldots,a_m,b_{m+1},\ldots,b_{\left\vert\mathrm{H}_{\Gamma}\right\vert})$ respectively. The edge and leg contributions are the same as in Proposition \ref{prop:contributions}. The vertex contribution $\mathrm{Cont}_{\Gamma}^{\mathrm{A}}(\mathfrak{v})$ is given by
\begin{equation*}
\begin{split}
    \mathrm{Cont}&_{\Gamma}^{\mathrm{A}}(\mathfrak{v})
    =\sum_{k \geq 0} \frac{g({e}_{\mathrm{p}(\mathfrak{v})},{e}_{\mathrm{p}(\mathfrak{v})})^{-\frac{2\mathrm{g}(\mathfrak{v})-2+\mathrm{n}(\mathfrak{v})+k}{2}}}{k !}\\
    &\times\int_{\overline{M}_{\mathrm{g}(\mathfrak{v}),\mathrm{n}(\mathfrak{v})+k}}\psi_1^{a_{\mathfrak{v}1}}\cdots\psi_{\mathrm{l}_0(\mathfrak{v})}^{a_{\mathfrak{v}\mathrm{l}_0(\mathfrak{v})}}
    \psi_{\mathrm{l}_0(\mathfrak{v})+1}^{a_{\mathfrak{v}(\mathrm{l}_0(\mathfrak{v})+1)}+1}\cdots
    \psi_{\mathrm{l}_1(\mathfrak{v})}^{a_{\mathfrak{v}\mathrm{l}_1(\mathfrak{v})}+1}
    \psi_{\mathrm{l}_1(\mathfrak{v})+1}^{a_{\mathfrak{v}(\mathrm{l}_1(\mathfrak{v})+1)}+1}\cdots
    \psi_{\mathrm{l}(\mathfrak{v})}^{a_{\mathfrak{v}\mathrm{l}(\mathfrak{v})}+1}\\
    &\hspace{10em}\cdot\psi_{\mathrm{l}(\mathfrak{v})+1}^{b_{\mathfrak{v}1}}\cdots\psi_{\mathrm{n}(\mathfrak{v})}^{b_{\mathfrak{v}\mathrm{h}(\mathfrak{v})}}t_{\mathrm{p}(\mathfrak{v})}(\psi_{\mathrm{n}(\mathfrak{v})+1})\cdots t_{\mathrm{p}(\mathfrak{v})}(\psi_{\mathrm{n}(\mathfrak{v})+k})
\end{split}
\end{equation*}
where $t_{\mathrm{p}(\mathfrak{v})}$ is the same as in Proposition \ref{prop:contributions}.
\end{prop}
To be more precise about the leg contributions, we add the following remark.
\begin{rem}
In Proposition \ref{prop:insertions_cont_even}, the leg contributions with $\phi_{c_i}$ insertions are the same as in Proposition \ref{prop:contributions}. The leg contributions corresponding to $\pi^{\star}(\psi_i)\cdot \mathrm{ev}^{\star}_i(\phi_{j})$ insertions with $j=s,s-1$ are the same as the leg contributions with $\phi_{j}$ insertions in Proposition \ref{prop:contributions} and the effect of the $\psi$-class is handled by the contribution of the vertex that the leg is attached to.
\end{rem}

Furthermore, for the potentials $\mathcal{F}_{g, m}^{\CnZn}[r_0,\ldots,r_{n-1};\delta]$ and $\mathcal{F}_{g, m}^{\CnZn}[r_0,\ldots,r_{n-1};\delta_1,\delta_2]$, the statement of Corollary \ref{cor:VertexEdgeCont} remains true. In other words, the vertex, edge, and leg contributions lie in the same rings, and hence, we get

\begin{equation*}
\mathcal{F}_{g, m}^{\CnZn}[r_0,\ldots,r_{n-1};\delta],\,\mathcal{F}_{g, m}^{\CnZn}[r_0,\ldots,r_{n-1};\delta_1,\delta_2] \in\mathds{F}_{\CnZn}.
\end{equation*}

When $n$ is odd, the holomorphic anomaly equations with insertions for $\mathcal{F}_{g, m}^{\CnZn}[r_0,\ldots,r_{n-1};\delta]$ are given by the following result.

\begin{thm}\label{thm:HAE_Odd_Insertions} Let $n=2s+1\geq{3}$. Then, in the stable range $2g-2+m>0$, we have
\begin{equation*}
\begin{aligned}
&\frac{C_{s+1}}{(2s+1)L}\frac{\partial}{\partial A_{s}}\mathcal{F}_{g, m}^{\CnZn}[r_0,\ldots,r_{n-1}]\\
=&\frac{1}{2}\mathcal{F}_{g-1,m+2}^{\CnZn}[r_0,\ldots,r_{s-1},r_s+2,r_{s+1},\ldots, r_{n-1}]\\
&+\frac{1}{2}\sum_{\substack{g_1+g_2=g \\m_1+m_2=m \\ \bar{r}_i+\tilde{r}_i=r_i}}\mathcal{F}_{g_1,m_1+1}^{\CnZn}[\bar{r}_0,\ldots,\bar{r}_{s-1},\bar{r}_s+1,\bar{r}_{s+1},\ldots, \bar{r}_{n-1}]\mathcal{F}_{g_2,m_2+1}^{\CnZn}[\tilde{r}_0,\ldots,\tilde{r}_{s-1},\tilde{r}_s+1,\tilde{r}_{s+1},\ldots, \tilde{r}_{n-1}]\\
&-\frac{r_{s+1}}{2s+1}\mathcal{F}_{g,m}^{\CnZn}[r_0,\ldots,r_s,r_{s+1}-1,r_{s+2},\ldots, r_{n-1};1].
\end{aligned}
\end{equation*}
\end{thm}

\begin{proof}
Let $\Gamma\in\mathrm{G}_{g,m}^{\text{Dec}}(n)$ be a decorated stable graph and  $\mathrm{Cont}_{\Gamma}[r_0,\ldots,r_{n-1}]$  be its contribution to the potential $\mathcal{F}_{g, m}^{\CnZn}[r_0,\ldots,r_{n-1}]$. Then, its partial derivative with respect to $A_s$ is

\begin{equation}
\begin{split}
&\frac{\partial \mathrm{Cont}_{\Gamma}[r_0,\ldots,r_{n-1}]}{\partial A_s}\\
=&\underbrace{\frac{1}{|\mathrm{Aut}(\Gamma^{\mathrm{st}})|} \sum_{\mathrm{A} \in \mathbb{Z}_{\geq 0}^{\mathrm{F}(\Gamma)}} \prod_{\mathfrak{v} \in \mathrm{V}_{\Gamma}} \mathrm{Cont}_{\Gamma}^{\mathrm{A}}(\mathfrak{v})\left(\frac{\partial}{\partial A_s}\prod_{\mathfrak{e}\in \mathrm{E}_{\Gamma}} \mathrm{Cont}_{\Gamma}^{\mathrm{A}}(\mathfrak{e})\right)\prod_{\mathfrak{l} \in \mathrm{L}_{\Gamma}}\mathrm{Cont}_{\Gamma}^{\mathrm{A}}(\mathfrak{l})}_{=(\star)}\\
+&\underbrace{\frac{1}{|\mathrm{Aut}(\Gamma^{\mathrm{st}})|} \sum_{\mathrm{A} \in \mathbb{Z}_{\geq 0}^{\mathrm{F}(\Gamma)}} \prod_{\mathfrak{v} \in \mathrm{V}_{\Gamma}} \mathrm{Cont}_{\Gamma}^{\mathrm{A}}(\mathfrak{v})\prod_{\mathfrak{e}\in \mathrm{E}_{\Gamma}} \mathrm{Cont}_{\Gamma}^{\mathrm{A}}(\mathfrak{e})\left(\frac{\partial}{\partial A_s}\prod_{\mathfrak{l} \in \mathrm{L}_{\Gamma}}\mathrm{Cont}_{\Gamma}^{\mathrm{A}}(\mathfrak{l})\right)}_{=(\star \star)}
\end{split}
\end{equation}
due to the vanishing result
\begin{equation*}
\frac{\partial \mathrm{Cont}_{\Gamma}^{\mathrm{A}}(\mathfrak{v}) }{\partial A_s}=0.
\end{equation*}

A reasoning similar to the proof of Theorem \ref{thm:HAE_n_odd} shows that $(\star)$ yields the terms 
\begin{equation*}
\begin{aligned}
&\frac{1}{2}\mathcal{F}_{g-1,m+2}^{\CnZn}[r_0,\ldots,r_{s-1},r_s+2,r_{s+1},\ldots, r_{n-1}]\\
&+\frac{1}{2}\sum_{\substack{g_1+g_2=g \\m_1+m_2=m \\ \bar{r}_i+\tilde{r}_i=r_i}}\mathcal{F}_{g_1,m_1+1}^{\CnZn}[\bar{r}_0,\ldots,\bar{r}_{s-1},\bar{r}_s+1,\bar{r}_{s+1},\ldots, \bar{r}_{n-1}]\mathcal{F}_{g_2,m_2+1}^{\CnZn}[\tilde{r}_0,\ldots,\tilde{r}_{s-1},\tilde{r}_s+1,\tilde{r}_{s+1},\ldots, \tilde{r}_{n-1}]
\end{aligned}
\end{equation*}
that appear on the right-hand side of the holomorphic anomaly equations in Theorem \ref{thm:HAE_Odd_Insertions}. The crucial difference is that, when breaking the decorated stable graph $\Gamma\in\mathrm{G}_{g,m}^{\text{Dec}}(n)$ at an edge $\tilde{\mathfrak{e}}$, we either get a decorated stable graph $\Gamma_{\tilde{\mathfrak{e}}}\in\mathrm{G}_{g-1,m+2}^{\text{Dec}}(n)$ or a graph with connected components $\Gamma_{\tilde{\mathfrak{e}}}^1\in\mathrm{G}_{g_1,m_1+1}^{\text{Dec}}(n)$ and $\Gamma_{\tilde{\mathfrak{e}}}^2\in\mathrm{G}_{g_2,m_2+1}^{\text{Dec}}(n)$ where $g=g_1+g_2$ and $m_1+m_2=m$.

Now we will analyze the effect of the term $(\star \star)$ more closely. By Lemma \ref{lem:oddderivativeflatness}, we obtain the following :
\begin{equation*}
\begin{aligned}
\frac{C_{s+1}}{(2s+1)L}\frac{\partial \mathrm{Cont}_{\Gamma}^{\mathrm{A}}(\mathfrak{l}) }{\partial A_s}
&=\frac{C_{s+1}}{(2s+1)L}\frac{(-1)^{a_{\ell(\mathfrak{l})}}}{n}\frac{K_{\mathrm{Inv}(c_{\ell(\mathfrak{l})})}}{L^{\mathrm{Inv}(c_{\ell(\mathfrak{l})})}}
    \frac{\frac{\partial}{\partial A_s}\widetilde{P}_{\mathrm{Inv}(c_{\ell(\mathfrak{l})}),\mathrm{p}(\nu(\mathfrak{l}))}^{{a_{\ell(\mathfrak{l})}}}}{     \zeta^{({a_{\ell(\mathfrak{l})}}+{\mathrm{Inv}(c_{\ell(\mathfrak{l})})})\mathrm{p}(\nu(\mathfrak{l}))}}\\
&=\begin{cases}
\frac{(-1)^{a_{\ell(\mathfrak{l})}}}{(2s+1)^2}\frac{C_{s+1}\cdot K_{s}}{L\cdot L^{s}}
    \frac{\widetilde{P}_{s+1,\mathrm{p}(\nu(\mathfrak{l}))}^{{a_{\ell(\mathfrak{l})}}-1}}{     \zeta^{({a_{\ell(\mathfrak{l})}}+s)\mathrm{p}(\nu(\mathfrak{l}))}}\quad & \text{if}\quad c_{\ell(\mathfrak{l})}=s+1,\\
0\quad & \text{if} \quad c_{\ell(\mathfrak{l})}\neq s+1.
\end{cases}
\end{aligned}
\end{equation*}
This means the only possible non-zero derivative of a leg contribution $\mathrm{Cont}_{\Gamma}^{\mathrm{A}}(\mathfrak{l})$ with respect to $A_s$ is when the insertion of the leg $\mathfrak{l}$ is $\phi_{s+1}$. In this case, we get
\begin{equation*}
\frac{C_{s+1}}{(2s+1)L}\frac{\partial \mathrm{Cont}_{\Gamma}^{\mathrm{A}}(\mathfrak{l}) }{\partial A_s}=\frac{(-1)^{a_{\ell(\mathfrak{l})}}}{(2s+1)^2}\frac{K_{s+1}}{ L^{s+1}}
    \frac{\widetilde{P}_{s+1,\mathrm{p}(\nu(\mathfrak{l}))}^{{a_{\ell(\mathfrak{l})}}-1}}{     \zeta^{({a_{\ell(\mathfrak{l})}}+s)\mathrm{p}(\nu(\mathfrak{l}))}}.
\end{equation*}
Shifting $a_{\ell(\mathfrak{l})}$ by $1$, we see that right hand side becomes
\begin{equation*}
\frac{(-1)^{a_{\ell(\mathfrak{l})}+1}}{(2s+1)^2}\frac{K_{s+1}}{ L^{s+1}}
    \frac{\widetilde{P}_{s+1,\mathrm{p}(\nu(\mathfrak{l}))}^{{a_{\ell(\mathfrak{l})}}}}{     \zeta^{({a_{\ell(\mathfrak{l})}}+s+1)\mathrm{p}(\nu(\mathfrak{l}))}}
\end{equation*}
which seems to correspond to a leg contribution with $\phi_s$ scaled by $$-\frac{1}{2s+1}.$$ However, this also shifts the power of $\psi_{\bullet}^{{a_{\ell(\mathfrak{l})}}}$ by $1$, the  $\psi$-class appearing in the vertex contribution $\mathrm{Cont}_{\Gamma}^{\mathrm{A}}(\mathfrak{v})$ of the vertex $\mathfrak{v}$ that is $\mathfrak{l}$ attached to. So, by Proposition \ref{prop:insertions_cont_odd}, we see that one $\mathrm{ev}^{\star}_i(\phi_{s+1})$  in the potential is traded with $\pi^{\star}(\psi_i)\cdot \mathrm{ev}^{\star}_i(\phi_{s})$ after a rescaling of $$-\frac{1}{2s+1}.$$ When combined with the product rule for derivatives, we see that this happens $r_{s+1}$ many times by the term $(\star \star)$. Hence, the term $(\star \star)$ gives us the contribution
\begin{equation*}
-\frac{r_{s+1}}{2s+1}\mathcal{F}_{g,m}^{\CnZn}[r_0,\ldots,r_s,r_{s+1}-1,r_{s+2},\ldots, r_{n-1};1]
\end{equation*}
after summing over all possible decorated stable graphs.
\end{proof}

We remark that if we let $n=3$, then Theorem \ref{thm:HAE_Odd_Insertions} agrees with the holomorphic anomaly equations for $K\mathbb{P}^2$ with insertions given in \cite[Theorem 29]{lho-p} when combined with the crepant resolution correspondence for $[\mathbb{C}^3/\mathbb{Z}_3]$ proved in \cite{lho-p2}.

When $n$ is even, the holomorphic anomaly equations with insertions for $\mathcal{F}_{g, m}^{\CnZn}[r_0,\ldots,r_{n-1};\delta_1,\delta_2]$ are given by the following result.

\begin{thm} Let $n=2s\geq{4}$. Then, for the stable range $2g-2+m>0$, we have
\begin{equation*}
\begin{aligned}
&\frac{C_{s+1}}{2sL}\frac{\partial}{\partial A_{s-1}}\mathcal{F}_{g, m}^{\CnZn}[r_0,\ldots,r_{n-1}]\\
=&\mathcal{F}_{g-1,m+2}^{\CnZn}[r_0,\ldots,r_{s-2},r_{s-1}+1,r_s+1,r_{s+1},\ldots, r_{n-1}]\\
&+\sum_{\substack{g_1+g_2=g \\m_1+m_2=m \\ \bar{r}_i+\tilde{r}_i=r_i}}\mathcal{F}_{g_1,m_1+1}^{\CnZn}[\bar{r}_0,\ldots,\bar{r}_{s-2},\bar{r}_{s-1}+1,\bar{r}_s,\ldots, \bar{r}_{n-1}]\mathcal{F}_{g_2,m_2+1}^{\CnZn}[\tilde{r}_0,\ldots,\tilde{r}_{s-1},\tilde{r}_s+1,\tilde{r}_{s+1},\ldots, \tilde{r}_{n-1}]\\
&-\frac{r_{s+1}}{2s}\mathcal{F}_{g,m}^{\CnZn}[r_0,\ldots,r_s,r_{s+1}-1,r_{s+2},\ldots, r_{n-1};1,0]\\
&-\frac{r_{s}}{2s}\mathcal{F}_{g,m}^{\CnZn}[r_0,\ldots,r_{s-1},r_s-1,r_{s+1},\ldots, r_{n-1};0,1]
\end{aligned}
\end{equation*}
\end{thm}

\begin{proof}
The proof is similar to the odd case. The explanation for the terms
\begin{equation*}
\begin{aligned}
&\mathcal{F}_{g-1,m+2}^{\CnZn}[r_0,\ldots,r_{s-2},r_{s-1}+1,r_s+1,r_{s+1},\ldots, r_{n-1}]\\
&+\sum_{\substack{g_1+g_2=g \\m_1+m_2=m \\ \bar{r}_i+\tilde{r}_i=r_i}}\mathcal{F}_{g_1,m_1+1}^{\CnZn}[\bar{r}_0,\ldots,\bar{r}_{s-2},\bar{r}_{s-1}+1,\bar{r}_s,\ldots, \bar{r}_{n-1}]\mathcal{F}_{g_2,m_2+1}^{\CnZn}[\tilde{r}_0,\ldots,\tilde{r}_{s-1},\tilde{r}_s+1,\tilde{r}_{s+1},\ldots, \tilde{r}_{n-1}]
\end{aligned}
\end{equation*}
follows by the strategy of the proof of Theorem \ref{thm:HAE_n_even} and the first half of the proof Theorem \ref{thm:HAE_Odd_Insertions}. We will briefly explain the other two terms in the right-hand side of the equation.

Lemma \ref{lem:evenderivativeflatness} implies that the derivative of a leg contribution $\mathrm{Cont}_{\Gamma}^{\mathrm{A}}(\mathfrak{l})$ with respect to $A_{s-1}$ is zero if the insertion of the leg $\mathfrak{l}$ is not one of the insertions $\phi_{s}$ and $\phi_{s+1}$.

If the insertion of the leg $\mathfrak{l}$ is $\phi_{s+1}$, then by Lemma \ref{lem:evenderivativeflatness} we have

\begin{equation*}
\frac{C_{s+1}}{2sL}\frac{\partial \mathrm{Cont}_{\Gamma}^{\mathrm{A}}(\mathfrak{l}) }{\partial A_{s-1}}=\frac{(-1)^{a_{\ell(\mathfrak{l})}}}{(2s)^2}\frac{C_{s+1} K_{s-1}}{ L^{s}}
    \frac{\widetilde{P}_{s,\mathrm{p}(\nu(\mathfrak{l}))}^{{a_{\ell(\mathfrak{l})}}-1}}{     \zeta^{({a_{\ell(\mathfrak{l})}}+s-1)\mathrm{p}(\nu(\mathfrak{l}))}}.
\end{equation*}
We have $C_{s+1}=C_s$ by Lemma \ref{propertiesofCfunctions} and also we know $C_sK_{s-1}=K_s$. Hence, we see that
\begin{equation*}
\frac{C_{s+1}}{2sL}\frac{\partial \mathrm{Cont}_{\Gamma}^{\mathrm{A}}(\mathfrak{l}) }{\partial A_{s-1}}=\frac{(-1)^{a_{\ell(\mathfrak{l})}}}{(2s)^2}\frac{K_{s}}{ L^{s}}
    \frac{\widetilde{P}_{s,\mathrm{p}(\nu(\mathfrak{l}))}^{{a_{\ell(\mathfrak{l})}}-1}}{     \zeta^{({a_{\ell(\mathfrak{l})}}+s-1)\mathrm{p}(\nu(\mathfrak{l}))}}.
\end{equation*}
As in the proof of Theorem \ref{thm:HAE_Odd_Insertions}, shifting $a_{\ell(\mathfrak{l})}$ by $1$ we see that right hand side corresponds to a leg contribution with insertion $\phi_s$ scaled by $$-\frac{1}{2s}$$ together with a shift of the power of the class $\psi_{\bullet}^{a_{\ell(\mathfrak{l})}}$ 
by $1$, the $\psi$-class corresponding to the vertex contribution $\mathrm{Cont}_{\Gamma}^{\mathrm{A}}(\mathfrak{v})$ of the vertex $\mathfrak{v}$ that $\mathfrak{l}$ is attached to.  So,  by Proposition \ref{prop:insertions_cont_even}, we see that  one $\mathrm{ev}^{\star}_i(\phi_{s+1})$  in the potential is traded with $\pi^{\star}(\psi_i)\cdot \mathrm{ev}^{\star}_i(\phi_{s})$ together with a rescaling $$-\frac{1}{2s}.$$ Because of the product rule for derivatives , this happens $r_{s+1}$ many times.

If the insertion of the leg $\mathfrak{l}$ is $\phi_{s}$, then by Lemma \ref{lem:evenderivativeflatness} we have

\begin{equation*}
\frac{C_{s+1}}{2sL}\frac{\partial \mathrm{Cont}_{\Gamma}^{\mathrm{A}}(\mathfrak{l}) }{\partial A_{s-1}}=\frac{(-1)^{a_{\ell(\mathfrak{l})}}}{(2s)^2}\frac{C_{s+1} K_{s}}{ L^{s+1}}
    \frac{\widetilde{P}_{s+1,\mathrm{p}(\nu(\mathfrak{l}))}^{{a_{\ell(\mathfrak{l})}}-1}}{     \zeta^{({a_{\ell(\mathfrak{l})}}+s)\mathrm{p}(\nu(\mathfrak{l}))}}.
\end{equation*}
Again, by a similar shifting argument, this implies that one $\mathrm{ev}^{\star}_i(\phi_{s})$  in the potential is traded with $\pi^{\star}(\psi_i)\cdot \mathrm{ev}^{\star}_i(\phi_{s-1})$ together with a rescaling $$-\frac{1}{2s},$$ and this happens $r_{s}$ many times. Hence, we complete the proof.
\end{proof}

\appendix
\section{Stirling numbers}\label{appendix:stirling}
In this section, we provide a brief account of Stirling numbers and their properties used in the paper.  A detailed treatment of Stirling numbers can be found in \cite{gq}. Convenient online references for Stirling numbers include \cite{dlmf-s}.

The \textit{Stirling number of first kind} $s_{m,k}$ is defined to be the coefficient of $x^k$ of the falling factorial:
\begin{equation}\label{eq:defnitionStirlingFirst}
(x)_m=x(x-1)\cdots(x-m+1)=\sum_{k=0}^{m}s_{m,k}x^k.
\end{equation}
The special case $s_{0,0}$ is set to be 1, and certain Stirling numbers of the first kind we use to do some explicit computations in the paper are:
\begin{equation*}
s_{m,0}=0\quad\text{for } m\geq{1},\\
\end{equation*}
\begin{equation*}
s_{m,m-1}=-\binom{m}{2},\quad s_{m,m-2}=\frac{3m-1}{4}\binom{m}{3},\quad s_{m,m-3}=-\binom{m}{2}\binom{m}{4}.
\end{equation*}
The \textit{Stirling number of the second kind} $S_{m,k}$  is the number of ways to partition a set of $m$ objects into $k$ non-empty subsets. Stirling numbers of the second kind satisfy the following basic recurrence:
\begin{equation}\label{eq:Stirling_2nd_Recursion}
S_{m,k}=kS_{m-1,k}+S_{m-1,k-1}\quad\text{with}\quad S_{m,0}=\delta_{m,0}.
\end{equation}
A well-known formula for Stirling numbers of the second kind called Euler's formula is
\begin{equation}\label{eq:EulerForStirlingSecond}
 S_{m, k}=\frac{1}{k !} \sum_{i=0}^k(-1)^{k-i}\binom{k}{i} i^m.
\end{equation}
If $k$ is not in the range $0\leq{k}\leq{m}$, Stirling numbers $s_{m,k}$ and $S_{m,k}$ are defined to be 0. The following relation holds
\begin{equation}
\sum_{j\geq{0}}{s_{m,j}S_{j,k}}=\sum_{j\geq{0}}{S_{m,j}s_{j,k}}=\delta_{m,k}
\end{equation}
i.e., Stirling numbers are inverses of each other when they are seen as triangular matrices.
%%%%%%%%%%%%%%%%%%%%%%%%%%%%%%%%%%%%%%%%%%

\section{A note on \textit{I}-functions}\label{appendix:I-function}
In this Appendix, we carry out a detailed analysis for the $I$-function of $\left[\mathbb{C}^{n} / \mathbb{Z}_{n}\right]$ by following the methodology of \cite{zz}. The techniques we use are more or less identical to \cite{zz}. However, the $I$-function for $\CnZn$ is different than the main hypergeometric series used in \cite{zz}. For this reason and for the convenience of the reader, we provide a detailed treatment of these techniques for the $I$-function of $\CnZn$.

\subsection{Series associated to \textit{I}-functions}\label{appendix:I-function-Part1}

\begin{lem}[See \cite{zz}]\label{zzlemma}
Suppose $y_0,...,y_m$, $f$, $g$, $a$ are functions of $t$ (with $f$ not identically $0$) satisfying
\begin{align*}
\begin{split}
y_mf^{(m)}+y_{m-1}f^{(m-1)}+...+y_0f=&0,\\
y_mg^{(m)}+y_{m-1}g^{(m-1)}+...+y_0g=&a,
\end{split}
\end{align*}
where $f^{(k)}=\frac{d^kf}{dt^k}$. Then the function $h=(g/f)'$ satisfies
\begin{equation*}
    \tilde{y}_{m-1}h^{(m-1)}+\tilde{y}_{m-2}h^{(m-2)}+...+\tilde{y}_{0}h=a,
\end{equation*}
where $\tilde{y}_{s}(t)=\sum_{r=s+1}^{m} \binom{r}{s+1} y_r(t)f^{(r-1-s)}(t).$
\end{lem}

We obtain the following result that is similar to \cite[Corollary 1]{zz}.
\begin{cor}\label{zzcorollary}
Suppose $F(x,z)$ satisfies
\begin{equation}\label{corrolary1eqn}
\sum_{r=0}^{m}W_r(x)D^rF(x,z)=A(x,z)
\end{equation}
with $A(\infty,x)\equiv 0$, then we have
\begin{equation*}
    \left(\sum_{s=0}^{m-1}\tilde{W}_s(x)D^s\right)\mathds{M}F(x,z)=zA(x,z),
\end{equation*}
where $\tilde{W}_{s}(x)=\sum_{r=s+1}^{m} \binom{r}{s+1} W_r(x)D^{(r-1-s)}F(x,\infty)$ and $\mathds{M}$ is defined in (\ref{eqn:BF_operator}).
\end{cor}

\begin{proof}
Apply Lemma \ref{zzlemma} with $f(t)=F(e^t,\infty)$, $g(t)=F(e^t,z)$, $a(t)=A(e^t,z)$, and $y_r(t)=W_r(e^t)$ for $0\leq r \leq m$.
\end{proof}

\begin{lem}\label{propertiesofCfunctions}
The series $C_k$ in $x$ satisfy the following properties:
\begin{enumerate}
    \item $C_{k+n}=C_k$ for all $k\geq 1$,
    \item $\prod_{k=1}^nC_k={L}^n$,
    \item $C_{k}=C_{n+1-k}$ for all $1\leq k \leq n$.
\end{enumerate}
\end{lem}

\begin{proof}
Since we basically set all $\phi_i$'s in $I(x,z)$ to $1$ to obtain the series $E(x,z)$, it also satisfies the Picard-Fuchs equation:
\begin{equation*}
x^{-n}\left(\left(1-(-1)^n\left(\frac{x}{n}\right)^n\right)D^nE(x,z)+\sum_{k=1}^{n-1}s_{n,k}D^kE(x,z)\right)=z^{-n}E(x,z)
\end{equation*}
which is of the form (\ref{corrolary1eqn}) with $F(x,z)=E(x,z)$, $A(x,z)=z^{-n}E(x,z)$, $m=n$, and
\begin{align*}
W_n(x)=&x^{-n}\left(1-(-1)^n\left(\frac{x}{n}\right)^n\right)=L^{-n},\\
W_r(x)=&x^{-n}s_{n,r}\quad\text{for}\quad(1\leq r\leq n-1),\\
W_0(x)=&0.
\end{align*}
Applying Corollary \ref{zzcorollary} repeatedly, we obtain
\begin{equation}\label{inductivesteplemma}
    \sum_{s=0}^{n-1-p}W_{s,p}(x)D^sE_{p+1}(x,z)=z^{-n+p+1}E(x,z)\quad (0\leq p \leq n-1),
\end{equation}
where $E_i(x,z)$ is defined in (\ref{def:EiCi}) and $W_{s,p}(x)$ is given inductively by
\begin{equation*}
W_{s,p}(x)=\sum_{r=s+1}^{n-p}\binom{r}{s+1} W_{r,p-1}(x)D^{r-1-s}C_p(x).
\end{equation*}
By induction on $p$, we see that the first coefficient in (\ref{inductivesteplemma}) is given by
\begin{equation*}
    W_{n-1-p,p}(x)=W_n(x)\prod_{i=1}^pC_i(x)\label{coefficient1}.
\end{equation*}
Then, equation (\ref{inductivesteplemma}) for $p=n-1$ gives
\begin{equation}\label{equationgiving12oflemma}
\left(W_n(x)\prod_{i=1}^{n-1}C_i(x)\right)E_n(x,z)=E(x,z).
\end{equation}
Letting $z=\infty$ in (\ref{equationgiving12oflemma}), using $E(x,\infty)=1$, $W_n(x)=L^{-n}$ and $E_n(x,\infty)=C_n(x)$ we obtain
\begin{equation*}
L^{-n}\prod_{i=1}^{n}C_i(x)=1.
\end{equation*}
which proves part (\ref{Cfunctions2}) of Lemma \ref{propertiesofCfunctions}. 

Substituting part (\ref{Cfunctions2}) of Lemma \ref{propertiesofCfunctions} into equation (\ref{equationgiving12oflemma}) gives
\begin{equation*}
    \frac{E_n(x,z)}{C_n(x)}=E(x,z).
\end{equation*}
Applying $\mathds{M}^{k-1}zD$ to both sides of this equality for $k\geq 1$ results in
\begin{equation*}
    E_{n+k}(x,z)=E_{k}(x,z),
\end{equation*}
which proves part (\ref{Cfunctions1}) of Lemma \ref{propertiesofCfunctions}.

Now, equation (\ref{twopointfunctioncalc}) yields that for any $0\leq i,j\leq n-1$ we have
\begin{equation*}
\mathfrak{L}_{i}...\mathfrak{L}_{0}I_{i+1}=\mathfrak{L}_{j}...\mathfrak{L}_{0}I_{j+1}\quad\text{if}\quad i+j=n-1.
\end{equation*}
Applying the operator $D$ to both sides gives part (\ref{Cfunctions4}) of Lemma \ref{propertiesofCfunctions}.
\end{proof}

For any $l\geq 0$, we define the following series in $x$
\begin{equation*}
K_l=\prod_{i=0}^lC_i.
\end{equation*}

\begin{cor}\label{Kfunctions} The series $K_l$ satisfy
\begin{enumerate}
    \item $K_{n+l}=L^nK_l$ for all $l\geq{0}$, in particular $K_n=L^n$, 
    \item $K_lK_{n-l}=L^n$ and $K_lK_{\mathrm{Inv}(l)}=L^{l+\mathrm{Inv}(l)}$ for all $0\leq l \leq n-1$.
\end{enumerate}
\end{cor}

\begin{proof}
For the first part, the special case $K_n=L^n$ is just part (\ref{Cfunctions2}) of Lemma \ref{propertiesofCfunctions}. Then, general case $K_{n+l}=L^nK_l$ follows from part (\ref{Cfunctions1}) of Lemma \ref{propertiesofCfunctions}.

For the second part, we calculate
\begin{align*}
K_lK_{n-l}
=&\left(\prod_{i=1}^{l}C_i\right)\left(\prod_{j=1}^{n-l}C_j\right)\\
=&\left(\prod_{i=1}^{l}C_i\right)\left(\prod_{j=1}^{n-l}C_{n+1-j}\right)\quad \text{ by part (\ref{Cfunctions4}) of Lemma \ref{propertiesofCfunctions}}\\
=&\left(\prod_{i=1}^{l}C_i\right)\left(\prod_{i=l+1}^{n}C_{i}\right)=K_n=L^n,
\end{align*}
the rest follows from the fact that $K_0K_n=1\cdot L^n$ and $\mathrm{Inv}(l)=n-l$ for $1\leq{l}\leq{n-1}$.
\end{proof}

\subsection{Asymptotic solutions of Picard-Fuchs equations}\label{appendix:I-function-Part2}
In this part of the Appendix \ref{appendix:I-function}, our aim is to prove a recursive equation and polynomiality result similar to the some parts of \cite[Theorem 4]{zz}. In some sense, Lemma \ref{lem:AymptoticPFSolution} is analogous to \cite[Theorem 4(ii)]{zz} and Corollary \ref{cor:Phi_jk_identity} is analogous to \cite[Theorem 4(i)]{zz}. These results will be applied to the $P_{i,j}^k$'s appearing in the main body of the paper, see Corollary \ref{cor:P0jk_is_in_CL} and Corollary \ref{cor:LPijkidentity}. The proofs follow the path of \cite[Section 2.4]{zz}.

For $0\leq{j}\leq{n-1}$, define
\begin{equation*}
D_{L_j}=D+\frac{L_j}{z} \quad\text{and}\quad \mu_j=\int_0^x\frac{L_j(u)}{u}du
\end{equation*}
where $L_j=L\zeta^j$. 

\begin{lem}\label{lem:AymptoticPFSolution}
Assume for $0\leq{j}\leq{n-1}$ a function of the form $e^{\frac{\mu_{j}}{z}}\Phi_{j}(z)$ satisfies the Picard-Fuchs equation (\ref{eq:PF3}), i.e.
\begin{equation*}
L^{-n}\left(D^{n}\left(e^{\frac{\mu_{j}}{z}}\Phi_{j}(z)\right)+\frac{D L}{L} \sum_{r=1}^{n-1} s_{n,r} D^{r}\left(e^{\frac{\mu_{j}}{z}}\Phi_{j}(z)\right)\right)=z^{-n}e^{\frac{\mu_{j}}{z}}\Phi_{j}(z)
\end{equation*}
where
\begin{equation*}
\Phi_{j}(z)=\sum_{k=0}^{\infty} \Phi_{j,k} z^{k}\quad\text{with } \Phi_{j,k}\in\mathbb{C}[\![x]\!]\text{ and }\Phi_{j,k}=0\text{ if } k<0.
\end{equation*}
Then, we have $\Phi_{j,k}\in\mathbb{C}[L_j]=\mathbb{C}[L]$.
\end{lem}
For the rest of this section, our aim is to prove Lemma \ref{lem:AymptoticPFSolution}, hence we adopt its assumptions. For any function $F(x,z)$, observe the following
\begin{equation}\label{eqn:mujcommutation}
\begin{split}
D\left(e^{\frac{\mu_{j}}{z}} F(x,z)\right)
&=e^{\frac{\mu_{j}}{z}} D F(x,z)+\frac{D \mu_{j}}{z} e^{\frac{\mu_{j}}{z}} F(x,z) \\
&=e^{\frac{\mu_{j}}{z}} D F(x,z)+\frac{L_{j}}{z} e^{\frac{\mu_{j}}{z}} F(x,z)\\
&=e^{\frac{\mu_{j}}{z}} D_{L_j} F(x,z).
\end{split}
\end{equation}
Then, the equation in Lemma \ref{lem:AymptoticPFSolution} reads as
\begin{equation}\label{eq:LPz0}
\mathds{L}_j\Phi_{j}(z)=0\quad\text{where}\quad
\mathds{L}_j=-\left(\frac{L_j}{z}\right)^n+D_{L_{j}}^{n}+\frac{D L_j}{L_j} \sum_{r=1}^{n-1} s_{n, r} D_{L_{j}}^{r}
\end{equation}
using equation \eqref{eqn:mujcommutation}, $L^n=(L_j)^n$, and $\frac{DL}{L}=\frac{DL_j}{L_j}$.

For $1\leq{k}\leq{n}$, define\footnote{Note that the definition of $\mathds{L}_{j,k}$ does not depend on $j$ since $L^n=(L_j)^n$, and $\frac{DL}{L}=\frac{DL_j}{L_j}$.}
\begin{equation}\label{eq:LLjk}
\mathds{L}_{j,k}=\sum_{i=0}^{k}\left(\binom{n}{i} H_{n-i, k-i}+\frac{D L_j}{L_j} \sum_{r=1}^{k-i}\binom{n-r}{i} s_{n, n-r} H_{n-i-r, k-i-r}\right) D^{i}
\end{equation}
where $H_{m,l}$ are defined\footnote{$H_{m,l}$ is set to be $0$ outside the range $m\geq{1}$, $0\leq{l}\leq{m}$.} by the following recursion for $m\geq{1}$ and $0\leq{l}\leq{m}$:
\begin{equation}\label{eq:HmlRecursion}
H_{0,l}=\delta_{0,l},\quad\text{and}\quad H_{m,l}=H_{m-1,l}+n\left(1+(-1)^n\frac{X}{n^n}\right)\left(X\frac{d}{d X}+\frac{m-l}{n}\right)H_{m-1,l-1},
\end{equation}
here $X$ is a formal variable.

Equations (\ref{eq:LLjk}) and (\ref{eq:HmlRecursion}) are adaptations of \cite[Equation (7)]{zz} and \cite[Equation (9)]{zz}. By a detailed analysis of $H_{m,l}$'s, the operators $\mathds{L}_j$ will be decomposed into a summation involving the operators $\mathds{L}_{j,k}$ in Lemma \ref{lem:Lj_Intermsof_Ljk}.

By induction, we see that
\begin{equation*}
\begin{split}
H_{m,l}=&0\quad \text{ if } l>m,\\
H_{m,0}=&1\quad \text{ for } m\geq{0},\\
H_{m,1}=&\binom{m}{2}\left(1+(-1)^n\frac{X}{n^n}\right)\quad \text{ for } m\geq{1},\\
H_{m,2}=&3\binom{m}{4}\left(1+(-1)^n\frac{X}{n^n}\right)^2\\
        &+\binom{m}{3}\left((n+1)\left(1+(-1)^n\frac{X}{n^n}\right)^2-n\left(1+(-1)^n\frac{X}{n^n}\right)\right)\quad \text{ for } m\geq{2}.
\end{split}
\end{equation*}

Now we specialize $X$ and define $Y$ as follows:
\begin{equation}\label{eqn:XandY}
\begin{split}
X=&L_j^n=L^n\\
Y=&\frac{DL_j}{L_j}=\frac{DL}{L}=1+(-1)^n\frac{L^n}{n^n}=1+(-1)^n\frac{X}{n^n}.
\end{split}
\end{equation}
Then, we see that
\begin{equation*}
\begin{split}
DY=&(-1)^n\frac{L^{n-1}}{n^{n-1}}DL=(-1)^n\frac{L^{n}}{n^{n-1}}\frac{DL}{L}=(-1)^n\frac{1}{n^{n-1}}XY,\\
DX=&nL^{n-1}DL=nL^n\frac{DL}{L}=nXY=nX\left(1+(-1)^n\frac{X}{n^n}\right).
\end{split}
\end{equation*}

Also, using Stirling numbers of the first kind which are explicitly given in Appendix \ref{appendix:stirling}, we compute the first two terms of $\mathds{L}_{j,k}$:
\begin{equation*}
\begin{split}
\mathds{L}_{j,1}=&nD,\\
\mathds{L}_{j,2}=&\binom{n+1}{4}(Y^2-Y)-\binom{n}{2}YD+\binom{n}{2}D^2.
\end{split}
\end{equation*}
The following lemma will be useful in the decomposion of $\mathds{L}_j$ with respect to the operators $\mathds{L}_{j,k}$.
\begin{lem}\label{lem:DLj_Hml_Equation}
For all $k\geq{0}$, we have
\begin{equation*} 
D_{L_{j}}^{k}=\sum_{m=0}^{k}\sum_{l=0}^{m} \binom{k}{m} H_{m,l}\left(\frac{L_{j}}{z}\right)^{m-l} D^{k-m}.
\end{equation*}
\end{lem}

\begin{proof}
First, we prove by induction that
\begin{equation*}
D_{L_{j}}^{k}=\sum_{m=0}^{k}\binom{k}{m}D_{L_{j}}^{m}(1)D^{k-m}.
\end{equation*}

We need to note that
\begin{equation}\label{eq:DLjLeibniz}
\begin{split}
D_{L_{j}}\left(FD^k\right)
&=\left(D+\frac{L_j}{z}\right)\left(FD^k \right)\\
&=(DF)D^k+FD^{k+1}+\frac{L_j}{z}FD^k\\
&=\left(D_{L_{j}}F\right)D^k+FD^{k+1}.
\end{split}
\end{equation}

For the base step $k=1$, we have
\begin{equation*}
D_{L_j}=D+\frac{L_j}{z}=D+D_{L_j}(1)=\sum_{m=0}^{1}\binom{1}{m}D_{L_{j}}^{m}(1)D^{1-m}.
\end{equation*}
For the inductive step, we have
\begin{align*}
D_{L_{j}}^{k}
&=D_{L_{j}}D_{L_{j}}^{k-1}\\
&=\sum_{m=0}^{k-1}\binom{k-1}{m}D_{L_{j}}\left(D_{L_{j}}^{m}(1)D^{k-1-m}\right)\\
&=\sum_{m=0}^{k-1}\binom{k-1}{m}\left(D_{L_{j}}^{m+1}(1)D^{k-1-m}+D_{L_{j}}^{m}(1)D^{k-m}\right)\quad\text{ by equation \eqref{eq:DLjLeibniz}}\\
&=\sum_{m=0}^{k-1}\binom{k-1}{m}D_{L_{j}}^{m+1}(1)D^{k-1-m}+\sum_{m=0}^{k-1}\binom{k-1}{m}D_{L_{j}}^{m}(1)D^{k-m}\\
&=\sum_{m=1}^{k}\binom{k-1}{m-1}D_{L_{j}}^{m}(1)D^{k-m}+\sum_{m=0}^{k-1}\binom{k-1}{m}D_{L_{j}}^{m}(1)D^{k-m}\\
&=\underbrace{\binom{k-1}{k-1}}_{=\binom{k}{k}}D_{L_{j}}^{k}(1)D^{k-k}+\sum_{m=1}^{k-1}\underbrace{\left(\binom{k-1}{m-1}+\binom{k-1}{m}\right)}_{=\binom{k}{m}}D_{L_{j}}^{m}(1)D^{k-m}+\underbrace{\binom{k-1}{0}}_{=\binom{k}{0}}D_{L_{j}}^{0}D^{k-0}\\
&=\sum_{m=0}^{k}\binom{k}{m}D_{L_{j}}^{m}(1)D^{k-m}.
\end{align*}

Next, using another induction we prove
\begin{equation*}
D_{L_{j}}^{m}(1)=\sum_{l=0}^{m}  H_{m,l}\left(\frac{L_{j}}{z}\right)^{m-l}.
\end{equation*}
We begin with some observations:
\begin{equation}\label{eq:DLj1}
\begin{split}
DH_{m-1,l}(X)
&=\frac{d}{dX}H_{m-1,l}DX\\
&=n\left(1+(-1)^n\frac{X}{n^n}\right)X\frac{d}{dX}H_{m-1,l}
\end{split}
\end{equation}
and
\begin{equation}\label{eq:DLj2}
\begin{split}
D\left(\frac{L_j}{z}\right)^{m-1-l}
&=(m-1-l)\left(\frac{L_j}{z}\right)^{m-2-l}D\left(\frac{L_j}{z}\right)\\
&=(m-1-l)\left(\frac{L_j}{z}\right)^{m-2-l}\left(\frac{L_j}{z}\right)\left(1+(-1)^n\frac{X}{n^n}\right)\\
&=\left(1+(-1)^n\frac{X^n}{n^n}\right)(m-1-l)\left(\frac{L_j}{z}\right)^{m-1-l}
\end{split}
\end{equation}
and
\begin{equation}\label{eq:DLjLeibniz2}
D_{L_j}(FG)=(DF)G+F(DG)+\frac{L_j}{z}(FG).
\end{equation}

For the base step $m=0$, we have
\begin{equation*}
D^0_{L_j}(1)=1=H_{0,0}\left(\frac{L_j}{z}\right)^0.
\end{equation*}
For the inductive step, assume the statement holds for $m-1$, then
\begin{equation*}
\begin{split}
&\quad D_{L_j}^m\\
&=D_{L_j}D_{L_j}^{m-1}\\
&=D_{L_j}\sum_{l=0}^{m-1}H_{m-1,l}\left(\frac{L_j}{z}\right)^{m-1-l}\\
&=\sum_{l=0}^{m-1}\left((DH_{m-1,l})\left(\frac{L_j}{z}\right)^{m-1-l}+H_{m-1,l}D\left(\frac{L_j}{z}\right)^{m-1-l}+H_{m-1,l}\left(\frac{L_j}{z}\right)^{m-l}\right)\quad \text{ by equation \eqref{eq:DLjLeibniz2}}.
\end{split}
\end{equation*}
Then, by equations \eqref{eq:DLj1} and \eqref{eq:DLj2}, we see that 
\begin{equation*}
\begin{split}
&\quad D_{L_j}^m\\
&=\sum_{l=0}^{m-1}\left(\left(n(1+(-1)^n\frac{X}{n^n})\left(X\frac{d}{dX}+\frac{m-1-l}{n}\right)H_{m-1,l}\right)\left(\frac{L_j}{z}\right)^{m-1-l}+H_{m-1,l}\left(\frac{L_j}{z}\right)^{m-l}\right)\\
&=\sum_{l=1}^{m}\left(n(1+(-1)^n\frac{X}{n^n})\left(X\frac{d}{dX}+\frac{m-l}{n}\right)H_{m-1,l-1}\right)\left(\frac{L_j}{z}\right)^{m-l}+\sum_{l=0}^{m-1}H_{m-1,l}\left(\frac{L_j}{z}\right)^{m-l}\\
&=\sum_{l=0}^{m}\left(n(1+(-1)^n\frac{X}{n^n})\left(X\frac{d}{dX}+\frac{m-l}{n}\right)H_{m-1,l-1}+H_{m-1,l}\right)\left(\frac{L_j}{z}\right)^{m-l}\text{by $H_{m-1,-1}=0$, $H_{m-1,m}=0$}\\
&=\sum_{l=0}^{m}H_{m,l}\left(\frac{L_j}{z}\right)^{m-l}.
\end{split}
\end{equation*}
\end{proof}
Now, we are ready to prove the following lemma which describes the operator $\mathds{L}$ as a summmation of the operators $\mathds{L}_{j,k}$.
\begin{lem}\label{lem:Lj_Intermsof_Ljk}
For all $0\leq{j}\leq{n-1}$, we have
\begin{equation*}
\mathds{L}_j=\sum_{k=1}^n\left(\frac{L_j}{z}\right)^{n-k}\mathds{L}_{j,k}.
\end{equation*}
\end{lem}

\begin{proof}
By Lemma \ref{lem:DLj_Hml_Equation}, and the definition\footnote{Note that the sum in the definition of $\mathds{L}_j$ in equation (\ref{eq:LPz0}) can start at $r=0$ since $s_{n,0}=0$.} of $\mathds{L}_j$ we have
\begin{equation*}
\begin{split}
\mathds{L}_j
&=-\left(\frac{L_j}{z}\right)^{n}+\sum_{m=0}^{n}\sum_{l=0}^{m} H_{m, l}\binom{n}{m}\left(\frac{L_j}{z}\right)^{m-l} D^{n-m}+\frac{D L_j}{L_j} \underbrace{{\sum_{r=0}^{n-1} \sum_{m=0}^{r} }}_{=\sum_{m=0}^{m=n-1}\sum_{r=m}^{r=n-1}}\sum_{l=0}^{m} s_{n, r} \binom{r}{m} H_{m, l}\left(\frac{L_j}{z}\right)^{m-l} D^{r-m}\\
&=-\left(\frac{L_j}{z}\right)^{n}+\sum_{m=0}^{n}\sum_{l=0}^{m} H_{m, l}\binom{n}{m}\left(\frac{L_j}{z}\right)^{m-l} D^{n-m}+\frac{D L_j}{L_j} \sum_{m=0}^{n-1}\sum_{l=0}^{m}\sum_{r=m}^{n-1} s_{n, r} \binom{r}{m} H_{m, l}\left(\frac{L_j}{z}\right)^{m-l} D^{r-m}.
\end{split}
\end{equation*}
By separating $m=n$ case from the first double summation and by the change of indices via $m=n-i$ and $l=k-i$, we obtain
\begin{align*}
\mathds{L}_j=&-\left(\frac{L_j}{z}\right)^{n}+\sum_{l=0}^{n}\binom{n}{n} H_{n, l}\left(\frac{L_j}{z}\right)^{n-l}\\
&+\underbrace{\sum_{i=1}^{n} \sum_{k=i}^{n}}_{={\sum_{k=1}^{k=n} \sum_{i=1}^{i=k}}}\left(\frac{L_j}{z}\right)^{n-k}\left(\binom{n}{n-i} H_{n-i, k-i} D^{i}+\frac{D L_j}{L_j} \sum_{r=n-i}^{n-1} s_{n, r}\binom{r}{n-i} H_{n-i, k-i} D^{r-n+i}\right).
\end{align*}
Change the index $l$ to $k$ in the first summation and shift $r$ by $n-i$:
\begin{equation}\label{eq:L_jlastform}
\begin{split}
\mathds{L}_j=&-\left(\frac{L_j}{z}\right)^{n}+\sum_{k=0}^{n}\binom{n}{n} H_{n, l}\left(\frac{L_j}{z}\right)^{n-k}+\sum_{k=1}^{n} \sum_{i=1}^{k}\left(\frac{L_j}{z}\right)^{n-k}\underbrace{\binom{n}{n-i}}_{=\binom{n}{i}} H_{n-i, k-i} D^{i}\\
&+\frac{D L_j}{L_j} \sum_{k=1}^{n}\left(\frac{L_j}{z}\right)^{n-k} \overbrace{\sum_{i=1}^{k}\sum_{r=0}^{i-1} s_{n, r+n-i}\underbrace{\binom{r+n-i}{n-i}}_{=\binom{r+n-i}{r}} H_{n-i, k-i} D^{r}}^{(\star)}.
\end{split}
\end{equation}

For $(\star)$, we have
\begin{align*}
\underbrace{\sum_{i=1}^{k} \sum_{r=0}^{i-1}}_{={\sum_{r=0}^{r=k-1} \sum_{i=r+1}^{i=k}}} s_{n, r+n-i}\binom{r+n-i}{r} H_{n-i, k-i} D^{r}
&=\sum_{i=0}^{k-1} \sum_{r=i+1}^{k} s_{n, i+n-r}\binom{i+n-r}{i} H_{n-r, k-r} D^{i}\quad (i\leftrightarrow r)\\
&=\sum_{i=0}^{k-1} \sum_{r=1}^{k-i} s_{n, n-r} \binom{n-r}{i}{H}_{n-i-r, k-i-r} D^{i}\quad\text{shift $r$ by $i$}\\
&=\sum_{i=0}^{k} \sum_{r=1}^{k-i} s_{n, n-r} \binom{n-r}{i}{H}_{n-i-r, k-i-r} D^{i}\quad\text{since }H_{n-k-r,-r}=0.
\end{align*}
Note also that
\begin{equation*}
-\left(\frac{L_j}{z}\right)^{n}+\sum_{k=0}^{n}\binom{n}{n} H_{n, k}\left(\frac{L_j}{z}\right)^{n-k}=\sum_{k=1}^{n}\binom{n}{n} H_{n, k}\left(\frac{L_j}{z}\right)^{n-k}
\end{equation*}
Hence, \eqref{eq:L_jlastform} reads as
\begin{equation*}
\begin{split}
\mathds{L}_j=&\underbrace{\sum_{k=1}^{n}\binom{n}{n} H_{n, k}\left(\frac{L_j}{z}\right)^{n-k}+\sum_{k=1}^{n} \sum_{i=1}^{k}\left(\frac{L_j}{z}\right)^{n-k}\binom{n}{i} H_{n-i, k-i} D^{i}}_{\text{These can be combined.}}\\
&+\frac{D L_j}{L_j} \sum_{k=1}^{n}\left(\frac{L_j}{z}\right)^{n-k} \sum_{i=0}^{k} \sum_{r=1}^{k-i} s_{n, n-r} \binom{n-r}{i}{H}_{n-i-r, k-i-r} D^{i}.
\end{split}
\end{equation*}
So, we have
\begin{align*}
\mathds{L}_j=&\sum_{k=1}^{n}\left(\frac{L_j}{z}\right)^{n-k} \sum_{i=0}^{k}\binom{n}{i} H_{n-i, k-i} D^{i}+\frac{D L_j}{L_j} \sum_{k=1}^{n}\left(\frac{L_j}{z}\right)^{n-k} \sum_{i=0}^{k} \sum_{r=1}^{k-i} s_{n, n-r} \binom{n-r}{i}{H}_{n-i-r, k-i-r} D^{i}\\
=&\sum_{k=1}^{n}\left(\frac{L_j}{z}\right)^{n-k}\underbrace{\sum_{i=0}^{k}\left(\binom{n}{i} H_{n-i, k-i}+\frac{D L_j}{L_j} \sum_{r=1}^{k-i} s_{n, n-r} \binom{n-r}{i}{H}_{n-i-r, k-i-r} \right)D^{i}}_{=\mathds{L}_{j,k}}.
\end{align*}
\end{proof}

Using Lemma \ref{lem:Lj_Intermsof_Ljk} above, we obtain the following result which can be used to determine $\Phi_{j,k}$'s recursively for a given set initial conditions $\Phi_{j,k}\vert_{x=0}$.

\begin{cor}\label{cor:Phi_jk_identity}
For $0\leq{j}\leq{n-1}$ and $k\geq 0$, we have
\begin{equation*}
\mathds{L}_{j,1}(\Phi_{j,k})+\frac{1}{L_j}\mathds{L}_{j,2}(\Phi_{j,k-1})+\frac{1}{L_j^2}\mathds{L}_{j,3}(\Phi_{j,k-2})+\cdots+\frac{1}{L_j^{n-1}}\mathds{L}_{j,n}(\Phi_{j,k+1-n})=0.
\end{equation*}
\end{cor}

\begin{proof}
We calculate 
\begin{align*}
0=\mathds{L}_j\Phi_{j}(z)
=&\sum_{l=1}^n\left(\frac{L_j}{z}\right)^{n-l}\mathds{L}_{j,l}\Phi_{j}(z)\quad\text{by Lemma \ref{lem:Lj_Intermsof_Ljk}}\\
=&\sum_{l=1}^n\sum_{k=0}^{\infty}\left(\frac{L_j}{z}\right)^{n-l}\mathds{L}_{j,l}\Phi_{j,k}z^k\\
=&\sum_{l=1}^n\sum_{k=0}^{\infty}{L_j}^{n-l}\mathds{L}_{j,l}\Phi_{j,k}z^{k+l-n}\\
=&\sum_{l=1}^n\sum_{k=l-1}^{\infty}{L_j}^{n-l}\mathds{L}_{j,l}\Phi_{j,k+1-l}z^{k+1-n}\quad\text{shift }k\text{ by }l-1\\
=&\sum_{l=1}^n\sum_{k=0}^{\infty}{L_j}^{n-l}\mathds{L}_{j,l}\Phi_{j,k+1-l}z^{k+1-n}\quad\quad\text{since }\Phi_{j,k+1-l}=0\text{ for }k<l-1\\
=&\sum_{k=0}^{\infty}\sum_{l=1}^n{L_j}^{n-l}\mathds{L}_{j,l}\Phi_{j,k+1-l}z^{k+1-n}.
\end{align*}
Then equation (\ref{eq:LPz0}) reads as
\begin{equation*}
\sum_{l=1}^n{L_j}^{n-l}\mathds{L}_{j,l}\Phi_{j,k+1-l}=0
\end{equation*}
for any $k\geq{0}$. By dividing out $L_j^{n-1}$, we finish the proof.
\end{proof}

Let $$\mathcal{I}\subset\mathbb{C}[L]$$ be the ideal generated by $XY$.

The following result shows that the operators $\mathds{L}_{j,k}$ have a simple form modulo the ideal $\mathcal{I}$. This will be useful in the polynomiality result Lemma \ref{lem:AymptoticPFSolution} we want to prove.
\begin{lem}\label{lem:Ljkequivalence}
We have
\begin{equation*}
\mathds{L}_{j,k}\equiv\binom{n}{k}(D)(D-Y)\cdots(D-(k-1)Y)\mod\mathcal{I}.
\end{equation*}
\end{lem}

\begin{proof}
Note that $Y$ and $D$ commute modulo $\mathcal{I}$:
\begin{equation*}
D(YF)=(DY)F+Y(DF)=\frac{(-1)^n}{n^{n-1}}XYF+Y(DF).
\end{equation*}
Also observe that for any $r\geq{1}$ we have,
\begin{equation}\label{eq:Yr_equiv_YmodI}
\begin{split}
Y^r
&\equiv (Y)^{r-1}Y\mod \mathcal{I}\\
&\equiv (1+(-1)^n\frac{X}{n^n})^{r-1}Y \mod \mathcal{I}\\
&\equiv Y \mod \mathcal{I}.
\end{split}   
\end{equation}

We first show by induction that
\begin{equation}\label{eq:Hml_equiv_hmYl}
H_{m,l}\equiv h_{m,l}Y^l\mod \mathcal{I}
\end{equation}
where $h_{m,l}=S_{m,m-l}$ is the Stirling number of the second kind. The only thing we need to prove is that if $H_{m,l}\equiv h_{m,l}Y^l\mod \mathcal{I}$, then the numbers $h_{m,l}$ are given by the recursion
\begin{equation*}
h_{0,l}=\delta_{0,l},\text{ and }h_{m,l}=h_{m-1,l}+(m-l)h_{m-1,l-1}\text{ for all }m\geq{1}.
\end{equation*}
This will imply $h_{m,l}=S_{m,m-l}$ by recursion (\ref{eq:Stirling_2nd_Recursion}). The base case $l=0$ is given by
\begin{equation*}
H_{m,0}=1\text{, and } h_{m,0}=S_{m,m}=1.
\end{equation*}
The recursion above is equivalent to equation \eqref{eq:HmlRecursion}:
\begin{equation*}
\begin{split}
H_{m,l}
&\equiv H_{m-1,l}+nY\left(X\frac{d}{d X}+\frac{m-l}{n}\right)H_{m-1,l-1} \mod\mathcal{I}\\
&\equiv H_{m-1,l}+(m-l)YH_{m-1,l-1}\mod\mathcal{I}
\end{split}
\end{equation*}
which is the same as
\begin{equation*}
\begin{split}
h_{m,l}Y^l
&\equiv H_{m-1,l}Y^l+(m-l)YH_{m-1,l-1}Y^{l-1}\mod\mathcal{I}\\
&\equiv H_{m-1,l}Y^l+(m-l)H_{m-1,l-1}Y^l\mod\mathcal{I}.
\end{split}
\end{equation*}
By induction, cancelling out $Y^l$ on both sides we get what we want; that is, $H_{m,l}\equiv h_{m,l}Y^l\mod \mathcal{I}$.

By the definitions of $Y$ and $\mathds{L}_{j,k}$, and using equation (\ref{eq:Yr_equiv_YmodI}) in the second line, we obtain
\begin{equation}\label{eq:LLjkmodI}
\begin{split}
\mathds{L}_{j,k}
&\equiv\sum_{i=0}^{k}\left(\binom{n}{i} H_{n-i, k-i}+Y\sum_{r=1}^{k-i}\binom{n-r}{i} s_{n, n-r} H_{n-i-r, k-i-r}\right) D^{i}\mod\mathcal{I}\\
&\equiv\sum_{i=0}^{k}\left(\binom{n}{i} H_{n-i, k-i}+\sum_{r=1}^{k-i}Y^r\binom{n-r}{i} s_{n, n-r} H_{n-i-r, k-i-r}\right) D^{i}\mod\mathcal{I}\\
&\equiv\sum_{i=0}^{k}\left(\sum_{r=0}^{k-i}Y^r\binom{n-r}{i} s_{n, n-r} H_{n-i-r, k-i-r}\right) D^{i}\mod\mathcal{I}.
\end{split}
\end{equation}
Then, by equation (\ref{eq:Hml_equiv_hmYl}) we have
\begin{equation}\label{eq:LLjkmodI2}
\begin{split}
\mathds{L}_{j,k}
&\equiv\sum_{i=0}^{k}\left(\sum_{r=0}^{k-i}Y^r\binom{n-r}{i} s_{n, n-r} h_{n-i-r, k-i-r}Y^{k-i-r}\right) D^{i}\mod\mathcal{I}\\
&\equiv\sum_{i=0}^{k}\left(\sum_{r=0}^{k-i}\binom{n-r}{i} s_{n, n-r} h_{n-i-r, k-i-r}\right) Y^{k-i}D^{i}\mod\mathcal{I}\\
&\equiv\sum_{i=0}^{k}\left(\sum_{r=0}^{k-i}\binom{n-r}{i} s_{n, n-r} S_{n-i-r, n-k}\right) Y^{k-i}D^{i}\mod\mathcal{I}\quad\text{by }h_{m,l}=S_{m,m-l}.
\end{split}
\end{equation}

\iffalse
Next, we calculate 
\begin{align*}
&\sum_{i=0}^{k}\left(\sum_{r=0}^{k-i}\binom{n-r}{i}  s_{n, n-r}S_{n-i-r, n-k}\right) t^{i}\\
&=\sum_{i=0}^{n}\left(\sum_{r=0}^{n-i}\binom{n-r}{i}  s_{n, n-r} S_{n-i-r, n-k}\right) t^{i}\quad\text{since } S_{n-i-r, n-k}=0\text{ if }i+r>k\\
&=\sum_{i=0}^{n}\left(\sum_{r=0}^{n-i}\binom{n-r}{i}  s_{n, n-r}\left[\frac{1}{(n-k) !} \sum_{l=0}^{n-k}(-1)^{n-k-l}\binom{n-k}{l} l^{n-r-i}\right]\right) t^{i}\quad\text{by Euler's formula (\ref{eq:EulerForStirlingSecond})}\\
&=\frac{1}{(n-k) !} \sum_{l=0}^{n-k}(-1)^{n-k-l}\binom{n-k}{l} \underbrace{\sum_{i=0}^{n} \sum_{r=0}^{n-i}}_{=\sum_{r=0}^{n} \sum_{i=0}^{n-r}}s_{n, n-r}\binom{n-r}{i} l^{n-r-i} t^{i} \\
&=\frac{1}{(n-k) !} \sum_{l=0}^{n-k}(-1)^{n-k-l}\binom{n-k}{l} \sum_{r=0}^{n}s_{n, n-r}(l+t)^{n-r}\quad\text{by binomial formula} \\
&=\frac{n !}{(n-k) !} \sum_{l=0}^{n-k}(-1)^{n-k-l}\binom{n-k}{l}\binom{l+t}{n}\quad\text{by equation (\ref{eq:defnitionStirlingFirst}) }\\
&=\frac{n !}{(n-k) !}\binom{t}{k}\\
&=\binom{n}{k}t(t-1)\cdots(t-(k-1)).
\end{align*}
The second-to-last equality is obtained by expanding $(1+u)^tu^{n-k}=(1+u)^t((1+u)-1)^{n-k}$ via binomial formula, matching coefficients of $u^n$ on both sides and using Chu-Vandermonde identity:
\begin{equation*}
\binom{s+t}{m}=\sum_{k=0}^m\binom{s}{k}\binom{t}{m-k}.
\end{equation*}
\fi
Since $(-1)^ks_{n,n-k}$ are the elementary symmetric polynomials evaluated at $0,1,\ldots,n-1$, then by the combinatorial identity\footnote{The difference is that they have elementary symmetric polynomials evaluated at $1,2,\ldots, n$. The notation used in \cite{zz} for elementary symmetric polynomials evaluated at $1,2,\ldots, n$ is $S_k(n)$, and for the Stirling numbers of second kind is $\mathfrak{S}_n^{(k)}$.} in the proof of \cite[Lemma 4]{zz} we have
\begin{equation*}
\sum_{i=0}^{k}\left(\sum_{r=0}^{k-i}\binom{n-r}{i}  s_{n, n-r}S_{n-i-r, n-k}\right) t^{i}=\binom{n}{k}t(t-1)\cdots(t-(k-1)).
\end{equation*}
Hence, we obtain
\begin{equation*}
\sum_{i=0}^{k}\left(\sum_{r=0}^{k-i}\binom{n-r}{i}  s_{n, n-r}S_{n-i-r, n-k}\right) y^{k-i}t^{i}
=\binom{n}{k}t(t-y)\cdots(t-(k-1)y).
\end{equation*}
This together with equation (\ref{eq:LLjkmodI2}) completes the proof of the lemma when it is combined with the commutation of $Y$ and $D$ modulo $\mathcal{I}$.
\end{proof}

Now, we are ready to prove Lemma \ref{lem:AymptoticPFSolution}. Since $D$ and $Y$ commutes modulo $\mathcal{I}$, inductively we show that
\begin{equation*}
 (D)(D-Y)\cdots(D-(k-1)Y)L_j^r\mod\mathcal{I}\equiv
\begin{cases}
0 &\text{if } 0\leq{r}\leq{k-1},\\
r(r-1)\cdots{(r-(k-1))}{L_j^rY^k} &\text{if } {r}\geq{k}.
\end{cases}
\end{equation*}
Then, Lemma \ref{lem:Ljkequivalence} implies that
\begin{equation*}
\mathds{L}_{j,k}\left(L_j^r\right)\in
\begin{cases}
\mathcal{I} &\text{if } 0\leq{r}\leq{k-1},\\
{L_j^rY^k}+\mathcal{I} &\text{if } {r}\geq{k}.
\end{cases}
\end{equation*}
From this, we conclude that $\mathds{L}_{j,k}(\mathbb{C}[L_j])\subseteq L_j^kY\mathbb{C}[L_j]$ for any $1\leq{k}\leq{n}$ since $\mathcal{I}$ is generated by $XY$ and $X=L_j^n$. Moreover, for the case $k=1$, we have the equality $\mathds{L}_{j,1}(\mathbb{C}[L_j])=L_jY\mathbb{C}[L_j]$. This is because we have $\mathds{L}_{j,1}(L_{j}^r)=nDL_j^r=nrL_j^rY$ for any $r\geq{1}$ and $\mathds{L}_{j,1}(a)=nDa=0$ for any $a\in\mathbb{C}$. 

It is clear that the statement is true if $k=0$. Now, we prove the statement inductively. By Corollary \ref{cor:LPijkidentity}, we have the following
\begin{equation*}
\mathds{L}_{j,1}(\Phi_{j,k})=-\sum_{l=2}^n{L_j}^{1-l}\mathds{L}_{j,l}\left(\Phi_{j,k+1-l}\right)\in{L_jY}\mathbb{C}\left[L_j\right].
\end{equation*}
The right hand side belongs to ${L_jY}\mathbb{C}\left[L_j\right]$ by inductive hypothesis since $\mathds{L}_{j,l}(\mathbb{C}[L_j])\subseteq L_j^lY\mathbb{C}[L_j]$. This shows $\Phi_{j,k}\in\mathbb{C}[L_j]$ and completes the proof of Lemma \ref{lem:AymptoticPFSolution} since $\mathds{L}_{j,1}(\mathbb{C}[L_j])= L_jY\mathbb{C}[L_j]$.

\end{document}